\documentclass[a4paper]{article}
\usepackage[margin=30mm]{geometry}
\usepackage{amsmath}
\usepackage{amssymb}
\usepackage{amsthm}
\usepackage{braket}
\usepackage{titlesec}
\usepackage{titlefoot}
\usepackage{comment}
\usepackage{bm}
\titleformat*{\section}{\normalsize\bfseries}
\titleformat*{\subsection}{\normalsize\bfseries}
\allowdisplaybreaks
\def\R{\mathbb{R}}
\def\e{{\varepsilon}}        

\def\p{\partial}

\newtheorem{thm}{Theorem}[section]
\newtheorem{lem}[thm]{Lemma}
\newtheorem{cor}[thm]{Corollary}
\newtheorem{prop}[thm]{Proposition}
\newtheorem{rem}[thm]{Remark}

\makeatletter
\@addtoreset{equation}{section}
 
  \makeatother
 
\begin{document}
%%%%%%%%%%%%%%%%%%%%%%%%%%%%%%%%%%%%%%%%%%%%%%%%%%%%

%%%%%%%%%%%%%%%%%%%%%%%%%%%%%%%%%%%%%%%%%%%%%%%%%%%%
\title{
\vspace{-1cm}
\large{\bf Optimal Decay Estimate and Asymptotic Profile for Solutions to the Generalized Zakharov--Kuznetsov--Burgers Equation in 2D}}
\author{Ikki Fukuda and Hiroyuki Hirayama 
}
\date{}
\maketitle
%%%%%%%%%%%%%%%%%%%%%%%%%%%%%%%%%%%%%%%%%%%%%%%%%%%%

%%%%%%%%%%%%%%%%%%%%%%%%%%%%%%%%%%%%%%%%%%%%%%%%%%%%
\footnote[0]{2020 Mathematics Subject Classification: 35B40, 35Q53.}
%%%%%%%%%%%%%%%%%%%%%%%%%%%%%%%%%%%%%%%%%%%%%%%%%%%%

\vspace{-0.75cm}
%%%%%%%%%%%%%%%%%%%%%%%%%%%%%%%%%%%%%%%%%%%%%%%%%%%%
\begin{abstract}
We consider the Cauchy problem for the generalized Zakharov--Kuznetsov--Burgers equation in 2D. 
This is one of the nonlinear dispersive-dissipative equations, which has a spatial anisotropic dissipative term $-\mu u_{xx}$. 
In this paper, we prove that the solution to this problem decays at the rate of $t^{-\frac{3}{4}}$ in the $L^{\infty}$-sense, provided that the initial data $u_{0}(x, y)$ satisfies $u_{0}\in L^{1}(\R^{2})$ and some appropriate regularity assumptions. 
Moreover, we investigate the more detailed large time behavior and obtain a lower bound of the $L^{\infty}$-norm of the solution. As a result, we prove that the given decay rate $t^{-\frac{3}{4}}$ of the solution to be optimal. Furthermore, combining the techniques used for the parabolic equations and for the Schr$\ddot{\mathrm{o}}$dinger equation, we derive the explicit asymptotic profile for the solution. 
\end{abstract}
%%%%%%%%%%%%%%%%%%%%%%%%%%%%%%%%%%%%%%%%%%%%%%%%%%%%

\medskip
%%%%%%%%%%%%%%%%%%%%%%%%%%%%%%%%%%%%%%%%%%%%%%%%%%%%
\noindent
{\bf Keywords:} 
Zakharov--Kuznetsov--Burgers equation; optimal decay estimate; asymptotic profile. 
%%%%%%%%%%%%%%%%%%%%%%%%%%%%%%%%%%%%%%%%%%%%%%%%%%%%

%%%%%%%%%%%%%%%%%%%%%%%%%%%%%%%%%%%%%%%%%%%%%%%%%%%%%%%%%%%%%%%%%%%%%%%%%%%%%%%%%%%%%%%%%%%%%%%%%%%%%%%%
\section{Introduction}  
%%%%%%%%%%%%%%%%%%%%%%%%%%%%%%%%%%%%%%%%%%%%%%%%%%%%%%%%%%%%%%%%%%%%%%%%%%%%%%%%%%%%%%%%%%%%%%%%%%%%%%%%

We consider the Cauchy problem for the generalized Zakharov--Kuznetsov--Burgers equation in 2D:
\begin{align}\label{ZKB}
\begin{split}
& u_{t} + u_{xxx} + u_{yyx} - \mu u_{xx} + \beta u^{p}u_{x} = 0, \ \ (x, y) \in \R^{2}, \ t>0,\\
& u(x, y, 0) = u_{0}(x, y), \ \ (x, y) \in \R^{2}, 
\end{split}
\end{align}
where $u = u(x, y, t)$ is a real-valued unknown function, $u_{0}(x, y)$ is a given initial data, $p\ge1$ is an integer, $\beta \in \R$ and $\mu>0$. 
The subscripts $x$, $y$ and $t$ denote the partial derivatives with respect to $x$, $y$ and $t$, respectively. 
When $p=1$, this equation \eqref{ZKB} is called the Zakharov--Kuznetsov--Burgers equation, which appears in the dust-ion-acoustic-waves in dusty-plasmas (cf.~\cite{MS08}). Also, \eqref{ZKB} can be considered as a two dimensional model of the generalized Korteweg-de Vries--Burgers (KdV--Burgers) equation (see \eqref{KdVB} below). 
The main purpose of our study is to analyze the large time asymptotic behavior of the solutions to the Cauchy problem \eqref{ZKB}. 
In particular, we would like to derive the optimal decay estimate and the asymptotic profile for solutions. 

First of all, let us recall some known results related to this problem. 
Taking $\mu=0$ and replacing $t>0$ with $t\in \R$ in \eqref{ZKB}, we derive the following generalized Zakharov--Kuznetsov equation: 
\begin{align}\label{gZK}
\begin{split}
& u_{t} + u_{xxx} + u_{yyx} + \beta u^{p}u_{x} = 0, \ \ (x, y) \in \R^{2}, \ t\in \R,\\
& u(x, y, 0) = u_{0}(x, y), \ \ (x, y) \in \R^{2}. 
\end{split}
\end{align}
This equation has the conserved mass and energy:
\[
M(u(t)):=\int_{\R^2}u(t)^2dxdy,\ \ 
E(u(t)):=\frac{1}{2}\int_{\R^2}|\nabla u(t)|^2dxdy
-\frac{\beta}{(p+1)(p+2)}\int_{\R^{2}} u(t)^{p+2}dxdy. 
\]
The local well-posedness of 2D generalized Zakharov--Kuznetsov equation 
in the Sobolev space $H^s(\R^2)$ is well studied by many authors 
(cf.~\cite{FLP12, GH14, Ki21, Ki22, LP11, LP09, MP15, RV12}). 
In particular, the local well-posedness of \eqref{gZK} in $H^s(\R^2)$ 
is established for the following $s\in \R$ (see \cite{Ki21, Ki22, RV12}):
\[
\begin{cases}
s>-\frac{1}{4}& {\rm if}\ \ p=1,\\
s\ge \frac{1}{4}& {\rm if}\ \ p=2,\\
s>\frac{5}{12}& {\rm if}\ \ p=3,\\
s>s_c:=1-\frac{2}{p}& {\rm if}\ \ p\ge 4, 
\end{cases}
\]
where $s_c$ is the scaling critical Sobolev exponent. The global well-posedness is obtained by using the conserved quantities $M(u(t))$ and $E(u(t))$ as the above. 
In \cite{Ki21}, Kinoshita proved that 
\eqref{gZK} with $p=1$ is globally well-posed in $L^2(\R^2)$. 
In \cite{BL03}, Biagioni--Linares proved that 
(\ref{gZK}) with $p=2$ is globally well-posed in $H^1(\R^2)$. 
For the case $p\ge 3$, 
Linares--Pastor \cite{LP11} showed that 
\eqref{gZK} is globally well-posed for small initial data in $H^1(\R^2)$. 
The large time behavior of the solution to \eqref{gZK} related to the decay and scattering results are considered by Farah--Linares--Pastor \cite{FLP12}. In particular, they proved that the global solution to \eqref{gZK} with $p\ge 3$ 
for small initial data scatters in $H^1(\R^2)$. 
More precisely, let $q=2p+2$, $q'=\frac{2p+2}{2p+1}$ and $\theta=\frac{p}{p+1}$, if $u_0\in H^1(\R^2)\cap L^{q'}(\R^2)$ 
is small enough, then the global solution $u(x, y, t)$ to \eqref{gZK} with $p\ge 3$ 
satisfies the following estimate: 
\[
\left\|u(\cdot, \cdot, t)\right\|_{L^{q}}\le C(1+|t|)^{-\frac{2\theta}{3}}, \ \ t\in \R.
\]
Moreover, the following asymptotic formula has been established: 
\[
\lim_{t\rightarrow \pm \infty}
\left\|u(t)-e^{-t\left(\partial_x^3+\partial_x\partial_y^2\right)}f_{\pm}\right\|_{H^1}=0
\]
for some $f_{\pm}\in H^1(\R^2)$. 
Namely, the original solution $u(x, y, t)$ to \eqref{gZK} converges to the solution to the linearized problem of \eqref{gZK} 
as $t\rightarrow \pm \infty$. 

Compared with \eqref{gZK}, our target problem \eqref{ZKB}, which has the dissipative term $-\mu u_{xx}$, 
should be considered as a dispersive-dissipative equation, not as a purely dispersive equation. 
However, especially in the higher dimensions, there are very few results on the large time behavior of the solutions to dispersive-dissipative equations with spatial anisotropic dissipative term like that of \eqref{ZKB}. 
Actually, there are only two such papers \cite{FH23, M99}, up to the authors knowledge (we will explain them later). 
On the other hand, for the one-dimensional case, there are a lot of results for dispersive-dissipative equations. 
Before introducing the results in the higher dimensional case, let us recall about the one-dimensional problem. 
First, we shall introduce some related results on the following Cauchy problem for the generalized KdV--Burgers equation, which is a one-dimensional version for \eqref{ZKB}: 
\begin{align}\label{KdVB}
\begin{split}
& u_{t} + u_{xxx} -\mu u_{xx} + \beta u^{p}u_{x} = 0, \ \ x \in \R, \ t>0,\\
& u(x, 0) = u_{0}(x), \ \ x \in \R.
\end{split}
\end{align}
For the above problem \eqref{KdVB}, since the global well-posedness in appropriate Sobolev spaces can be easily proved by virtue of the dissipation effect, the main research topic is about the large time asymptotic behavior of the solution. This topic was first studied by Amick--Bona--Schonbek \cite{ABS89}. They derived the time decay estimates of the solution, in the case of $p=1$. In particular, they showed that if $u_{0}\in H^{2}(\R) \cap L^{1}(\R)$, then the solution satisfies the following estimate: 
\begin{equation}\label{KdVB-decay}
\left\|u(\cdot, t)\right\|_{L^{q}} \le Ct^{-\frac{1}{2}\left(1-\frac{1}{q}\right)}, \ \ t>0, \ 2\le q\le \infty. 
\end{equation}
We note that the decay rate of the solution given by \eqref{KdVB-decay} is the same as the one for the solution to the parabolic equations such as the linear heat equation: 
\[
u_{t}-\mu u_{xx}=0, \ \ x \in \R, \ t>0. 
\] 
It means that the dissipation effect is stronger than the dispersion effect. 
Moreover, in view of the asymptotic profile, it is known that the nonlinearity and the dissipation are balanced each other. 
Actually, Karch \cite{K99-2} studied \eqref{KdVB} with $p=1$ in more details and extended the results given in \cite{ABS89}. If $u_{0}\in H^{1}(\R)\cap L^{1}(\R)$, then the following asymptotic formula can be established: 
\begin{equation}\label{Karch}
\lim_{t\to \infty}t^{\frac{1}{2}\left(1-\frac{1}{q}\right)}\left\|u(\cdot, t)-\chi(\cdot, t)\right\|_{L^{q}}
=0, \ \ 1\le q\le \infty,
\end{equation}
where $\chi(x, t)$ is the self-similar solution to the following Burgers equation: 
\begin{equation*}
\chi_{t}-\mu \chi_{xx}+\beta \chi \chi_{x}=0, \ \ x\in \R, \ t>0. 
\end{equation*}
This self-similar solution $\chi(x, t)$ can be obtained explicitly as follows (cf.~\cite{C51, H50}): 
\begin{equation*}
\chi(x, t):=\frac{1}{\sqrt{t}} \chi_{*} \left(\frac{x}{\sqrt{t}}\right), \ \ x\in \R,\ t>0,
\end{equation*}
where
\begin{equation}\label{chi_star_def}
\chi_{*}(x):=\frac{\sqrt{\mu}}{\beta}\frac{\left(e^{\frac{\beta M}{2\mu}}-1\right)e^{-\frac{x^{2}}{4\mu}}}{\sqrt{\pi}+\left(e^{\frac{\beta M}{2\mu}}-1\right)\displaystyle \int_{x/\sqrt{4\mu}}^{\infty}e^{-y^{2}}dy}, 
\ \ M:=\int_{\R}u_{0}(x)dx. 
\end{equation}
Moreover, Hayashi--Naumkin \cite{HN06} improved the asymptotic rate given in \eqref{Karch} of the solution $u(x, t)$ to the above self-similar solution $\chi(x, t)$. More precisely, they showed that if we additionally assume $xu_{0}\in L^{1}(\R)$, then $u(x, t)$ goes to $\chi(x, t)$ at the rate of $t^{-\frac{1}{2}-\delta}$ in the $L^{\infty}$-sense, where $\delta\in (0, 1/2)$. Furthermore, Kaikina--Ruiz-Paredes \cite{KP05} succeeded to construct the second asymptotic profile of the solution, under the condition $xu_{0}\in L^{1}(\R)$. In view of the second asymptotic profile, they proved that the optimal asymptotic rate to $\chi(x, t)$ is $t^{-1}\log t$ in the $L^{\infty}$-sense, and also showed that the effect of the dispersion term $u_{xxx}$ appears from the second term of its asymptotics. For more related results concerning this problem, we can also refer to \cite{F19-1, F19-2, HKNS06}. 

Compared with the case of $p=1$, we can say that the nonlinearity is weak if $p\ge2$ in \eqref{KdVB}. Because if the solution $u(x, t)$ decays, then the nonlinear term $\beta u^{p}u_{x}$ decays fast enough. In this case, the main part of the solution $u(x, t)$ is governed by the heat kernel:
\[
G(x, t):=\frac{1}{\sqrt{4\pi \mu t}}\exp\left(-\frac{x^{2}}{4\mu t}\right), \ \ x\in \R, \ t>0. 
\]
Actually, if $p\ge2$ and $u_{0}\in H^{1}(\R)\cap L^{1}(\R)$ is small enough, then we are able to obtain the same decay estimate \eqref{KdVB-decay}. Moreover, the following asymptotic formula also has been established: 
\begin{equation}\label{Karch-heat}
\lim_{t\to \infty}t^{\frac{1}{2}\left(1-\frac{1}{q}\right)}\left\|u(\cdot, t)-MG(\cdot, t)\right\|_{L^{q}}
=0, \ \ 1\le q\le \infty,
\end{equation}
where $M$ is defined in \eqref{chi_star_def}. 
This asymptotic formula \eqref{Karch-heat} can be found in e.g.~\cite{K97-2, K99-1} by Karch. Especially in \cite{K99-1}, he studied \eqref{KdVB} in details when $p\ge2$ (in fact, for more general nonlinearities are treated), and constructed the second asymptotic profiles of the solution. Indeed, it was shown in \cite{K99-1} that if we additionally assume $xu_{0}\in L^{1}(\R)$, then the optimal asymptotic rates to $MG(x, t)$ are given by $t^{-1+\frac{1}{2q}}$ when $p>2$ and $t^{-1+\frac{1}{2q}}\log t$ when $p=2$, respectively, in the $L^{q}$-sense. Furthermore, according to his results, it can be seen that the effects of the nonlinear term $\beta u^{p}u_{x}$ and the dispersion term $u_{xxx}$ appear from the second asymptotic profile of the solution, in the case of $p>2$. However, we are able to find that the effect of 
the dispersion term does not appear in the second asymptotic profile when $p=2$. The results introduced in this paragraph also hold for more general dispersive-dissipative equations which has a similar structure to \eqref{KdVB}, such as the (generalized) BBM--Burgers equation: 
\[
u_{t}-u_{xxt}+ u_{xxx} -\mu u_{xx}+\beta u^{p}u_{x} =0, \ \ x\in \R, \ t>0. 
\]
For details, see e.g.~\cite{ABS89, HKN07, HKNS06, K97-1, K97-2, K99-1} and also references therein. 

The above asymptotic formulas \eqref{Karch} and \eqref{Karch-heat} imply that the dispersion term $u_{xxx}$ can be negligible in the sense of the first approximation, due to the dissipation effect. From this point of view, we can expect that the same phenomenon also to occur in the higher dimensional problem like our target equation \eqref{ZKB}. However, \eqref{ZKB} has a spatial anisotropy and the dissipative term $-\mu u_{xx}$ works only in the $x$-direction. 
Moreover, we also need to consider the effect of the dispersion term $u_{yyx}$ 
which works in both the $x$-direction and the $y$-direction. 
Therefore, the method used in the one-dimensional problems is difficult to apply to the analysis for \eqref{ZKB}. Of course, the method used for dispersive equations like \eqref{gZK} also cannot be applied to \eqref{ZKB}.
For this reason, in order to analyze the structure of the solution to the anisotropic dispersive-dissipative equations, 
we need to establish some different methods. 
As we mentioned in the above, the only known results related to the anisotropic dispersive-dissipative equations, are \cite{FH23, M99} for the following generalized Kadomtsev--Petviashvili--Burgers equation: 
\begin{align}\label{KPB}
\begin{split}
& u_{t} + u_{xxx} + \e \p_{x}^{-1}u_{yy} -\mu u_{xx} + \beta u^{p}u_{x} = 0, \ \ (x, y) \in \R^{2}, \ t>0,\\
& u(x, y, 0) = u_{0}(x, y), \ \ (x, y) \in \R^{2},  
\end{split}
\end{align}
where $\e \in \{-1, 1\}$, while the anti-derivative operator $\p_{x}^{-1}$ is defined by 
\begin{equation}\label{anti-derivative}
\p_{x}^{-1}f(x, y):=\mathcal{F}^{-1}\left[ (i\xi)^{-1}\hat{f}(\xi, \eta) \right](x, y).
\end{equation}
The definition of the Fourier transform $\hat{f}(\xi, \eta)=\mathcal{F}[f](\xi, \eta)$ will be given at the end of this section. 
In what follows, we would like to explain about the known results for \eqref{KPB} given in \cite{FH23, M99}. 
In order to introduce the results about \eqref{KPB}, for $s\ge0$, let us introduce the following function space: 
\begin{align}
X^{s}(\R^{2}) := \left\{ f \in H^{s}(\R^{2}); \ \|f\|_{X^{s}} := \|f\|_{H^{s}} + \left\| \partial_x^{-1}f \right\|_{H^{s}} <\infty \right\}.  \label{space-Xs}
\end{align}
First of all, we shall recall the fundamental results about the global well-posedness of \eqref{KPB}. Molinet~\cite{M99} constructed the solutions to \eqref{KPB}, provided the initial data $u_{0} \in X^{s}(\R^{2})$ for $s>2$. More precisely, he showed that if $p\ge1$ and $\left\|u_{0}\right\|_{H^{1}}+\left\|\p_{x}^{2}u_{0}\right\|_{L^{2}}+\left\|\p^{-1}_{x}\p_{y}u_{0}\right\|_{L^{2}}$ is sufficiently small, then \eqref{KPB} has a unique global mild solution $u(x, y, t)$. 
Moreover, concerning the well-posedness for \eqref{KPB}, we can also refer to another his paper \cite{M00}.
\begin{comment}
satisfying the following properties: 
\begin{align*}
&u \in C_{b}([0, \infty); X^{s}(\R^{2})) \cap L^{2}(0, \infty; H^{s}(\R^{2})), \\
&\p_{x}^{l}u \in C_{b}([1, \infty); H^{s}(\R^{2})) \cap L^{2}(1, \infty; H^{s}(\R^{2})), \ \ l \in \mathbb{N}. 
\end{align*}
\end{comment}

In addition to the global well-posedness, Molinet also discussed the large time behavior of the solution to \eqref{KPB} in \cite{M99}. More precisely, he gave some results about the decay estimates for the solution. In particular, under the additional assumptions $u_{0} \in X^{s}(\R^{2})$ for $s\ge3$, $\p_{x}^{-1}u_{0} \in L^{1}(\R^{2})$ and $\p_{x}^{-1}\p_{y}^{m}u_{0} \in L^{1}(\R^{2})$ for some integer $m\ge3$, then the solution to \eqref{KPB} satisfies 
\begin{equation}\label{u-decay-Molinet}
\left\|\p_{x}^{l}u(\cdot, \cdot, t)\right\|_{L^{\infty}} \le C
\begin{cases}
t^{-\frac{7}{4}-\frac{l}{2}},& p\ge2,  \\
t^{-\frac{3}{2}-\frac{l}{2}},& p=1,  
\end{cases}
\end{equation}
for all $t\ge1$ and non-negative integer $l$. 
Moreover, the same decay rate is also obtained 
for $\left\|\p_{x}^{l}\partial_y^ju(\cdot, \cdot, t)\right\|_{L^{\infty}}$
if $\p_{x}^{-1}\p_{y}^{j}u_{0} \in L^{1}(\R^{2})$ 
holds for $j \in \left\{1, \cdots, \min\{[s], m\}-3\right\}$. 
Here, we note that this decay estimate \eqref{u-decay-Molinet} is a different from not only \eqref{KdVB-decay} but also the one for the solution to two-dimensional parabolic equations and dispersive equations. Actually, this estimate \eqref{u-decay-Molinet} is constructed by the interactions of the effects of the dissipation $-\mu u_{xx}$ in the $x$-direction, the dispersion $\e \p_{x}^{-1}u_{yy}$ in both the $x$-direction and the $y$-direction, and also the anti-derivative $\p_{x}^{-1}$ on the initial data $u_{0}(x, y)$. 

In \cite{M99}, Molinet also mentioned about the optimality for the decay rate $t^{-\frac{7}{4}}$ given in \eqref{u-decay-Molinet} if $p\ge2$. 
To state such a result, we introduce the following linearized Cauchy problem for \eqref{KPB}: 
\begin{align}\label{LKPB}
\begin{split}
& \tilde{u}_{t} + \tilde{u}_{xxx} + \e \p_{x}^{-1}\tilde{u}_{yy} -\mu \tilde{u}_{xx} = 0, \ \ (x, y) \in \R^{2}, \ t>0, \\
& \tilde{u}(x, y, 0) = u_{0}(x, y), \ \ (x, y) \in \R^{2}. 
\end{split}
\end{align}
For this problem, he proved that if we assume $u_{0}\in X^{1}(\R^{2})$ and $(1+y^{4})\p_{x}^{-1}u_{0}\in L^{1}(\R^{2})$, then the following estimate holds: 
\begin{equation}\label{Molinet-linear}
\liminf_{t\to \infty}t^{\frac{7}{4}}\left\|\tilde{u}(\cdot, \cdot, t)\right\|_{L^{\infty}}
\ge C\left(\int_{\R^{2}}\p_{x}^{-1}u_{0}(x, y)dxdy\right)^{2}. 
\end{equation}
In \cite{M99}, the optimality for the decay rate $t^{-\frac{7}{4}}$ given in \eqref{u-decay-Molinet} is concluded by based on the above estimate \eqref{Molinet-linear}. 
However, \eqref{Molinet-linear} is the estimate for the solution $\tilde{u}(x, y, t)$ to the linearized problem \eqref{LKPB} not for the original solution $u(x, y, t)$ to the nonlinear problem \eqref{KPB}. 
In order to conclude that the decay rate $t^{-\frac{7}{4}}$ given in \eqref{u-decay-Molinet} is optimal, it is necessary to evaluate $t^{\frac{7}{4}}\left\|u(\cdot, \cdot, t)\right\|_{L^{\infty}}$ uniformly from below. 
Recently in \cite{FH23}, the authors succeed to solve this problem by deriving the lower bound of the $L^{\infty}$-norm for the original solution $u(x, y, t)$ to \eqref{KPB}. More precisely, under the same assumptions in the results given by Molinet~\cite{M99} and $p\ge2$, there exist positive constants $C_{\dag}>0$ and $T_{\dag}>0$ such that the solution $u(x, y, t)$ to \eqref{KPB} satisfies 
\begin{equation}\label{u-est-lower-FH}
\left\|\p_{x}^{l}\p_{y}^{j}u(\cdot, \cdot, t)\right\|_{L^{\infty}} \ge C_{\dag}\left|\mathcal{M}_{j}\right|t^{-\frac{7}{4}-\frac{l}{2}}, \ \ t \ge T_{\dag}, 
\end{equation}
for all non-negative integer $l$, where the constant $\mathcal{M}_{j}$ is defined by 
\begin{equation}\label{DEF-mathM-FH}
\mathcal{M}_{j} := \int_{\R^{2}}\p_{x}^{-1}\p_{y}^{j}u_{0}(x, y)dxdy-\frac{\beta}{p+1}\int_{0}^{\infty}\int_{\R^{2}}\p_{y}^{j}\left(u^{p+1}(x, y, t)\right)dxdydt. 
\end{equation}
By virtue of the above result, we can conclude that the $L^{\infty}$-decay estimate \eqref{u-decay-Molinet} is optimal with respect to the time decaying order, under the condition $\mathcal{M}_{j}\neq0$. In addition, they also mentioned a sufficient condition for $\mathcal{M}_{j}\neq0$ in \cite{FH23}. Moreover, their result does not need the weight condition $(1+y^{4})\p_{x}^{-1}u_{0}\in L^{1}(\R^{2})$ for the initial data $u_{0}(x, y)$ such as in \cite{M99}. 

Furthermore, they also showed that an approximation formula for the solution to \eqref{KPB}. To introduce such a result, let us consider the following problem: 
\begin{align}\label{w-asymp}
\begin{split}
& w_{t} + \e \p_{x}^{-1}w_{yy} -\mu w_{xx} = -\beta u^{p}u_{x}, \ \ (x, y) \in \R^{2}, \ t>0,\\
& w(x, y, 0) = u_{0}(x, y), \ \ (x, y) \in \R^{2}, 
\end{split}
\end{align}
where $u(x, y, t)$ is the original solution to \eqref{KPB}. Under some appropriate regularity assumptions on the initial data $u_{0}(x, y)$, the solution $w(x, y, t)$ to \eqref{w-asymp} can be expressed by 
\begin{equation}\label{w-sol}
w(t)=K(t)*u_{0} -\frac{\beta}{p+1}\int_{0}^{t} \p_{x}K(t-\tau)*\left(u^{p+1}(\tau)\right) d\tau, 
\end{equation}
where $K(x, y, t)$ is defined by 
\begin{align}
& K(x, y, t) := t^{ -\frac{5}{4} } K_{*} \left( xt^{-\frac{1}{2}}, yt^{-\frac{3}{4}} \right), \ \ (x, y)\in \R^{2}, \ \ t>0,  \label{DEF-K}\\
& K_{*}(x, y) := \frac{ 1 }{ 4\pi^{ \frac{3}{2} } \mu^{\frac{3}{4}} } \int_{0}^{\infty} r^{-\frac{1}{4} }e^{ -r } \cos \left( x \sqrt{\frac{r}{\mu}} +\frac{y^{2}}{4\e}\sqrt{\frac{r}{\mu}} - \frac{\pi}{4}\e \right) dr, \ \ (x, y)\in \R^{2}.  \label{DEF-K*}
\end{align}
Here, we note that $K(x, y, t)$ is the integral kernel of the solution to the linearized problem for \eqref{w-asymp}. Under the same assumptions to get \eqref{u-decay-Molinet} and \eqref{u-est-lower-FH}, we can see that the original solution $u(x, y, t)$ to \eqref{KPB} is well approximated by the above function $w(x, y, t)$. More precisely, the following approximation formula has been established: 
 \begin{equation}
\left\|\p_{x}^{l}\p_{y}^{j}\left(u(\cdot, \cdot, t)-w(\cdot, \cdot, t)\right)\right\|_{L^{\infty}}\le Ct^{-\frac{9}{4}-\frac{l}{2}}, \ \ t\ge2, \label{u-asymp}
\end{equation}
for all non-negative integer $l$. This approximation formula \eqref{u-asymp} tells us that the dispersion term $u_{xxx}$ in the direction in which the dissipative term $-\mu u_{xx}$ is working can be negligible as $t\to \infty$. 
Therefore, we can say that \eqref{KPB} has the similar property of the one-dimensional problem \eqref{KdVB}. 

Based on the above known results given in \cite{FH23, M99}, it is expected that the similar results as for \eqref{KPB} can be obtained for our target problem \eqref{ZKB}, since \eqref{ZKB} has a similar anisotropic structure to \eqref{KPB}. In this paper, applying and developing the method used in \cite{FH23, M99}, we analyze the large time asymptotic behavior of the solution to \eqref{ZKB}. More precisely, we shall derive the optimal decay estimate and the detailed asymptotic profile for the solution. 

\bigskip
\par\noindent
%%%%%%%%%%%%%%%%%%%%%%%%%%%%%%%%%%%%%%%%%%%%%%%%%%%%
\textbf{\bf{Main results.}} 
%%%%%%%%%%%%%%%%%%%%%%%%%%%%%%%%%%%%%%%%%%%%%%%%%%%%

\medskip
In what follows, we state our main results. In order to doing that, we shall define the Sobolev spaces $H^{s}(\R^{2})$ and the anisotropic Sobolev spaces $H^{s_{1}, s_{2}}(\R^{2})$ for $s, s_{1}, s_{2}\in \R$ as follows: 
\begin{align}
&H^{s}(\R^{2}):=\left\{f \in \mathcal{S}'(\R^{2}); \ \left\|f\right\|_{H^{s}}:=\left\|\left(1+\xi^{2}+\eta^{2}\right)^{\frac{s}{2}}\hat{f}(\xi, \eta)\right\|_{L^{2}_{\xi \eta}}<\infty \right\}, \label{DEF-Sobolev}\\
&H^{s_{1}, s_{2}}(\R^{2}):=\left\{f \in \mathcal{S}'(\R^{2}); \ \left\|f\right\|_{H^{s_{1}, s_{2}}}:=\left\|\left(1+\xi^{2}\right)^{\frac{s_{1}}{2}}\left(1+\eta^{2}\right)^{\frac{s_{2}}{2}}\hat{f}(\xi, \eta)\right\|_{L^{2}_{\xi \eta}}<\infty \right\}. \label{DEF-aniso-Sobolev}
\end{align}
First of all, we would like to introduce the results concerning the global well-posedness for \eqref{ZKB}: 
%%%%%%%%%%%%%%%%%%%%%%%%%%%%%%%%%%%%%%%%%%%%%%%%%%%%
\begin{thm}\label{thm.GWP-H21}
Let $p\ge1$ be an integer. Assume that $u_{0}\in H^{2, 1}(\R^{2})$ and $\left\|u_{0}\right\|_{H^{1}}+\left\|\p_{x}u_{0}\right\|_{H^{1}}$ is sufficiently small. 
Then, \eqref{ZKB} is globally well-posed in $H^{2, 1}(\R^{2})$. 
More precisely, there exists a unique global mild solution $u \in C([0, \infty); H^{2, 1}(\R^{2}))$ to \eqref{ZKB} satisfying 
\begin{equation}\label{apriori-H21}
\left\|u(\cdot, \cdot, t)\right\|_{H^{2, 1}}^{2}+\mu \int_{0}^{t} \left\|u_{x}(\cdot, \cdot , \tau)\right\|_{H^{2, 1}}^{2} d\tau \le C\left\|u_{0}\right\|_{H^{2, 1}}^{2}, \ \ t\ge0. 
\end{equation}
Moreover, the mapping $u_{0} \mapsto u$ is continuous from $H^{2, 1}(\R^{2})$ into $C([0, \infty); H^{2, 1}(\R^{2}))$. 
\end{thm}
%%%%%%%%%%%%%%%%%%%%%%%%%%%%%%%%%%%%%%%%%%%%%%%%%%%%
%
\begin{rem}
\begin{itemize}
\item[{\rm (i)}] The local well-posedness can be also 
    obtained in $H^s(\R^2)$ for $s>1$ 
    and $H^{s_1,s_2}(\R^2)$ for $s_1>\frac{1}{2}$, $s_2>\frac{1}{2}$ 
    without the smallness assumption of the initial data $u_{0}(x, y)$. (See, Theorem~\ref{thm.LWP} below.) 
    The key tool for the local results is     
    the property of Banach algebra for
    $H^s(\R^2)$ and $H^{s_1,s_2}(\R^2)$. 
    (For $H^{s_1,s_2}(\R^2)$, see Proposition~\ref{prop.Banach-al} below.)
\item[{\rm (ii)}] The well-posedness of \eqref{ZKB} is studied by Esfahani {\rm \cite{E11}}. 
He proved that \eqref{ZKB} is locally well-posed in $H^s(\R^2)$ for $s>2$. 
Our local result is a refinement of his result. 
The proof in {\rm \cite{E11}} is based on the parabolic regularization for the $y$-direction, 
and the continuity of the solution $u(x, y, t)$ with respect to $t$ is 
obtained only weak sense. 
But the parabolic regularization is not necessary 
because the derivative loss from the nonlinear term $\beta u^pu_x$
can be recovered by using the effect from the dissipative term $-\mu u_{xx}$. 
Our proof is based on this fact and the fixed point argument. 
\item[{\rm (iii)}] For the case $p=1$, the local well-posedness of \eqref{ZKB} in $H^s(\R^2)$ 
and the global well-posedness of \eqref{ZKB} in $H^{s,0}(\R^2)$ for $s>-\frac{1}{2}$ 
are established by the second author in {\rm \cite{H19}}. 
\end{itemize}
\end{rem}
%

%%%%%%%%%%%%%%%%%%%%%%%%%%%%%%%%%%%%%%%%%%%%%%%%%%%%
\begin{rem}
We believe that the similar results in Theorem~\ref{thm.GWP-H21} can be obtained for the equations of the following form, under some suitable assumptions: 
\[
u_{t} + \mathcal{P}u - \mu u_{xx} + \beta u^{p}u_{x} = 0, \ \ (x, y) \in \R^{2}, \ t>0, 
\]
where the dispersive operator $\mathcal{P}$ is defined via the Fourier transform as $\widehat{\mathcal{P}f}(\xi, \eta)=iP(\xi, \eta)\hat{f}(\xi, \eta)$ with a real-valued measurable function $P(\xi, \eta)$. 
Note that $\mathcal{P}$ is a skew-symmetric operator. 
\end{rem}
%%%%%%%%%%%%%%%%%%%%%%%%%%%%%%%%%%%%%%%%%%%%%%%%%%%%

From now on, we consider the large time behavior of the global solution constructed in the above Theorem~\ref{thm.GWP-H21}. 
Our next result is related to the $L^{\infty}$ and $L^{2}$-decay estimates of the solution. To state such a result, for simplicity, we set 
\begin{equation}\label{data}
E_{0}:=\left\|u_{0}\right\|_{H^{2, 1}}+\left\|u_{0}\right\|_{L^{1}}, \ \ E_{1}:=\left\|u_{0}\right\|_{H^{2, 1}}+\left\|u_{0}\right\|_{L^{1}}+\left\|\p_{y}u_{0}\right\|_{L^{1}}. 
\end{equation}
Then, we have the following result: 
%%%%%%%%%%%%%%%%%%%%%%%%%%%%%%%%%%%%%%%%%%%%%%%%%%%%
\begin{thm}\label{thm.main-decay}
Let $p\ge2$ be an integer. 
Assume that $u_{0}\in H^{2, 1}(\R^{2})\cap L^{1}(\R^{2})$ and $E_{0}$ is sufficiently small. 
Then, the solution $u(x, y, t)$ to \eqref{ZKB} satisfies the following estimate: 
\begin{equation}\label{u-sol-decay-Linf}
\left\|\p_{x}^{l}u(\cdot, \cdot, t)\right\|_{L^{\infty}} \le CE_{0}(1+t)^{-\frac{3}{4}-\frac{l}{2}}, \ \ t\ge0, \ l=0, 1.  
\end{equation}
Moreover, if we additionally assume $\p_{y}u_{0}\in L^{1}(\R^{2})$ and $E_{1}$ is sufficiently small, 
then, the solution $u(x, y, t)$ to \eqref{ZKB} satisfies the following estimate: 
\begin{equation}\label{u-sol-decay-L2}
\left\|\p_{x}^{l}u(\cdot, \cdot, t)\right\|_{L^{2}} \le CE_{1}(1+t)^{-\frac{1}{4}-\frac{l}{2}}, \ \ t\ge0, \ l=0, 1.  
\end{equation}
\end{thm}
%%%%%%%%%%%%%%%%%%%%%%%%%%%%%%%%%%%%%%%%%%%%%%%%%%%%

%%%%%%%%%%%%%%%%%%%%%%%%%%%%%%%%%%%%%%%%%%%%%%%%%%%%
\begin{rem}
By virtue of the above estimates \eqref{u-sol-decay-Linf} and \eqref{u-sol-decay-L2}, 
we can evaluate the $L^{q}$-norm of the solution $u(x, y, t)$ for any $2\le q\le \infty$. 
Actually, applying the interpolation inequality, we have 
\begin{align*}
\left\|\p_{x}^{l}u(\cdot, \cdot, t)\right\|_{L^{q}}
&\le \left\|\p_{x}^{l}u(\cdot, \cdot, t)\right\|_{L^{\infty}}^{1-\frac{2}{q}}\left\|\p_{x}^{l}u(\cdot, \cdot, t)\right\|_{L^{2}}^{\frac{2}{q}} \\
&\le C(1+t)^{-\frac{3}{4}+\frac{1}{q}-\frac{l}{2}}, \ \ t\ge0, \ l=0, 1, \ 2\le q\le \infty. 
\end{align*}
\end{rem}
%%%%%%%%%%%%%%%%%%%%%%%%%%%%%%%%%%%%%%%%%%%%%%%%%%%%

%%%%%%%%%%%%%%%%%%%%%%%%%%%%%%%%%%%%%%%%%%%%%%%%%%%%
\begin{rem}
We believe that the similar results in Theorem~\ref{thm.main-decay} can be obtained for the equations of the following form, under some suitable assumptions: 
\[
u_{t} + \mathcal{Q}u + u_{yyx} - \mu u_{xx} + \beta u^{p}u_{x} = 0, \ \ (x, y) \in \R^{2}, \ t>0, 
\]
where the dispersive operator $\mathcal{Q}$ is defined via the Fourier transform as $\widehat{\mathcal{Q}f}(\xi, \eta)=iQ(\xi)\hat{f}(\xi, \eta)$ with a real-valued measurable function $Q(\xi)$. 
Note that $\mathcal{Q}$ is a skew-symmetric operator which acts only on the variable $x$. 
\end{rem}
%%%%%%%%%%%%%%%%%%%%%%%%%%%%%%%%%%%%%%%%%%%%%%%%%%%%

Next, we would like to introduce an approximation formula for the solution to \eqref{ZKB}. 
In order to state that formula, let us consider the following Cauchy problem: 
\begin{align}\label{approx-CP}
\begin{split}
& v_{t} + v_{yyx} -\mu v_{xx}=0, \ \ (x, y) \in \R^{2}, \ t>0, \\
& v(x, y, 0) = u_{0}(x, y), \ \ (x, y) \in \R^{2}. 
\end{split}
\end{align}
We note that the solution $v(x, y, t)$ to \eqref{approx-CP} can be expressed by 
\begin{equation}\label{approx-sol}
v(x, y, t)=\left(V(t)*u_{0}\right)(x, y), \ \ (x, y) \in \R^{2}, \ t>0, 
\end{equation}
where the function $V(x, y, t)$ is given by
\begin{align}
&V(x, y, t) := t^{ -\frac{3}{4} } V_{*} \left( xt^{-\frac{1}{2}}, yt^{-\frac{1}{4}} \right), \ \ (x, y)\in \R^{2}, \ \ t>0,  \label{DEF-V}\\
&V_{*}(x, y) := \frac{ 1 }{ 4\pi^{ \frac{3}{2} }\mu^{\frac{1}{4}} } \int_{0}^{\infty} r^{-\frac{3}{4} }e^{ -r } \cos \left( x \sqrt{\frac{r}{\mu}} -\frac{y^{2}}{4}\sqrt{\frac{\mu}{r}} +\frac{\pi}{4}\right) dr, \ \ (x, y)\in \R^{2}.  \label{DEF-V*}
\end{align}
We will discuss this fact \eqref{approx-sol} in Section~3 below. Under the appropriate assumptions, if $p$ is large enough (weak nonlinearity), we are able to prove that the original solution $u(x, y, t)$ to \eqref{ZKB} is well approximated by the above function $v(x, y, t)$. Actually, the following result can be established: 
%%%%%%%%%%%%%%%%%%%%%%%%%%%%%%%%%%%%%%%%%%%%%%%%%%%%
\begin{thm}\label{thm.main-approximation}
Let $p\ge2$ be an integer and $l\in \{0, 1\}$ satisfying $p>(4+2l)/3$. Assume that $u_{0}\in H^{2, 1}(\R^{2})\cap L^{1}(\R^{2})$ and $E_{0}$ is sufficiently small. Then, the solution $u(x, y, t)$ to \eqref{ZKB} and the solution $v(x, y, t)$ to \eqref{approx-CP} satisfy the following approximation formula: 
\begin{align}\label{u-sol-approximation}
\lim_{t\to \infty}t^{\frac{3}{4}+\frac{l}{2}}\left\|\p_{x}^{l}\left(u(\cdot, \cdot, t)-v(\cdot, \cdot, t)\right)\right\|_{L^{\infty}}=0. 
\end{align}
\end{thm}
%%%%%%%%%%%%%%%%%%%%%%%%%%%%%%%%%%%%%%%%%%%%%%%%%%%%

%%%%%%%%%%%%%%%%%%%%%%%%%%%%%%%%%%%%%%%%%%%%%%%%%%%%
\begin{rem}
If $l=0$, the assumption for the nonlinear exponent $p>4/3$ actually means that $p\ge2$ because $p$ is an integer. 
Analogously, if $l=1$, this assumption $p>2$ implies $p\ge3$. 
However, this expression $p>(4+2l)/3$ actually appears in the proof of the theorem (see \eqref{Duhamel-est-final} below).
\end{rem}
%%%%%%%%%%%%%%%%%%%%%%%%%%%%%%%%%%%%%%%%%%%%%%%%%%%%

%%%%%%%%%%%%%%%%%%%%%%%%%%%%%%%%%%%%%%%%%%%%%%%%%%%%
\begin{rem}
We remark that the dispersion term $v_{xxx}$ is not included in \eqref{approx-CP}. Therefore, this approximation formula \eqref{u-sol-approximation} implies that the dispersion effect in the direction in which the dissipative term is working can be negligible as $t\to \infty$. 
This result is similar to the formula \eqref{u-asymp} given in {\rm \cite{FH23}} for the generalized KP--Burgers equation \eqref{KPB}. On the other hand, compared with \eqref{w-asymp}, we note that the equation \eqref{approx-CP} does not include the term $\beta u^{p}u_{x}$ related to the unknown function $u(x, y, t)$. Namely, we can say that the formula \eqref{u-sol-approximation} is more useful than \eqref{u-asymp}.  
\end{rem}
%%%%%%%%%%%%%%%%%%%%%%%%%%%%%%%%%%%%%%%%%%%%%%%%%%%%

Moreover, in addition to the above Theorem~\ref{thm.main-approximation}, we have succeeded to get a lower bound of the $L^{\infty}$-norm of the leading term of $v(x, y, t)$ (see Proposition~\ref{prop.mathV-lower} below). Thus, combining the estimate for that leading term of $v(x, y, t)$ (see \eqref{mathV-lower} below) and the above formula \eqref{u-sol-approximation}, we are able to obtain the following lower bound of the $L^{\infty}$-norm for the solution $u(x, y, t)$ to \eqref{ZKB}: 
%%%%%%%%%%%%%%%%%%%%%%%%%%%%%%%%%%%%%%%%%%%%%%%%%%%%
\begin{thm}\label{thm.main-u-est-lower}
Let $p\ge2$ be an integer and $l\in \{0, 1\}$ satisfying $p>(4+2l)/3$. Assume that $u_{0}\in H^{2, 1}(\R^{2})\cap L^{1}(\R^{2})$ and $E_{0}$ is sufficiently small. 
Then, there exists positive constant $c_{0}>0$ such that the solution $u(x, y, t)$ to \eqref{ZKB} satisfies the following estimate: 
\begin{equation}\label{u-est-lower}
\liminf_{t\to \infty}t^{\frac{3}{4}+\frac{l}{2}}\left\|\p_{x}^{l}u(\cdot, \cdot, t)\right\|_{L^{\infty}} \ge c_{0}\left|\int_{\R^{2}}u_{0}(x, y)dxdy\right|. 
\end{equation}
\end{thm}
%%%%%%%%%%%%%%%%%%%%%%%%%%%%%%%%%%%%%%%%%%%%%%%%%%%%

%%%%%%%%%%%%%%%%%%%%%%%%%%%%%%%%%%%%%%%%%%%%%%%%%%%%
\begin{rem}
By virtue of the above result \eqref{u-est-lower}, we can conclude that the $L^{\infty}$-decay estimate \eqref{u-sol-decay-Linf} is optimal with respect to the time decaying order $t^{-\frac{3}{4}-\frac{l}{2}}$, under the assumptions in Theorem~\ref{thm.main-approximation} and the additional condition 
\[
\int_{\R^{2}}u_{0}(x, y)dxdy \neq 0.
\] 
Compared with \eqref{u-est-lower-FH}, while the definition of $\mathcal{M}_{j}$ depends on the unknown function $u(x, y, t)$, in our result \eqref{u-est-lower}, $\int_{\R^{2}}u_{0}(x, y)dxdy$ in the lower bound depends only on the initial data $u_{0}(x, y)$.
\end{rem}
%%%%%%%%%%%%%%%%%%%%%%%%%%%%%%%%%%%%%%%%%%%%%%%%%%%%

Finally, we shall introduce the detailed asymptotic profile for the solution $u(x, y, t)$ to \eqref{ZKB}. 
Under the additional regularity and weight assumption on the initial data $u_{0}(x, y)$, combining the techniques used for the parabolic equations and for the Schr$\ddot{\mathrm{o}}$dinger equation, we have achieved to derive the asymptotic profile. Now, we define the new function $\psi(x, y, t)$ as follows: 
\begin{align}\label{DEF-psi}
\begin{split}
\psi(x, y, t):=&\, \mathcal{M}\left[u_{0}\right]\left(yt^{-1}\right)V(-x, y, t), \\
=&\, \mathcal{M}\left[u_{0}\right]\left(yt^{-1}\right)V_{*} \left( -xt^{-\frac{1}{2}}, yt^{-\frac{1}{4}} \right)t^{ -\frac{3}{4} }, \ \ (x, y)\in \R^{2}, \ t>0,
\end{split}
\end{align}
where $V(x, y, t)$ and $V_{*}(x, y)$ are defined by \eqref{DEF-V} and \eqref{DEF-V*}, respectively. On the other hand, the functional $\mathcal{M}\left[u_{0}\right](y)$ is defined by
\begin{equation}\label{DEF-mathM}
\mathcal{M}\left[u_{0}\right](y):=\sqrt{2\pi }\left(\int_{\R}\mathcal{F}_{\xi}^{-1}\left[\mathcal{F}\left[u_{0}\right]\left(\xi, \frac{y}{2\xi }\right)\right](x)dx\right), \ \ u_{0}\in L^{1}(\R^{2}). 
\end{equation}
Moreover, in order to state our final result, let us introduce the fractional derivative $D_{x}^{\gamma}$ for $\gamma \in \R$, via the Fourier transform as follows: 
\begin{equation}\label{frac-derivative}
D^{\gamma}_{x}h(x, y):=\mathcal{F}^{-1}_{\xi}\left[ |\xi|^{\gamma}\mathcal{F}_{x}[h](\xi, y) \right](x), \ \ \gamma \in \R. 
\end{equation}
Then, the following asymptotic formula can be established: 
%%%%%%%%%%%%%%%%%%%%%%%%%%%%%%%%%%%%%%%%%%%%%%%%%%%%
\begin{thm}\label{thm.main-u-asymptotic}
Let $p\ge2$ be an integer and $l\in \{0, 1\}$ satisfying $p>(4+2l)/3$. Assume that $u_{0}\in H^{2, 1}(\R^{2})\cap L^{1}(\R^{2})$ and $E_{0}$ is sufficiently small. 
In addition, suppose $y^{2}D_{x}^{-\alpha}u_{0}\in L^{1}(\R^{2})$ with $\alpha>1/2$ if $l=0$, while $y^{2}u_{0}\in L^{1}(\R^{2})$ if $l=1$.  
Then, the solution $u(x, y, t)$ to \eqref{ZKB} and the function $\psi(x, y, t)$ given by \eqref{DEF-psi} satisfy the following asymptotic formula: 
\begin{equation}\label{u-asymptotic}
\lim_{t\to \infty}t^{\frac{3}{4}+\frac{l}{2}}\left\|\p_{x}^{l}\left(u(\cdot, \cdot, t)-\psi(\cdot, \cdot, t)\right)\right\|_{L^{\infty}}=0. 
\end{equation}
\end{thm}
%%%%%%%%%%%%%%%%%%%%%%%%%%%%%%%%%%%%%%%%%%%%%%%%%%%%

%%%%%%%%%%%%%%%%%%%%%%%%%%%%%%%%%%%%%%%%%%%%%%%%%%%%
\begin{rem}
For $\psi(x, y, t)$, we remark that the following fact holds: 
\[
\left\| \p_{x}^{l}\psi(\cdot, \cdot, t)\right\|_{L^{\infty}}=\left\|\mathcal{M}\left[u_{0}\right](\cdot)\p_{x}^{l}V_{*}(\cdot, \cdot)\right\|_{L^{\infty}}t^{-\frac{3}{4}-\frac{l}{2}}, \ \ t>0, \ l \in \mathbb{N}\cup \{0\}. 
\]
Therefore, \eqref{u-asymptotic} certainly has meaning as the asymptotic formula, because the original solution $\p_{x}^{l}u(x, y, t)$ to \eqref{ZKB} has the same $L^{\infty}$-decay order $t^{-\frac{3}{4}-\frac{l}{2}}$ as $\p_{x}^{l}\psi(x, y, t)$ for $l=0, 1$. 
\end{rem}
%%%%%%%%%%%%%%%%%%%%%%%%%%%%%%%%%%%%%%%%%%%%%%%%%%%%

The rest of this paper is organized as follows. In Section~2, we discuss the local and global well-posedness for \eqref{ZKB}. First, we introduce a couple of properties for the functions belong the anisotropic Sobolev spaces $H^{s_{1}, s_{2}}(\R^{2})$ in Subsection~2.1, and then we prove the local well-posedness in Subsection~2.2. After that, in order to discuss the global well-posedness, we derive a priori estimates for the solution in Subsection~2.3 and prove Theorem~\ref{thm.GWP-H21} as a result. The $L^{\infty}$ and $L^{2}$-decay estimates are treated in Section~3. More precisely, Subsection~3.1 is devoted to the decay estimates for the linear solution. The estimates for the nonlinear solution, i.e. Theorem~\ref{thm.main-decay} is proved in Subsection~3.2. Next, the proofs of Theorems~\ref{thm.main-approximation} and \ref{thm.main-u-est-lower} are given in Section~4. In Subsection~4.1, we give the proof of Theorem~\ref{thm.main-approximation}, i.e. the approximation formula for the solution. After that, by using Theorem~\ref{thm.main-approximation}, we shall derive the lower bound for the $L^{\infty}$-norm of the solution $u(x, y, t)$ in Subsection~4.2 and prove Theorem~\ref{thm.main-u-est-lower}. Finally in Section~5, we derive the detailed asymptotic profile for the solution $u(x, y, t)$ to \eqref{ZKB} and give the proof of Theorem~\ref{thm.main-u-asymptotic}. 
The main novelty of this paper lies in the derivation of the optimal $L^{\infty}$-decay estimate (i.e. Theorems~\ref{thm.main-decay} and \ref{thm.main-u-est-lower}) and of the asymptotic profile (i.e. Theorem~\ref{thm.main-u-asymptotic}) for the solution $u(x, y, t)$ to \eqref{ZKB}. 
In particular, for deriving the asymptotic profile of the solution, we provide a new method that combines the techniques used for the parabolic equations and for the Schr$\ddot{\mathrm{o}}$dinger equation. 

\medskip
\par\noindent
%%%%%%%%%%%%%%%%%%%%%%%%%%%%%%%%%%%%%%%%%%%%%%%%%%%%
\textbf{\bf{Notations.}} 
%%%%%%%%%%%%%%%%%%%%%%%%%%%%%%%%%%%%%%%%%%%%%%%%%%%%

\medskip
We define the Fourier transform of $f$ and the inverse Fourier transform of $g$ as follows:
\[
\hat{f}(\xi, \eta) = \mathcal{F}[f](\xi, \eta) := \frac{1}{2\pi}\int_{\R^{2}}e^{-ix\xi-iy\eta}f(x, y)dxdy, \ \ 
\mathcal{F}^{-1}[g](x, y):=\frac{1}{2\pi}\int_{\R^{2}}e^{ix\xi+iy\eta}g(\xi, \eta)d\xi d\eta.
\]
We also define $\mathcal{F}_x$, $\mathcal{F}^{-1}_{\xi}$, $\mathcal{F}_y$ and $\mathcal{F}^{-1}_{\eta}$ as follows:
\begin{align*}
&\mathcal{F}_x[f](\xi, y) := \frac{1}{\sqrt{2\pi}}\int_{\R}e^{-ix\xi}f(x, y)dx, \ \ 
\mathcal{F}_{\xi}^{-1}[g](x, y):=\frac{1}{\sqrt{2\pi}}\int_{\R}e^{ix\xi}g(\xi, y)d\xi, \\
&\mathcal{F}_y[f](x, \eta) := \frac{1}{\sqrt{2\pi}}\int_{\R}e^{-iy\eta}f(x, y)dy, \ \ 
\mathcal{F}_{\eta}^{-1}[g](x, y):=\frac{1}{\sqrt{2\pi}}\int_{\R}e^{iy\eta}g(x, \eta)d\eta.
\end{align*}
Via the Fourier transform, we also define the fractional derivative $D_{x}^{\gamma}$ as the formula \eqref{frac-derivative}.  

\medskip
For $1\le p\le \infty$, $L^{p}(\R^{2})$ means the usual Lebesgue spaces. Also, the Schwartz space and its dual space are denoted by $\mathcal{S}(\R^{2})$ and $\mathcal{S}'(\R^{2})$, respectively. 
Moreover, for $s, s_{1}, s_{2}\in \R$, the Sobolev spaces $H^{s}(\R^{2})$ and the anisotropic Sobolev spaces $H^{s_{1}, s_{2}}(\R^{2})$ are defined by \eqref{DEF-Sobolev} and \eqref{DEF-aniso-Sobolev}, respectively. Note that if $s=s_{1}+s_{2}$, then we obtain $\left\|f\right\|_{H^{s_{1}, s_{2}}}\le \left\|f\right\|_{H^{s}}$ for $f\in H^{s}(\R^{2})$, i.e. $H^{s}(\R^{2})\hookrightarrow H^{s_{1}, s_{2}}(\R^{2})$.

\medskip
Let $T>0$, $1\le p \le \infty$ and $X$ be a Banach space. Then, $L^{p}(0,T; X)$ denotes the space of measurable functions $u: (0, T) \to X$ such that $\|u(t)\|_{X}$ belongs to $L^{p}(0, T)$. Also, $C([0, T]; X)$ denotes the subspace of $L^{\infty}(0, T; X)$ of all continuous functions $u: [0, T] \to X$. 

\medskip
Throughout this paper, $C$ denotes various positive constants, which may vary from line to line during computations. Also, it may depend on the norm of the initial data. However, we note that it does not depend on the space variable $(x, y)$ and the time variable $t$. 

%%%%%%%%%%%%%%%%%%%%%%%%%%%%%%%%%%%%%%%%%%%%%%%%%%%%%%%%%%%%%%%%%%%%%%%%%%%%%%%%%%%%%%%%%%%%%%%%%%%%%%%%
\section{Well-posedness}
%%%%%%%%%%%%%%%%%%%%%%%%%%%%%%%%%%%%%%%%%%%%%%%%%%%%%%%%%%%%%%%%%%%%%%%%%%%%%%%%%%%%%%%%%%%%%%%%%%%%%%%%

In this section, we shall prove our first main result Theorem~\ref{thm.GWP-H21}, i.e. we discuss the well-posedness for \eqref{ZKB}. First of all, we would like to transform \eqref{ZKB} to the corresponding integral equation \eqref{integral-eq} below. In order to doing that, let us introduce the following Cauchy problem: 
\begin{align}\label{LZKB}
\begin{split}
& \tilde{u}_{t} + \tilde{u}_{xxx} + \tilde{u}_{yyx} -\mu\tilde{u}_{xx}=0, \ \ (x, y) \in \R^{2}, \ t>0, \\
& \tilde{u}(x, y, 0) = u_{0}(x, y), \ \ (x, y) \in \R^{2}. 
\end{split}
\end{align}
The solution to the above Cauchy problem can be written by 
\begin{equation}\label{LZKB-sol}
\tilde{u}(x, y, t)=(U(t)*u_{0})(x, y), \ \ (x, y) \in \R^{2}, \ t>0, 
\end{equation}
where the integral kernel $U(x, y, t)$ is defined by 
\begin{equation}\label{DEF-U}
U(x, y, t) := \frac{1}{2\pi}\mathcal{F}^{-1} \left[ e^{ -\mu t \xi^{2} + it\xi^{3} + it\xi \eta^{2} } \right] (x, y). 
\end{equation}
Applying the Duhamel principle to \eqref{ZKB}, we obtain the following integral equation: 
\begin{equation}\label{integral-eq}
u(t)=U(t)*u_{0}-\frac{\beta}{p+1}\int_{0}^{t}\p_{x}U(t-\tau)*\left(u^{p+1}(\tau)\right)d\tau. 
\end{equation}
In what follows, we would like to consider this integral equation \eqref{integral-eq}. 

%%%%%%%%%%%%%%%%%%%%%%%%%%%%%%%%%%%%%%%%%%%%%%%%%%%%%%%%%%%%%%%%%%%%%%%%%%%%%%%%%%%%%%%%%%%%%%%%%%%%%%%%
\subsection{Preliminaries}  
%%%%%%%%%%%%%%%%%%%%%%%%%%%%%%%%%%%%%%%%%%%%%%%%%%%%%%%%%%%%%%%%%%%%%%%%%%%%%%%%%%%%%%%%%%%%%%%%%%%%%%%%

In this subsection, we shall prepare a couple of lemmas and propositions to discuss the well-posedness for \eqref{ZKB}. 
We start with introducing the properties of the integral kernel $U(x, y, t)$: 
%%%%%%%%%%%%%%%%%%%%%%%%%%%%%%%%%%%%%%%%%%%%%%%%%%%%
\begin{lem}\label{lem.U-Hs-est}
Let $s, s_{1}, s_{2}\in \R$ and $a>0$. Then, the following estimates 
\begin{align}
\left\|\left(1-\p_{x}^{2}\right)^{\frac{a}{2}}\left(U(t)*f\right)\right\|_{H^{s}}
&\le C\left\{1+(\mu t)^{-\frac{a}{2}}\right\} \left\|f\right\|_{H^{s}}, \ \ t>0, \label{U-Hs-est-1}\\
\left\|U(t)*g\right\|_{H^{s_{1}+a, s_{2}}}
&\le C\left\{1+(\mu t)^{-\frac{a}{2}}\right\} \left\|g\right\|_{H^{s_{1}, s_{2}}}, \ \ t>0, \label{U-Hs-est-2}
\end{align}
hold for $f\in H^{s}(\R^{2})$ and $g\in H^{s_{1}, s_{2}}(\R^{2})$, respectively. Here, $U(x, y, t)$ is defined by \eqref{DEF-U}. 
\end{lem}
%%%%%%%%%%%%%%%%%%%%%%%%%%%%%%%%%%%%%%%%%%%%%%%%%%%%
%%%%%%%%%%%%%%%%%%%%%%%%%%%%%%%%%%%%%%%%%%%%%%%%%%%%
\begin{proof}
For the second inequality \eqref{U-Hs-est-2}, it follows from a direct computation that 
\begin{align*}
&\left\|U(t)*g\right\|_{H^{s_{1}+a, s_{2}}}
=\left\| \left(1+\xi^{2}\right)^{\frac{s_{1}+a}{2}} \left(1+\eta^{2}\right)^{\frac{s_{2}}{2}} e^{-\mu t\xi^{2}}e^{it\xi^{3}+it\xi \eta^{2}}\hat{g}(\xi, \eta)\right\|_{L^{2}_{\xi \eta}} \nonumber \\
&\le \sup_{\xi \in \R}\left\{ \left(1+\xi^{2}\right)^{\frac{a}{2}} e^{-\mu t\xi^{2}}\right\}\left\|g\right\|_{H^{s_{1}, s_{2}}}
=\left\|g\right\|_{H^{s_{1}, s_{2}}}
\begin{cases}
\displaystyle 1, &\displaystyle t\ge \frac{a}{2\mu},  \\
\displaystyle \left(\frac{a}{2\mu t}\right)^{\frac{a}{2}}e^{-\mu t\left(\frac{a}{2\mu t}-1\right)}, &\displaystyle 0<t< \frac{a}{2\mu}
\end{cases}\nonumber \\
&\le \left\{1+\left(\frac{a}{2\mu t}\right)^{\frac{a}{2}}\right\}\left\|g\right\|_{H^{s_{1}, s_{2}}}
\le C\left\{1+(\mu t)^{-\frac{a}{2}}\right\}\left\|g\right\|_{H^{s_{1}, s_{2}}}, \ \ t>0. 
\end{align*}
The first inequality \eqref{U-Hs-est-1} can be shown in the same way. 
\end{proof}
%%%%%%%%%%%%%%%%%%%%%%%%%%%%%%%%%%%%%%%%%%%%%%%%%%%%

Secondly, we shall prove that the anisotropic Sobolev space $H^{s_{1}, s_{2}}(\R^{2})$ is a Banach algebra if $s_{1}>1/2$ and $s_{2}>1/2$. 
This fact (\eqref{BAnach-al} below) plays an important role in the proof of the local well-posedness for \eqref{ZKB}. 
In order to state such a result, let us prepare the following lemma. For simplicity, we set $\langle \xi \rangle:=\left(1+\xi^{2}\right)^{\frac{1}{2}}$ and $\langle \eta \rangle:=\left(1+\eta^{2}\right)^{\frac{1}{2}}$ in what follows.  
%%%%%%%%%%%%%%%%%%%%%%%%%%%%%%%%%%%%%%%%%%%%%%%%%%%%
\begin{lem}\label{lem.Banach-al-pre}
Let $s_{1}>1/2$, $s_{2}>1/2$, $F\in H^{s_{1}, 0}(\R^{2})$ and $G\in H^{0, s_{2}}(\R^{2})$. Then, we have 
\begin{equation}\label{BAnach-al-pre}
\left\|FG\right\|_{L^{2}} \le C\left\|F\right\|_{H^{s_{1}, 0}} \left\|G\right\|_{H^{0, s_{2}}}. 
\end{equation}
\end{lem}
%%%%%%%%%%%%%%%%%%%%%%%%%%%%%%%%%%%%%%%%%%%%%%%%%%%%
%%%%%%%%%%%%%%%%%%%%%%%%%%%%%%%%%%%%%%%%%%%%%%%%%%%%
\begin{proof}
By the Plancherel theorem, the Young inequality and the Cauchy--Schwarz inequality, we can see that 
\begin{align*}
\left\|FG\right\|_{L^{2}}
&=\frac{1}{2\pi } \left\|\hat{F}*\hat{G}\right\|_{L^{2}}
=\frac{1}{2\pi } \left\| \int_{\R^{2}}\hat{F}(\xi_{1}, \eta_{1})\hat{G}(\xi-\xi_{1}, \eta-\eta_{1})d\xi_{1} d\eta_{1}\right\|_{L^{2}_{\xi \eta}} \\
&\le \left\|\int_{\R} \left\|\hat{F}(\xi, \eta_{1})\right\|_{L^{1}_{\xi}} \left\|\hat{G}(\xi, \eta-\eta_{1})\right\|_{L^{2}_{\xi}} d\eta_{1} \right\|_{L^{2}_{\eta}}
\le \left\| \left\|\hat{F}(\xi, \eta)\right\|_{L^{1}_{\xi}} \right\|_{L^{2}_{\eta}} \left\|\left\|\hat{G}(\xi, \eta)\right\|_{L^{2}_{\xi}} \right\|_{L^{1}_{\eta}} \\
%&\le \left\| \left\|(1+\xi^{2})^{-\frac{s_{1}}{2}}\right\|_{L^{2}_{\xi}} \left\|(1+\xi^{2})^{\frac{s_{1}}{2}}\hat{F}(\xi, \eta)\right\|_{L^{2}_{\xi}}  \right\|_{L^{2}_{\eta}} \left\|(1+\eta^{2})^{-\frac{s_{2}}{2}}\right\|_{L^{2}_{\eta}} \left\|(1+\eta^{2})^{\frac{s_{2}}{2}}\left\|\hat{G}(\xi, \eta)\right\|_{L^{2}_{\xi}}  \right\|_{L^{2}_{\eta}} \\
&\le \left\| \left\|\langle \xi \rangle ^{-s_{1}}\right\|_{L^{2}_{\xi}} \left\|\langle \xi \rangle ^{s_{1}}\hat{F}(\xi, \eta)\right\|_{L^{2}_{\xi}}  \right\|_{L^{2}_{\eta}} \left\|\langle \eta \rangle ^{-s_{2}}\right\|_{L^{2}_{\eta}} \left\|\langle \eta \rangle ^{s_{2}}\left\|\hat{G}(\xi, \eta)\right\|_{L^{2}_{\xi}}  \right\|_{L^{2}_{\eta}} \\
&\le C\left\|F\right\|_{H^{s_{1}, 0}} \left\|G\right\|_{H^{0, s_{2}}}. 
\end{align*}
Here, we used the fact $\langle \xi \rangle ^{-s_{1}}, \langle \eta \rangle ^{-s_{2}}\in L^{2}(\R)$ if $s_{1}>1/2$ and $s_{2}>1/2$ holds. 
\end{proof}
%%%%%%%%%%%%%%%%%%%%%%%%%%%%%%%%%%%%%%%%%%%%%%%%%%%%
By virtue of the above Lemma~\ref{lem.Banach-al-pre}, we can show that the anisotropic Sobolev space $H^{s_{1}, s_{2}}(\R^{2})$ is a Banach algebra under the conditions $s_{1}>1/2$ and $s_{2}>1/2$ as follows: 
%%%%%%%%%%%%%%%%%%%%%%%%%%%%%%%%%%%%%%%%%%%%%%%%%%%%
\begin{prop}\label{prop.Banach-al}
Let $s_{1}>1/2$, $s_{2}>1/2$ and $f, g\in H^{s_{1}, s_{2}}(\R^{2})$. Then, we have 
\begin{equation}\label{BAnach-al}
\left\|fg\right\|_{H^{s_{1}, s_{2}}} \le C\left\|f\right\|_{H^{s_{1}, s_{2}}} \left\|g\right\|_{H^{s_{1}, s_{2}}}. 
\end{equation}
\end{prop}
%%%%%%%%%%%%%%%%%%%%%%%%%%%%%%%%%%%%%%%%%%%%%%%%%%%%
%%%%%%%%%%%%%%%%%%%%%%%%%%%%%%%%%%%%%%%%%%%%%%%%%%%%
\begin{proof}
It follows from the Plancherel theorem that 
\begin{align}
\left\|fg\right\|_{H^{s_{1}, s_{2}}}
%&=\frac{1}{2\pi }\left\| \langle \xi \rangle ^{s_{1}} \langle \eta \rangle ^{s_{2}}\left( \hat{f}*\hat{g}\right)\right\|_{L^{2}_{\xi \eta}} \nonumber \\
&\le \left\| \left( \langle \xi \rangle ^{s_{1}} \langle \eta \rangle ^{s_{2}} \hat{f} \right) * \hat{g} \right\|_{L^{2}_{\xi \eta}}
+\left\| \left( \langle \xi \rangle ^{s_{1}} \hat{f} \right) * \left( \langle \eta \rangle ^{s_{2}} \hat{g} \right) \right\|_{L^{2}_{\xi \eta}} \nonumber \\
&\ \ \ \ + \left\| \left( \langle \eta \rangle ^{s_{2}} \hat{f} \right) * \left(\langle \xi \rangle ^{s_{1}}  \hat{g} \right)\right\|_{L^{2}_{\xi \eta}}
+\left\| \hat{f} * \left( \langle \xi \rangle ^{s_{1}} \langle \eta \rangle ^{s_{2}} \hat{g} \right)\right\|_{L^{2}_{\xi \eta}} 
=: I_{1}+I_{2}+I_{3}+I_{4}. \label{fg-est}
\end{align}
Now, let us evaluate the right hand side of \eqref{fg-est}. For $I_{1}$, by using the Young inequality and the Cauchy--Schwarz inequality, we obtain 
\begin{align}
I_{1} 
%&\le \left\| \langle \xi \rangle ^{s_{1}} \langle \eta \rangle ^{s_{2}} \hat{f} \right\|_{L^{2}_{\xi \eta}} \left\|\hat{g}\right\|_{L^{1}_{\xi \eta}} \nonumber \\
&\le \left\| \langle \xi \rangle ^{s_{1}} \langle \eta \rangle ^{s_{2}} \hat{f} \right\|_{L^{2}_{\xi \eta}} \left\| \langle \xi \rangle ^{-s_{1}} \langle \eta \rangle ^{-s_{2}} \right\|_{L^{2}_{\xi \eta}} \left\| \langle \xi \rangle ^{s_{1}} \langle \eta \rangle ^{s_{2}} \hat{g} \right\|_{L^{2}_{\xi \eta}}  
\le C\left\|f\right\|_{H^{s_{1}, s_{2}}} \left\|g\right\|_{H^{s_{1}, s_{2}}}. \label{fg-I1-est}
\end{align}
Next, applying Lemma~\ref{lem.Banach-al-pre} for $I_{2}$ as $\hat{F}(\xi, \eta)=\langle \eta \rangle ^{s_{2}} \hat{g}(\xi, \eta)$ and $\hat{G}(\xi, \eta)= \langle \xi \rangle ^{s_{1}}\hat{f}(\xi, \eta)$, we get
\begin{equation}\label{fg-est-I2}
I_{2} = 2\pi \left\|FG\right\|_{L^{2}} \le C \left\|F\right\|_{H^{s_{1}, 0}} \left\|G\right\|_{H^{0, s_{2}}}
=\left\|f\right\|_{H^{s_{1}, s_{2}}} \left\|g\right\|_{H^{s_{1}, s_{2}}}.
\end{equation}
In addition, the estimates for $I_{3}$ and $I_{4}$ can be given by the same way as \eqref{fg-est-I2} and \eqref{fg-I1-est}, respectively. 
Therefore, combining these facts and \eqref{fg-est}, we can prove the desired result \eqref{BAnach-al}. 
\end{proof}
%%%%%%%%%%%%%%%%%%%%%%%%%%%%%%%%%%%%%%%%%%%%%%%%%%%%

Finally in this subsection, we would like to prepare some useful inequalities, which related to the Gagliardo--Nirenberg inequality. 
The following two lemmas will be used throughout the paper. 
%%%%%%%%%%%%%%%%%%%%%%%%%%%%%%%%%%%%%%%%%%%%%%%%%%%%
\begin{lem}\label{lem.est-Linfty}
Let $f \in H^{1, 1}(\R^{2})$. Then, the following inequality holds: 
\begin{equation}\label{est-Linf}
\left\|f\right\|_{L^{\infty}}^{2}\le 2\left(\left\|f_{x}\right\|_{L^{2}}\left\|f_{y}\right\|_{L^{2}} + \left\|f\right\|_{L^{2}}\left\|f_{xy}\right\|_{L^{2}}\right). 
\end{equation}
\end{lem}
%%%%%%%%%%%%%%%%%%%%%%%%%%%%%%%%%%%%%%%%%%%%%%%%%%%%
%%%%%%%%%%%%%%%%%%%%%%%%%%%%%%%%%%%%%%%%%%%%%%%%%%%%
\begin{proof}
From the fundamental theorem of calculus and the Cauchy--Schwarz inequality, we obtain 
\begin{align*}
f(x, y)^{2}
&=\int_{-\infty}^{x}\int_{-\infty}^{y}\p_{z}\p_{w}\left(f(z, w)^{2}\right)dwdz \\
&=\int_{-\infty}^{x}\int_{-\infty}^{y}2\left(\p_{z}f(z, w)\p_{w}f(z, w)+f(z, w)\p_{z}\p_{w}f(z, w)\right)dwdz \\
&\le 2\left(\left\|f_{x}\right\|_{L^{2}}\left\|f_{y}\right\|_{L^{2}} + \left\|f\right\|_{L^{2}}\left\|f_{xy}\right\|_{L^{2}}\right), \ \ (x, y)\in \R^{2}. 
\end{align*}
Therefore, we can say that the desired estimate \eqref{est-Linf} is true. 
\end{proof}
%%%%%%%%%%%%%%%%%%%%%%%%%%%%%%%%%%%%%%%%%%%%%%%%%%%%

%%%%%%%%%%%%%%%%%%%%%%%%%%%%%%%%%%%%%%%%%%%%%%%%%%%%
\begin{lem}\label{lem.est-L2q}
Let $f \in H^{1}(\R^{2})$ and $q\in \mathbb{N}$. Then, the following inequality holds: 
\begin{equation}\label{est-L2q}
\left\|f\right\|_{L^{2q}}^{2q}
\le (q!)^{2}\left\|f\right\|_{L^{2}}^{2} \left\|f_{x}\right\|_{L^{2}}^{q-1} \left\|f_{y}\right\|_{L^{2}}^{q-1}. 
\end{equation}
\end{lem}
%%%%%%%%%%%%%%%%%%%%%%%%%%%%%%%%%%%%%%%%%%%%%%%%%%%%
%%%%%%%%%%%%%%%%%%%%%%%%%%%%%%%%%%%%%%%%%%%%%%%%%%%%
\begin{proof}
From the fundamental theorem of calculus, we obtain 
\begin{align}
f(x, y)^{q}
&=\int_{-\infty}^{x}\p_{\theta}\left(f(\theta, y)^{q}\right)d\theta
=q\int_{-\infty}^{x}f(\theta, y)^{q-1}\p_{\theta}f(\theta, y)d\theta \nonumber \\
&\le q\left\|f(\cdot, y)^{q-1}\right\|_{L^{2}}\left\|f_{x}(\cdot, y)\right\|_{L^{2}}, \ \ (x, y) \in \R^{2}.  \label{pointwise-q-1}
\end{align}
In addition, we similarly have 
\begin{align}
f(x, y)^{q}\le q\left\|f(x, \cdot)^{q-1}\right\|_{L^{2}}\left\|f_{y}(x, \cdot)\right\|_{L^{2}}, \ \ (x, y) \in \R^{2}.  \label{pointwise-q-2}
\end{align}
Therefore, it follows from \eqref{pointwise-q-1}, \eqref{pointwise-q-2} and the Cauchy--Schwarz inequality that 
\begin{align}
\left\|f\right\|_{L^{2q}}^{2q}
&=\int_{\R^{2}}f(x, y)^{2q}dxdy  \nonumber \\
&\le q^{2}\left( \int_{\R}\left\|f(\cdot, y)^{q-1}\right\|_{L^{2}}\left\|f_{x}(\cdot, y)\right\|_{L^{2}}dy \right)
\left(\int_{\R}\left\|f(x, \cdot)^{q-1}\right\|_{L^{2}}\left\|f_{y}(x, \cdot)\right\|_{L^{2}}dx\right)  \nonumber \\
&\le q^{2}\left\|f^{q-1}\right\|_{L^{2}}^{2} \left\|f_{x}\right\|_{L^{2}} \left\|f_{y}\right\|_{L^{2}} 
=q^{2}\left\|f\right\|_{L^{2(q-1)}}^{2(q-1)} \left\|f_{x}\right\|_{L^{2}} \left\|f_{y}\right\|_{L^{2}}. \label{est-L2q-pre}
\end{align}
Finally, using \eqref{est-L2q-pre} repeatedly, we get 
\begin{align*}
\left\|f\right\|_{L^{2q}}^{2q}
&\le q^{2}\left\|f\right\|_{L^{2(q-1)}}^{2(q-1)} \left\|f_{x}\right\|_{L^{2}} \left\|f_{y}\right\|_{L^{2}} 
\le q^{2}(q-1)^{2}\left\|f\right\|_{L^{2(q-2)}}^{2(q-2)} \left\|f_{x}\right\|_{L^{2}}^{2} \left\|f_{y}\right\|_{L^{2}}^{2} \nonumber \\
%&\le q^{2}(q-1)^{2}(q-2)^{2}\left\|f\right\|_{L^{2(q-3)}}^{2(q-3)} \left\|f_{x}\right\|_{L^{2}}^{3} \left\|f_{y}\right\|_{L^{2}}^{3}  \nonumber \\
&\le \cdots \le q^{2}(q-1)^{2}(q-2)^{2}\cdots 2^{2} \cdot \left\|f\right\|_{L^{2}}^{2} \left\|f_{x}\right\|_{L^{2}}^{q-1} \left\|f_{y}\right\|_{L^{2}}^{q-1}.  
\end{align*}
Thus, we can say that the desired estimate \eqref{est-L2q} is true. 
\end{proof}
%%%%%%%%%%%%%%%%%%%%%%%%%%%%%%%%%%%%%%%%%%%%%%%%%%%%

%%%%%%%%%%%%%%%%%%%%%%%%%%%%%%%%%%%%%%%%%%%%%%%%%%%%%%%%%%%%%%%%%%%%%%%%%%%%%%%%%%%%%%%%%%%%%%%%%%%%%%%%
\subsection{Local well-posedness}
%%%%%%%%%%%%%%%%%%%%%%%%%%%%%%%%%%%%%%%%%%%%%%%%%%%%%%%%%%%%%%%%%%%%%%%%%%%%%%%%%%%%%%%%%%%%%%%%%%%%%%%%

In this subsection, we shall prove the local well-posedness for \eqref{ZKB} in the Sobolev spaces $H^{s}(\R^{2})$ and the anisotropic Sobolev spaces $H^{s_{1}, s_{2}}(\R^{2})$. Actually, the following result can be established: 
%%%%%%%%%%%%%%%%%%%%%%%%%%%%%%%%%%%%%%%%%%%%%%%%%%%%
\begin{thm}\label{thm.LWP}
Let $p\ge1$ be an integer, $s>1$, $s_{1}>1/2$ and $s_{2}>1/2$.  

\smallskip
\noindent
{\rm (i)}
If $u_{0}\in H^{s}(\R^{2})$, then \eqref{ZKB} is locally well-posed in $H^{s}(\R^{2})$. 
More precisely, there exist a positive constant $T>0$ and a unique local mild solution $u \in C([0, T]; H^{s}(\R^{2}))$ to \eqref{ZKB} satisfying 
\begin{equation}\label{local-est-Hs}
\sup_{0\le t \le T}\left\|u(\cdot, \cdot, t)\right\|_{H^{s}}\le 2\left\|u_{0}\right\|_{H^{s}}.  
\end{equation}
Moreover, the mapping $u_{0} \mapsto u$ is continuous from $H^{s}(\R^{2})$ into $C([0, T]; H^{s}(\R^{2}))$. 

\smallskip
\noindent
{\rm (ii)} 
If $u_{0}\in H^{s_{1}, s_{2}}(\R^{2})$, then \eqref{ZKB} is locally well-posed in $H^{s_{1}, s_{2}}(\R^{2})$. 
More precisely, there exist a positive constant $T>0$ and a unique local mild solution $u \in C([0, T]; H^{s_{1}, s_{2}}(\R^{2}))$ to \eqref{ZKB} satisfying 
\begin{equation}\label{local-est-Hs1s2}
\sup_{0\le t \le T}\left\|u(\cdot, \cdot, t)\right\|_{H^{s_{1}, s_{2}}}\le 2\left\|u_{0}\right\|_{H^{s_{1}, s_{2}}}.  
\end{equation}
Moreover, the mapping $u_{0} \mapsto u$ is continuous from $H^{s_{1}, s_{2}}(\R^{2})$ into $C([0, T]; H^{s_{1}, s_{2}}(\R^{2}))$. 
\end{thm}
%%%%%%%%%%%%%%%%%%%%%%%%%%%%%%%%%%%%%%%%%%%%%%%%%%%%
%%%%%%%%%%%%%%%%%%%%%%%%%%%%%%%%%%%%%%%%%%%%%%%%%%%%
\begin{proof}
We solve the integral equation \eqref{integral-eq} by using the contraction mapping principle for
\begin{equation}\label{integral-map}
\Phi[u](t):=U(t)*u_{0}-\frac{\beta}{p+1}\int_{0}^{t}\p_{x}U(t-\tau)*\left(u^{p+1}(\tau)\right)d\tau. 
\end{equation}
For $T\in (0, 1)$, let us define the Banach spaces $X_{T}^{s}$ and $Y_{T}^{s_{1}, s_{2}}$ as follows: 
\begin{align*}
&X_{T}^{s}:=\left\{u\in C([0, T]; H^{s}(\R^{2})); \ \sup_{0\le t \le T}\left\|u(\cdot, \cdot, t)\right\|_{H^{s}}\le 2\left\|u_{0}\right\|_{H^{s}}\right\}, \\
&Y_{T}^{s_{1}, s_{2}}:=\left\{u\in C([0, T]; H^{s_{1}, s_{2}}(\R^{2})); \ \sup_{0\le t \le T}\left\|u(\cdot, \cdot, t)\right\|_{H^{s_{1}, s_{2}}}\le 2\left\|u_{0}\right\|_{H^{s_{1}, s_{2}}}\right\}. 
\end{align*}

First, we shall prove {\rm (i)}. In what follows, let $s>1$ and $u\in X_{T}^{s}$. Then, from \eqref{integral-map}, Lemma~\ref{lem.U-Hs-est} and the property of Banach algebra for $H^{s}(\R^{2})$, we have
\begin{align} 
\left\|\Phi[u](t)\right\|_{H^{s}}
%&\le \left\|U(t)*u_{0}\right\|_{H^{s}}+\frac{|\beta|}{p+1}\left\| \int_{0}^{t}\p_{x}U(t-\tau)*\left(u^{p+1}(\tau)\right)d\tau \right\|_{H^{s}} \nonumber \\
&\le \left\|U(t)*u_{0}\right\|_{H^{s}} + |\beta|\int_{0}^{t}\left\|\left(1-\p_{x}^{2}\right)^{\frac{1}{2}}\left(U(t-\tau)*\left(u^{p+1}(\tau)\right)\right)\right\|_{H^{s}}d\tau \nonumber \\
%&\le \left\|u_{0}\right\|_{H^{s}} + |\beta|\int_{0}^{t}C\left\{1+\left(\mu(t-\tau)\right)^{-\frac{1}{2}}\right\}\left\|u^{p+1}(\cdot, \cdot, \tau)\right\|_{H^{s}}d\tau \nonumber \\
&\le \left\|u_{0}\right\|_{H^{s}} + C|\beta|\int_{0}^{t}\left\{1+\left(\mu(t-\tau)\right)^{-\frac{1}{2}}\right\}\left\|u(\cdot, \cdot, \tau)\right\|_{H^{s}}^{p+1}d\tau \nonumber \\
&\le \left\|u_{0}\right\|_{H^{s}} +C|\beta|\left(t+2\mu^{-\frac{1}{2}}t^{\frac{1}{2}}\right)2^{p+1}\left\|u_{0}\right\|_{H^{s}}^{p+1} \nonumber \\
&\le \left(1+C_{\beta, \mu, p}\left\|u_{0}\right\|_{H^{s}}^{p}T^{\frac{1}{2}}\right)\left\|u_{0}\right\|_{H^{s}}. \nonumber 
\end{align}
Therefore, by choosing $T\in (0, 1)$ small enough, we get 
\begin{equation*}
\sup_{0\le t\le T}\left\|\Phi[u](t)\right\|_{H^{s}}\le 2\left\|u_{0}\right\|_{H^{s}}. 
\end{equation*}
This implies $\Phi[u]\in X_{T}^{s}$. Moreover, for $u, v\in X_{T}^{s}$, by the same argument, we obtain 
\begin{align}
\left\|\Phi[u](t)-\Phi[v](t)\right\|_{H^{s}}
&\le C|\beta|\int_{0}^{t}\left\{1+\left(\mu(t-\tau)\right)^{-\frac{1}{2}}\right\}\left\|u^{p+1}(\cdot, \cdot, \tau)-v^{p+1}(\cdot, \cdot, \tau)\right\|_{H^{s}}d\tau \nonumber \\
&\le C_{\beta, \mu, p}\left\|u_{0}\right\|_{H^{s}}^{p}T^{\frac{1}{2}}\sup_{0\le t \le T}\left\|u(\cdot, \cdot, t)-v(\cdot, \cdot, t)\right\|_{H^{s}}. \nonumber 
\end{align}
Thus, by choosing $T\in (0, 1)$ small enough, $\Phi[u]$ becomes the contraction mapping on $X_{T}^{s}$. 
Therefore, applying the Banach fixed point theorem, we can see that there exists a unique local mild solution $u \in C([0, T]; H^{s}(\R^{2}))$ to \eqref{ZKB} satisfying \eqref{local-est-Hs}. We omit the remainder of the proof, since it can be given by a standard argument.

\smallskip
Next, we shall prove {\rm (ii)}. By virtue of Proposition~\ref{prop.Banach-al}, for $s_{1}>1/2$ and $s_{2}>1/2$, we have 
\[
\left\|u^{p+1}(\cdot, \cdot, \tau)\right\|_{H^{s_{1}, s_{2}}}\le C\left\|u(\cdot, \cdot, \tau)\right\|_{H^{s_{1}, s_{2}}}^{p+1}. 
\]
Therefore, we are able to prove {\rm (ii)} by the same argument for {\rm (i)}. This completes the proof. 
\end{proof}
%%%%%%%%%%%%%%%%%%%%%%%%%%%%%%%%%%%%%%%%%%%%%%%%%%%%

In the rest of this section, we shall introduce an additional result related to the regularity of the solution to \eqref{ZKB}. 
The following proposition will be used implicitly to derive a priori estimates for the solution in the next subsection. 
%%%%%%%%%%%%%%%%%%%%%%%%%%%%%%%%%%%%%%%%%%%%%%%%%%%%
\begin{prop}\label{prop.regularity}
Let $p\ge1$ be an integer and $u_{0}\in H^{2, 1}(\R^{2})$. 
Suppose that $u\in C([0, T]; H^{2, 1}(\R^{2}))$ is a solution to \eqref{ZKB} on $[0, T]$ for some $T>0$. 
Then, the solution satisfies $\p_{x}u\in L^{2}(0, T; H^{2, 1}(\R^{2}))$. 
\end{prop}
%%%%%%%%%%%%%%%%%%%%%%%%%%%%%%%%%%%%%%%%%%%%%%%%%%%%
%%%%%%%%%%%%%%%%%%%%%%%%%%%%%%%%%%%%%%%%%%%%%%%%%%%%
\begin{proof}
For the solution $u(x, y, t)$ to \eqref{ZKB} on $[0, T]$, it follows from \eqref{integral-eq} that 
\begin{align}
\p_{x}u(t)
&=U(t)*\p_{x}u_{0}-\frac{\beta}{p+1}\int_{0}^{t}U(t-\tau)*\p_{x}^{2}\left(u^{p+1}(\tau)\right)d\tau \nonumber \\
&=:P(x, y, t)-\frac{\beta}{p+1}Q(x, y, t). \label{integral-eq-dx}
\end{align}
By a standard calculation, we have 
\begin{align}
\int_{0}^{T}\left\|P(\cdot, \cdot, t)\right\|_{H^{2, 1}}^{2}dt
&=\int_{0}^{T}\int_{\R^{2}}\xi^{2}\left(1+\xi^{2}\right)^{2}\left(1+\eta^{2}\right)e^{-2\mu t\xi^{2}}\left|\hat{u}_{0}(\xi, \eta)\right|^{2}d\xi d\eta dt \nonumber \\
&= \int_{\R^{2}}\left(1+\xi^{2}\right)^{2}\left(1+\eta^{2}\right)\left|\hat{u}_{0}(\xi, \eta)\right|^{2}\left(\int_{0}^{T}\xi^{2}e^{-2\mu t\xi^{2}}dt \right)d\xi d\eta \nonumber \\
&\le \frac{1}{2\mu}  \int_{\R^{2}}\left(1+\xi^{2}\right)^{2}\left(1+\eta^{2}\right)\left|\hat{u}_{0}(\xi, \eta)\right|^{2} d\xi d\eta
=\frac{1}{2\mu} \left\|u_{0}\right\|_{H^{2, 1}}^{2}. \label{int-P}
\end{align}
Also, by using the Cauchy--Schwarz inequality for the $\tau$-integral, we get  
\begin{align}
&\left\|Q(\cdot, \cdot, t)\right\|_{H^{2, 1}}^{2}
=\int_{\R^{2}}\xi^{4}\left(1+\xi^{2}\right)^{2}\left(1+\eta^{2}\right)\left|\int_{0}^{t}e^{-\mu (t-\tau)\xi^{2}+it\xi^{3}+it\xi \eta^{2}}\widehat{u^{p+1}}(\xi, \eta, \tau)d\tau\right|^{2}d\xi d\eta \nonumber \\
&\le \int_{\R^{2}}\xi^{4}\left(1+\xi^{2}\right)^{2}\left(1+\eta^{2}\right)\left(\int_{0}^{t}e^{-\mu (t-\tau)\xi^{2}}d\tau \right)
\left(\int_{0}^{t}e^{-\mu (t-\tau)\xi^{2}}\left|\widehat{u^{p+1}}(\xi, \eta, \tau)\right|^{2}d\tau\right)d\xi d\eta \nonumber \\
&\le \frac{1}{\mu}\int_{\R^{2}}\int_{0}^{t}\xi^{2}\left(1+\xi^{2}\right)^{2}\left(1+\eta^{2}\right)e^{-\mu (t-\tau)\xi^{2}}\left|\widehat{u^{p+1}}(\xi, \eta, \tau)\right|^{2}d\tau d\xi d\eta.  \nonumber 
\end{align}
Therefore, changing the integration order, and applying Proposition~\ref{prop.Banach-al}, we can see that 
\begin{align}
\int_{0}^{T}\left\|Q(\cdot, \cdot, t)\right\|_{H^{2, 1}}^{2}dt
%&\le \frac{1}{\mu}\int_{\R^{2}}\int_{0}^{T}\int_{0}^{t}\xi^{2}\left(1+\xi^{2}\right)^{2}\left(1+\eta^{2}\right)e^{-\mu (t-\tau)\xi^{2}}\left|\widehat{u^{p+1}}(\xi, \eta, \tau)\right|^{2}d\tau dt d\xi d\eta \nonumber \\
&\le \frac{1}{\mu}\int_{\R^{2}}\int_{0}^{T}\int_{\tau}^{T}\xi^{2}\left(1+\xi^{2}\right)^{2}\left(1+\eta^{2}\right)e^{-\mu (t-\tau)\xi^{2}}\left|\widehat{u^{p+1}}(\xi, \eta, \tau)\right|^{2}dt d\tau d\xi d\eta \nonumber \\
&\le \frac{1}{\mu}\int_{\R^{2}}\int_{0}^{T}\left(\int_{0}^{\infty}\xi^{2}e^{-\mu \sigma \xi^{2}}d\sigma\right)\left(1+\xi^{2}\right)^{2}\left(1+\eta^{2}\right)\left|\widehat{u^{p+1}}(\xi, \eta, \tau)\right|^{2}d\tau d\xi d\eta \nonumber \\
%&\le \frac{1}{\mu^{2}}\int_{\R^{2}}\int_{0}^{T}\left(1+\xi^{2}\right)^{2}\left(1+\eta^{2}\right)\left|\widehat{u^{p+1}}(\xi, \eta, \tau)\right|^{2}d\tau d\xi d\eta \nonumber \\
&\le \frac{1}{\mu^{2}}\int_{0}^{T}\left\|u^{p+1}(\cdot, \cdot, t)\right\|_{H^{2, 1}}^{2}dt 
%\le \frac{1}{\mu^{2}}\int_{0}^{T}\left\|u(\cdot, \cdot, t)\right\|_{H^{2, 1}}^{2(p+1)}dt \nonumber \\
\le  \frac{T}{\mu^{2}}\left(\sup_{0\le t\le T}\left\|u(\cdot, \cdot, t)\right\|_{H^{2, 1}}\right)^{2(p+1)}. \label{int-Q}
\end{align}
Combining \eqref{integral-eq-dx}, \eqref{int-P} and \eqref{int-Q}, we can obtain the desired result $\p_{x}u\in L^{2}(0, T; H^{2, 1}(\R^{2}))$. 
\end{proof}
%%%%%%%%%%%%%%%%%%%%%%%%%%%%%%%%%%%%%%%%%%%%%%%%%%%%

%%%%%%%%%%%%%%%%%%%%%%%%%%%%%%%%%%%%%%%%%%%%%%%%%%%%
\begin{rem}
Because $H^{2, 1}(\R^{2})\hookrightarrow H^{1}(\R^{2})$, $\p_{x}u\in L^{2}(0, T; H^{1}(\R^{2}))$ also holds. 
\end{rem}
%%%%%%%%%%%%%%%%%%%%%%%%%%%%%%%%%%%%%%%%%%%%%%%%%%%%

%%%%%%%%%%%%%%%%%%%%%%%%%%%%%%%%%%%%%%%%%%%%%%%%%%%%%%%%%%%%%%%%%%%%%%%%%%%%%%%%%%%%%%%%%%%%%%%%%%%%%%%%
\subsection{Global well-posedness}
%%%%%%%%%%%%%%%%%%%%%%%%%%%%%%%%%%%%%%%%%%%%%%%%%%%%%%%%%%%%%%%%%%%%%%%%%%%%%%%%%%%%%%%%%%%%%%%%%%%%%%%%

In this subsection, let us prove Theorem~\ref{thm.GWP-H21}. Namely, we consider the global well-posedness for \eqref{ZKB}. 
More precisely, we would like to show that \eqref{ZKB} is globally well-posed in $H^{2, 1}(\R^{2})$. In order to extend the local solution constructed in the previous subsection globally in time, we shall derive a priori estimates for the solution to \eqref{ZKB}. Actually, the following three propositions are true. 
%%%%%%%%%%%%%%%%%%%%%%%%%%%%%%%%%%%%%%%%%%%%%%%%%%%%
\begin{prop}\label{prop.apriori-1st}
Let $p\ge1$ be an integer. Assume that $u_{0}\in H^{2, 1}(\R^{2})$ and $\left\|u_{0}\right\|_{H^{1}}$ is sufficiently small. 
Suppose that $u\in C([0, T]; H^{2, 1}(\R^{2}))$ is a solution to \eqref{ZKB} on $[0, T]$ for some $T>0$. 
Then, there exists a positive constant $C_{*}>0$, independent of $T$, such that 
\begin{equation}\label{apriori-1st}
\left\|u(\cdot, \cdot, t)\right\|_{H^{1}}^{2}+\mu \int_{0}^{t} \left\|u_{x}(\cdot, \cdot , \tau)\right\|_{H^{1}}^{2} d\tau \le C_{*}\left\|u_{0}\right\|_{H^{1}}^{2}, \ \ t\in [0, T]. 
\end{equation}
\end{prop}
%%%%%%%%%%%%%%%%%%%%%%%%%%%%%%%%%%%%%%%%%%%%%%%%%%%%
%%%%%%%%%%%%%%%%%%%%%%%%%%%%%%%%%%%%%%%%%%%%%%%%%%%%
\begin{proof}
First, multiplying $u(x, y, t)$ on the both sides of \eqref{ZKB} and integrating over $\R^{2}$, we have  
\begin{equation*}
\frac{d}{dt}\left\|u(t)\right\|_{L^{2}}^{2}=-2\mu \left\|u_{x}(t)\right\|_{L^{2}}^{2}. 
\end{equation*}
Thus, integrating the both sides of the above equation over $[0, t]$, we get the following relation: 
\begin{equation}\label{CL-L2}
\left\|u(t)\right\|_{L^{2}}^{2}+2\mu \int_{0}^{t}\left\|u_{x}(\tau)\right\|_{L^{2}}^{2}d\tau =\left\|u_{0}\right\|_{L^{2}}^{2}, \ \ t\in [0, T].
\end{equation}

Next, we shall evaluate the first order derivative of the solution. First, multiplying $-u_{xx}$ on the both sides of \eqref{ZKB} and integrating over $\R^{2}$, we obtain 
\begin{equation}\label{CL-L2dx}
\frac{1}{2}\frac{d}{dt}\left\|u_{x}(t)\right\|_{L^{2}}^{2}+\mu \left\|u_{xx}(t)\right\|_{L^{2}}^{2}=\beta \int_{\R^{2}}u^{p}u_{x}u_{xx}dxdy. 
\end{equation}
Similarly, multiplying $-u_{yy}$ on the both sides of \eqref{ZKB} and integrating over $\R^{2}$, we get 
\begin{equation}\label{CL-L2dy}
\frac{1}{2}\frac{d}{dt}\left\|u_{y}(t)\right\|_{L^{2}}^{2}+\mu \left\|u_{xy}(t)\right\|_{L^{2}}^{2}=\beta \int_{\R^{2}}u^{p}u_{x}u_{yy}dxdy. 
\end{equation}
Finally, multiplying $-\beta u^{p+1}/(p+1)$ on the both sides of \eqref{ZKB} and integrating over $\R^{2}$, we have 
\begin{equation}\label{CL-L2dp}
-\frac{\beta}{(p+1)(p+2)}\frac{d}{dt}\int_{\R^{2}}u^{p+2}dxdy=\frac{\beta}{p+1}\int_{\R^{2}}u^{p+1}\left(u_{xxx}+u_{yyx}-\mu u_{xx}\right)dxdy. 
\end{equation}
Now, we note that 
\begin{align*}
&u^{p}u_{x}u_{xx}+u^{p}u_{x}u_{yy}+\frac{1}{p+1}\left(u^{p+1}u_{xxx}+u^{p+1}u_{yyx}\right) \\
&=\frac{1}{p+1}\left(u^{p+1}\right)_{x}\left(u_{xx}+u_{yy}\right)+\frac{1}{p+1}u^{p+1}\left(u_{xx}+u_{yy}\right)_{x}
=\frac{1}{p+1}\left\{u^{p+1}\left(u_{xx}+u_{yy}\right)\right\}_{x}. 
\end{align*}
Therefore, combining \eqref{CL-L2dx}, \eqref{CL-L2dy} and \eqref{CL-L2dp}, we can see that 
\begin{align*}
&\frac{d}{dt}\left(\left\|u_{x}(t)\right\|_{L^{2}}^{2}+\left\|u_{y}(t)\right\|_{L^{2}}^{2}-\frac{2\beta}{(p+1)(p+2)}\int_{\R^{2}}u^{p+2}dxdy\right) \\
&+2\mu \left\|u_{xx}(t)\right\|_{L^{2}}^{2}+2\mu \left\|u_{xy}(t)\right\|_{L^{2}}^{2}
=-\frac{2\beta \mu}{p+1}\int_{\R^{2}}u^{p+1}u_{xx}dxdy. 
\end{align*}
Hence, integrating the both sides of the above equation over $[0, t]$, we obtain the following relation: 
\begin{align}
&\left\|u_{x}(t)\right\|_{L^{2}}^{2}+\left\|u_{y}(t)\right\|_{L^{2}}^{2}+2\mu \int_{0}^{t} \left(\left\|u_{xx}(\tau)\right\|_{L^{2}}^{2}+\left\|u_{xy}(\tau)\right\|_{L^{2}}^{2}\right)d\tau \nonumber \\
&=\left\|\p_{x}u_{0}\right\|_{L^{2}}^{2}+\left\|\p_{y}u_{0}\right\|_{L^{2}}^{2}-\frac{2\beta}{(p+1)(p+2)}\int_{\R^{2}}u_{0}^{p+2}dxdy \nonumber \\
&\ \ \ \ +\frac{2\beta}{(p+1)(p+2)}\int_{\R^{2}}u^{p+2}dxdy-\frac{2\beta\mu}{p+1}\int_{0}^{t}\int_{\R^{2}}u^{p+1}u_{xx}dxdyd\tau \nonumber \\
&=:\left\|\p_{x}u_{0}\right\|_{L^{2}}^{2}+\left\|\p_{y}u_{0}\right\|_{L^{2}}^{2}+I_{1}+I_{2}(t)+I_{3}(t). \label{CL-Energy}
\end{align}

In what follows, we shall evaluate $I_{1}$, $I_{2}(t)$ and $I_{3}(t)$ in \eqref{CL-Energy}. To doing that, let us introduce the following function: 
\begin{equation}\label{DEF-M(t)}
M(t):=\sup_{0\le \tau \le t}\left(\left\|u_{x}(\tau)\right\|_{L^{2}}^{2}+\left\|u_{y}(\tau)\right\|_{L^{2}}^{2}\right). 
\end{equation}
First, we shall evaluate $I_{2}(t)$. It follows from Lemma~\ref{lem.est-L2q}, \eqref{CL-L2} and \eqref{DEF-M(t)} that  
\begin{align}
\left|I_{2}(t)\right|
&\le %C\int_{\R^{2}}u^{2p}dxdy+C\int_{\R^{2}}u^{4}dxdy \nonumber 
C\left\|u(t)\right\|_{L^{2p}}^{2p}+C\left\|u(t)\right\|_{L^{4}}^{4}  \nonumber \\
&\le C\left\|u(t)\right\|_{L^{2}}^{2} \left\|u_{x}(t)\right\|_{L^{2}}^{p-1} \left\|u_{y}(t)\right\|_{L^{2}}^{p-1}
+C\left\|u(t)\right\|_{L^{2}}^{2} \left\|u_{x}(t)\right\|_{L^{2}} \left\|u_{y}(t)\right\|_{L^{2}} \nonumber \\
&\le C\left\|u_{0}\right\|_{L^{2}}^{2}\left(\left\|u_{x}(t)\right\|_{L^{2}}^{2} + \left\|u_{y}(t)\right\|_{L^{2}}^{2}\right)^{p-1}
+C\left\|u_{0}\right\|_{L^{2}}^{2}\left(\left\|u_{x}(t)\right\|_{L^{2}}^{2} + \left\|u_{y}(t)\right\|_{L^{2}}^{2}\right) \nonumber \\
&\le C\left\|u_{0}\right\|_{L^{2}}^{2}\left(M(t)^{p-1}+M(t)\right), \ \ t\in [0, T]. \label{est-I2}
\end{align}
From the above \eqref{est-I2}, we obtain the same estimate for $I_{1}$, because $|I_{1}|=|I_{2}(0)|$. 

Next, let us treat $I_{3}(t)$. The evaluation for $I_{3}(t)$ is divided into the following two cases: 

\smallskip
\noindent
\underline{Case~{\rm (i)}: $p=1$.} First, we have from Lemma~\ref{lem.est-Linfty} and \eqref{CL-L2} that 
\begin{align}
\left|I_{3}(t)\right|
&=\frac{2\mu |\beta|}{p+1}\left|\int_{0}^{t}\int_{\R^{2}}u^{2}u_{xx}dxdyd\tau\right| =\frac{4\mu |\beta|}{p+1}\left|\int_{0}^{t}\int_{\R^{2}}uu_{x}^{2} dxdyd\tau\right| \nonumber \\
&\le C\int_{0}^{t}\left\|u(\tau)\right\|_{L^{\infty}}\left\|u_{x}(\tau)\right\|_{L^{2}}^{2}d\tau 
\le C\int_{0}^{t}\left\|u(\tau)\right\|_{L^{\infty}}^{2}\left\|u_{x}(\tau)\right\|_{L^{2}}^{2}d\tau + C\int_{0}^{t}\left\|u_{x}(\tau)\right\|_{L^{2}}^{2}d\tau \nonumber \\
&\le C\int_{0}^{t}\left( \left\|u_{x}(\tau)\right\|_{L^{2}}\left\|u_{y}(\tau)\right\|_{L^{2}} + \left\|u(\tau)\right\|_{L^{2}}\left\|u_{xy}(\tau)\right\|_{L^{2}} \right) \left\|u_{x}(\tau)\right\|_{L^{2}}^{2}d\tau 
+C\left\|u_{0}\right\|_{L^{2}}^{2}\nonumber \\
%&\ \ \ \ +C\sup_{0\le t\le T}\int_{0}^{t}\left\|u_{x}(\tau)\right\|_{L^{2}}^{2}d\tau \nonumber \\
&\le C\int_{0}^{t}\left\|u_{x}(\tau)\right\|_{L^{2}}\left\|u_{y}(\tau)\right\|_{L^{2}} \left\|u_{x}(\tau)\right\|_{L^{2}}^{2}d\tau \nonumber \\
&\ \ \ \ +C\int_{0}^{t}\left\|u(\tau)\right\|_{L^{2}}\left\|u_{xy}(\tau)\right\|_{L^{2}} \left\|u_{x}(\tau)\right\|_{L^{2}}^{2}d\tau+C\left\|u_{0}\right\|_{L^{2}}^{2} \nonumber \\
&=:I_{3.1}(t)+I_{3.2}(t)+C\left\|u_{0}\right\|_{L^{2}}^{2}. \label{I3-devide}
\end{align}
Now, using \eqref{CL-L2} and \eqref{DEF-M(t)} again, we get
\begin{align}
I_{3.1}(t)
&\le C\int_{0}^{t}\left(\left\|u_{x}(\tau)\right\|_{L^{2}}^{2}+\left\|u_{y}(\tau)\right\|_{L^{2}}^{2}\right)\left\|u_{x}(\tau)\right\|_{L^{2}}^{2}d\tau 
%&\le CM(t)\sup_{0\le t\le T}\int_{0}^{t}\left\|u_{x}(\tau)\right\|_{L^{2}}^{2}d\tau
\le C\left\|u_{0}\right\|_{L^{2}}^{2}M(t).  \label{est-I3.1}
\end{align}
Analogously, we obtain 
\begin{align}
I_{3.2}(t)
&\le C\int_{0}^{t}\left\|u(\tau)\right\|_{L^{2}}^{2}\left\|u_{x}(\tau)\right\|_{L^{2}}^{4}d\tau + \mu\int_{0}^{t}\left\|u_{xy}(\tau)\right\|_{L^{2}}^{2}d\tau  \nonumber \\
&\le C\sup_{0\le \tau \le t}\left(\left\|u(\tau)\right\|_{L^{2}}^{2}\left\|u_{x}(\tau)\right\|_{L^{2}}^{2}\right)\left(\sup_{0\le t\le T}\int_{0}^{t}\left\|u_{x}(\tau)\right\|_{L^{2}}^{2}d\tau\right)+ \mu\int_{0}^{t}\left\|u_{xy}(\tau)\right\|_{L^{2}}^{2}d\tau \nonumber \\
&\le C\left\|u_{0}\right\|_{L^{2}}^{4}M(t) + \mu\int_{0}^{t}\left\|u_{xy}(\tau)\right\|_{L^{2}}^{2}d\tau. \label{est-I3.2}
\end{align}
Therefore, summing up \eqref{I3-devide}, \eqref{est-I3.1} and \eqref{est-I3.2}, we can see that 
\begin{equation}
|I_{3}(t)|\le C\left\|u_{0}\right\|_{L^{2}}^{2}\left(1+M(t)\right)
+C\left\|u_{0}\right\|_{L^{2}}^{4}M(t)+ \mu\int_{0}^{t}\left\|u_{xy}(\tau)\right\|_{L^{2}}^{2}d\tau, \ \ t\in [0, T]. \label{est-I3(p=1)}
\end{equation}

\noindent
\underline{Case~{\rm (ii)}: $p\ge2$.}
It follows from Lemma~\ref{lem.est-L2q}, \eqref{CL-L2} and \eqref{DEF-M(t)} that 
\begin{align}
\left|I_{3}(t)\right|
%&\le C\int_{0}^{t}\int_{\R^{2}}|u^{p+1}u_{xx}| dxdyd\tau \nonumber \\
%&\le C\int_{0}^{t}\int_{\R^{2}}u^{2(p+1)}dxdyd\tau + \mu\int_{0}^{t}\int_{\R^{2}}u_{xx}^{2}dxdyd\tau\nonumber \\
&\le C\int_{0}^{t}\left\|u(\tau)\right\|_{L^{2(p+1)}}^{2(p+1)}d\tau + \mu \int_{0}^{t}\left\|u_{xx}(\tau)\right\|_{L^{2}}^{2}d\tau \nonumber \\
&\le C\int_{0}^{t}\left\|u(\tau)\right\|_{L^{2}}^{2}\left\|u_{x}(\tau)\right\|_{L^{2}}^{p}\left\|u_{y}(\tau)\right\|_{L^{2}}^{p}d\tau
+ \mu \int_{0}^{t}\left\|u_{xx}(\tau)\right\|_{L^{2}}^{2}d\tau \nonumber \\
&\le C\left\|u_{0}\right\|_{L^{2}}^{2}\int_{0}^{t}\left(\left\|u_{x}(\tau)\right\|_{L^{2}}^{2}\right)^{\frac{p-2}{2}}\left(\left\|u_{y}(\tau)\right\|_{L^{2}}^{2}\right)^{\frac{p}{2}}\left\|u_{x}(\tau)\right\|_{L^{2}}^{2}d\tau 
+\mu \int_{0}^{t}\left\|u_{xx}(\tau)\right\|_{L^{2}}^{2}d\tau \nonumber \\
%&\le C\left\|u_{0}\right\|_{L^{2}}^{2}M(t)^{p-1}\left(\sup_{0\le t \le T}\int_{0}^{t}\left\|u_{x}(\tau)\right\|_{L^{2}}^{2}d\tau\right) +\mu \int_{0}^{t}\left\|u_{xx}(\tau)\right\|_{L^{2}}^{2}d\tau \nonumber \\
&\le C\left\|u_{0}\right\|_{L^{2}}^{4}M(t)^{p-1}+\mu \int_{0}^{t}\left\|u_{xx}(\tau)\right\|_{L^{2}}^{2}d\tau, \ \ t\in [0, T].  \label{est-I3}
\end{align}

Combining \eqref{CL-Energy}, \eqref{est-I2}, \eqref{est-I3(p=1)} and \eqref{est-I3}, there exists a positive constant $C_{0}>0$ such that the following estimate holds: 
\begin{align}
&\left\|u_{x}(t)\right\|_{L^{2}}^{2}+\left\|u_{y}(t)\right\|_{L^{2}}^{2}+\mu \int_{0}^{t} \left(\left\|u_{xx}(\tau)\right\|_{L^{2}}^{2}+\left\|u_{xy}(\tau)\right\|_{L^{2}}^{2}\right)d\tau  \nonumber \\
&\le C_{0}\biggl\{\left\|\p_{x}u_{0}\right\|_{L^{2}}^{2}+\left\|\p_{y}u_{0}\right\|_{L^{2}}^{2}
+\left\|u_{0}\right\|_{L^{2}}^{2}\left(1+M(t)^{p-1}+M(t)\right)\nonumber \\
&\ \ \ \ \ \ \ \ \ \ +\left\|u_{0}\right\|_{L^{2}}^{4}\left(M(t)^{p-1}+M(t)\right)\biggl\}, \ \ t\in [0, T]. \label{est-energy}
\end{align}
Choosing $\left\|u_{0}\right\|_{L^{2}}$ which satisfies $C_{0}\left\|u_{0}\right\|_{L^{2}}^{2}\le 1/4$ and $C_{0}\left\|u_{0}\right\|_{L^{2}}^{4}\le 1/4$, it follows from \eqref{DEF-M(t)} and Proposition~\ref{prop.regularity} that 
\begin{equation}\label{est-M(t)-pre}
M(t)\le 2C_0\left\{\left\|u_{0}\right\|_{H^{1}}^{2}
+\left(\left\|u_{0}\right\|_{L^{2}}^{2}+\left\|u_{0}\right\|_{L^{2}}^{4}\right)M(t)^{p-1}\right\}, \ \ t\in [0, T]. 
\end{equation}
Therefore, under the smallness assumption of $\left\|u_{0}\right\|_{H^{1}}$, by the standard continuity argument, 
we eventually arrive at the following estimate:
\begin{equation}\label{est-M(t)}
M(t)\le 4C_{0}\left\|u_{0}\right\|_{H^{1}}^{2}, \ \ t\in [0, T]. 
\end{equation}
Summarizing up \eqref{CL-L2}, \eqref{est-energy} and \eqref{est-M(t)}, the desired estimate \eqref{apriori-1st} can be established. 
\end{proof}
\begin{rem}
When $p=1$ or $2$, 
the smallness condition for $u_0(x, y)$ in {\rm Proposition~\ref{prop.apriori-1st}} 
can be reduced to ``$\|u_0\|_{L^{2}}$ is sufficiently small''. 
Indeed, if $p=1$, then a priori estimate \eqref{est-M(t)}
is directly obtained by \eqref{est-M(t)-pre} without the smallness of $\|u_0\|_{H^1}$. 
If $p=2$, then by choosing $\|u_0\|_{L^{2}}$ which satisfies $C_{0}\left\|u_{0}\right\|_{L^{2}}^{2}\le 1/8$ and $C_{0}\left\|u_{0}\right\|_{L^{2}}^{4}\le 1/8$ in {\rm (\ref{est-energy})}, 
we can obtain
\begin{equation*}
M(t)\le 2C_0\left\|u_{0}\right\|_{H^{1}}^{2}, \ \ t\in [0, T] 
\end{equation*}
without the smallness of $\|u_0\|_{H^1}$. 
\end{rem}
%
%%%%%%%%%%%%%%%%%%%%%%%%%%%%%%%%%%%%%%%%%%%%%%%%%%%%

%%%%%%%%%%%%%%%%%%%%%%%%%%%%%%%%%%%%%%%%%%%%%%%%%%%%
\begin{prop}\label{prop.apriori-2nd}
Let $p\ge1$ be an integer. Assume that $u_{0}\in H^{2, 1}(\R^{2})$ and $\left\|u_{0}\right\|_{H^{1}}+\left\|\p_{x}u_{0}\right\|_{H^{1}}$ is sufficiently small. 
Suppose that $u\in C([0, T]; H^{2, 1}(\R^{2}))$ is a solution to \eqref{ZKB} on $[0, T]$ for some $T>0$. 
Then, there exists a positive constant $C_{\#}>0$, independent of $T$, such that 
\begin{align}
&\left\|u_{xx}(\cdot, \cdot, t)\right\|_{L^{2}}^{2}+\left\|u_{xy}(\cdot, \cdot, t)\right\|_{L^{2}}^{2}+\mu \int_{0}^{t} \left(\left\|u_{xxx}(\cdot, \cdot , \tau)\right\|_{L^{2}}^{2} +\left\|u_{xxy}(\cdot, \cdot , \tau)\right\|_{L^{2}}^{2}\right)d\tau  \nonumber \\
&\le C_{\#}\left(\left\|u_{0}\right\|_{H^{1}}^{2}+\left\|\p_{x}u_{0}\right\|_{H^{1}}^{2}\right), \ \ t\in [0, T].  \label{apriori-2nd}
\end{align}
\end{prop}
%%%%%%%%%%%%%%%%%%%%%%%%%%%%%%%%%%%%%%%%%%%%%%%%%%%%
%%%%%%%%%%%%%%%%%%%%%%%%%%%%%%%%%%%%%%%%%%%%%%%%%%%%
\begin{proof}
First, multiplying $u_{xxxx}$ on the both sides of \eqref{ZKB} and integrating over $\R^{2}$, we have 
\begin{equation}\label{CL-L2dxx}
\frac{1}{2}\frac{d}{dt}\left\|u_{xx}(t)\right\|_{L^{2}}^{2}+\mu \left\|u_{xxx}(t)\right\|_{L^{2}}^{2}=-\beta \int_{\R^{2}}u^{p}u_{x}u_{xxxx}dxdy. 
\end{equation}
Also, multiplying $u_{xxyy}$ on the both sides of \eqref{ZKB} and integrating over $\R^{2}$, it follows that  
\begin{equation}\label{CL-L2dxy}
\frac{1}{2}\frac{d}{dt}\left\|u_{xy}(t)\right\|_{L^{2}}^{2}+\mu \left\|u_{xxy}(t)\right\|_{L^{2}}^{2}=-\beta \int_{\R^{2}}u^{p}u_{x}u_{xxyy}dxdy. 
\end{equation}
Therefore, integrating the both sides of \eqref{CL-L2dxx} and \eqref{CL-L2dxy}, respectively, we are able to see that the following relations hold: 
\begin{align}
&\left\|u_{xx}(t)\right\|_{L^{2}}^{2}+2\mu \int_{0}^{t}\left\|u_{xxx}(\tau)\right\|_{L^{2}}^{2}d\tau=\left\|\p_{x}^{2}u_{0}\right\|_{L^{2}}^{2}-2\beta \int_{0}^{t}\int_{\R^{2}}u^{p}u_{x}u_{xxxx}dxdyd\tau,  \label{CL-high-dxx}\\
&\left\|u_{xy}(t)\right\|_{L^{2}}^{2}+2\mu \int_{0}^{t}\left\|u_{xxy}(\tau)\right\|_{L^{2}}^{2}d\tau=\left\|\p_{x}\p_{y}u_{0}\right\|_{L^{2}}^{2}-2\beta \int_{0}^{t}\int_{\R^{2}}u^{p}u_{x}u_{xxyy}dxdyd\tau. \label{CL-high-dxy}
\end{align}

In what follows, we shall evaluate the right hand sides of \eqref{CL-high-dxx} and \eqref{CL-high-dxy}. 
In order to treat them, let us define the following function:  
\begin{equation}\label{DEF-N(t)}
N(t):=\sup_{0\le \tau \le t}\left(\left\|u_{xx}(\tau)\right\|_{L^{2}}^{2}+\left\|u_{xy}(\tau)\right\|_{L^{2}}^{2}\right). 
\end{equation}
Here and later, $\left\|u_{0}\right\|_{H^{1}}+\left\|\p_{x}u_{0}\right\|_{H^{1}}$ is assumed to be small as $\left\|u_{0}\right\|_{H^{1}}+\left\|\p_{x}u_{0}\right\|_{H^{1}}\le 1$. 
First, we would like to consider \eqref{CL-high-dxx}. By making the integration by parts, we have 
\begin{align}
&2|\beta|\left|\int_{0}^{t}\int_{\R^{2}}u^{p}u_{x}u_{xxxx}dxdyd\tau\right| 
=2|\beta|\left|\int_{0}^{t}\int_{\R^{2}}\p_{x}\left(u^{p}u_{x}\right)u_{xxx}dxdyd\tau\right|  \nonumber \\
&\le 2p|\beta|\int_{0}^{t}\int_{\R^{2}}\left|u^{p-1}u_{x}^{2}u_{xxx}\right|dxdyd\tau+2|\beta|\int_{0}^{t}\int_{\R^{2}}\left|u^{p}u_{xx}u_{xxx}\right|dxdyd\tau \nonumber \\
&\le C\int_{0}^{t}\int_{\R^{2}}u^{2(p-1)}u_{x}^{4}dxdyd\tau 
+ C\int_{0}^{t}\int_{\R^{2}}u^{2p}u_{xx}^{2}dxdyd\tau + \mu \int_{0}^{t}\int_{\R^{2}}u_{xxx}^{2}dxdyd\tau \nonumber \\
&=:J_{1}(t)+J_{2}(t) + \mu \int_{0}^{t}\left\|u_{xxx}(\tau)\right\|_{L^{2}}^{2}d\tau.  \label{split-J1-J2}
\end{align}
In the following, we shall evaluate $J_{1}(t)$ and $J_{2}(t)$. Before doing that, let us prepare an estimate. 
From Lemma~\ref{lem.est-Linfty}, \eqref{apriori-1st} and \eqref{DEF-N(t)}, we can see that 
\begin{align}
\left\|u(\tau)\right\|_{L^{\infty}}^{2q}
&\le C \left(\left\|u_{x}(\tau)\right\|_{L^{2}}\left\|u_{y}(\tau)\right\|_{L^{2}}+\left\|u(\tau)\right\|_{L^{2}}\left\|u_{xy}(\tau)\right\|_{L^{2}}\right)^{q} \nonumber \\
&\le C \left\{\left(\left\|u_{x}(\tau)\right\|_{L^{2}}^{2}+\left\|u_{y}(\tau)\right\|_{L^{2}}^{2}+\left\|u(\tau)\right\|_{L^{2}}^{2}\right)^{q}+\left\|u_{xy}(\tau)\right\|_{L^{2}}^{2q}\right\} \nonumber \\
&\le C\left(\left\|u_{0}\right\|_{H^{1}}^{2q}+N(t)^{q}\right), \ \ 0\le \tau \le t, \ q\in \mathbb{N}\cup \{0\}. \label{Linf-add}
\end{align}

Now, let us evaluate $J_{1}(t)$. It follows from \eqref{Linf-add}, Lemma~\ref{lem.est-L2q}, \eqref{DEF-N(t)} and \eqref{CL-L2} that 
\begin{align}
J_{1}(t)
&\le C\int_{0}^{t}\left\|u(\tau)\right\|_{L^{\infty}}^{2(p-1)}\left\|u_{x}(\tau)\right\|_{L^{4}}^{4}d\tau \nonumber \\
&\le C\left(\left\|u_{0}\right\|_{H^{1}}^{2(p-1)}+N(t)^{p-1}\right)\int_{0}^{t}\left\|u_{x}(\tau)\right\|_{L^{2}}^{2}\left\|u_{xx}(\tau)\right\|_{L^{2}}\left\|u_{xy}(\tau)\right\|_{L^{2}}d\tau \nonumber \\
&\le C\left(\left\|u_{0}\right\|_{H^{1}}^{2(p-1)}+N(t)^{p-1}\right)\int_{0}^{t}\left(\left\|u_{xx}(\tau)\right\|_{L^{2}}^{2}+\left\|u_{xy}(\tau)\right\|_{L^{2}}^{2}\right)\left\|u_{x}(\tau)\right\|_{L^{2}}^{2}d\tau \nonumber \\
%&\le C\left(\left\|u_{0}\right\|_{H^{1}}^{2(p-1)}+N(t)^{p-1}\right)N(t)\left(\sup_{0\le t\le T}\int_{0}^{t}\left\|u_{x}(\tau)\right\|_{L^{2}}^{2}d\tau\right) \nonumber \\
&\le C\left(\left\|u_{0}\right\|_{H^{1}}^{2(p-1)}+N(t)^{p-1}\right)N(t)\left\|u_{0}\right\|_{L^{2}}^{2} 
\le C\left\|u_{0}\right\|_{H^{1}}^{2}\left(N(t)+N(t)^{p}\right), \ \ t\in [0, T], \label{est-J1}
\end{align}
where we used $\left\|u_{0}\right\|_{L^2}\le \left\|u_{0}\right\|_{H^1}\le 1$ in the last inequality. 

For $J_{2}(t)$, we similarly obtain from \eqref{Linf-add} and \eqref{apriori-1st} that 
\begin{align}
J_{2}(t)
&\le C\int_{0}^{t}\left\|u(\tau)\right\|_{L^{\infty}}^{2p}\left\|u_{xx}(\tau)\right\|_{L^{2}}^{2}d\tau \nonumber \\
%&\le C \left(\left\|u_{0}\right\|_{H^{1}}^{2p}+N(t)^{p}\right)\left(\sup_{0\le t\le T}\int_{0}^{t}\left\|u_{xx}(\tau)\right\|_{L^{2}}^{2}d\tau\right) \nonumber \\
&\le C \left(\left\|u_{0}\right\|_{H^{1}}^{2p}+N(t)^{p}\right)\left\|u_{0}\right\|_{H^{1}}^{2}
\le C\left\|u_{0}\right\|_{H^{1}}^{2}\left(1+N(t)^{p}\right), \ \ t\in [0, T], \label{est-J2}
\end{align}
where we used $\left\|u_{0}\right\|_{H^1}\le 1$ in the last inequality. 

Therefore, summarizing up \eqref{CL-high-dxx}, \eqref{split-J1-J2}, \eqref{est-J1} and \eqref{est-J2}, we are able to see that 
\begin{align}
&\left\|u_{xx}(t)\right\|_{L^{2}}^{2}+\mu \int_{0}^{t}\left\|u_{xxx}(\tau)\right\|_{L^{2}}^{2}d\tau \nonumber \\
&\le C\left(\left\|u_{0}\right\|_{H^{1}}^{2}+\left\|\p_{x}^{2}u_{0}\right\|_{L^{2}}^{2}\right) +C\left\|u_{0}\right\|_{H^{1}}^{2}\left(N(t)+N(t)^{p}\right), \ \ t\in [0, T]. \label{est-high-dxx}
\end{align}

Next, let us evaluate the right hand side of \eqref{CL-high-dxy}. Making the integration by parts, we get
\begin{align}
&2|\beta|\left|\int_{0}^{t}\int_{\R^{2}}u^{p}u_{x}u_{xxyy}dxdyd\tau\right| 
=2|\beta|\left|\int_{0}^{t}\int_{\R^{2}}\p_{y}\left(u^{p}u_{x}\right)u_{xxy}dxdyd\tau\right|  \nonumber \\
&\le 2p|\beta|\int_{0}^{t}\int_{\R^{2}}\left|u^{p-1}u_{x}u_{y}u_{xxy}\right|dxdyd\tau+2|\beta|\left|\int_{0}^{t}\int_{\R^{2}}u^{p}u_{xy}u_{xxy}dxdyd\tau\right| \nonumber \\
&\le C\int_{0}^{t}\int_{\R^{2}}u^{2(p-1)}u_{x}^{2}u_{y}^{2}dxdyd\tau 
+C\int_{0}^{t}\int_{\R^{2}}u^{2p}u_{xy}^{2}dxdyd\tau 
+ \frac{\mu}{2} \int_{0}^{t}\int_{\R^{2}}u_{xxy}^{2}dxdyd\tau \nonumber \\
&=:K_{1}(t)+K_{2}(t)+\frac{\mu}{2} \int_{0}^{t}\left\|u_{xxy}(\tau)\right\|_{L^{2}}^{2}dxdyd\tau.  \label{split-K1-K2}
\end{align}
Now, we shall evaluate $K_{1}(t)$. It follows from Lemma~\ref{lem.est-Linfty}, \eqref{apriori-1st}, \eqref{Linf-add} and \eqref{CL-L2} that 
\begin{align}
K_{1}(t)
&\le C\int_{0}^{t} \left\|u(\tau)\right\|_{L^{\infty}}^{2(p-1)}\left\|u_{x}(\tau)\right\|_{L^{\infty}}^{2}\left\|u_{y}(\tau)\right\|_{L^{2}}^{2}d\tau \nonumber \\
&\le C\left(\sup_{0\le \tau \le t}\left\|u_{y}(\tau)\right\|_{L^{2}}^{2}\right)
\int_{0}^{t} \left\|u(\tau)\right\|_{L^{\infty}}^{2(p-1)}
\left(\left\|u_{xx}(\tau)\right\|_{L^{2}}\left\|u_{xy}(\tau)\right\|_{L^{2}}+\left\|u_{x}(\tau)\right\|_{L^{2}}\left\|u_{xxy}(\tau)\right\|_{L^{2}}\right)d\tau \nonumber \\
&\le C\left\|u_{0}\right\|_{H^{1}}^{2}
\int_{0}^{t} \left\|u(\tau)\right\|_{L^{\infty}}^{2(p-1)}\left(\left\|u_{xx}(\tau)\right\|_{L^{2}}^{2}+\left\|u_{xy}(\tau)\right\|_{L^{2}}^{2}\right)d\tau \nonumber \\
&\ \ \ +C\left\|u_{0}\right\|_{H^{1}}^{4}\int_{0}^{t} \left\|u(\tau)\right\|_{L^{\infty}}^{4(p-1)}\left\|u_{x}(\tau)\right\|_{L^{2}}^{2}d\tau
+\frac{\mu}{2}\int_{0}^{t}\left\|u_{xxy}(\tau)\right\|_{L^{2}}^{2}d\tau \nonumber \\
&\le C\left\|u_{0}\right\|_{H^{1}}^{4}\left(\left\|u_{0}\right\|_{H^{1}}^{2(p-1)}+N(t)^{p-1}\right)
+C\left\|u_{0}\right\|_{H^{1}}^{4}\left\|u_{0}\right\|_{L^{2}}^{2}\left(\left\|u_{0}\right\|_{H^{1}}^{4(p-1)}+N(t)^{2(p-1)}\right) \nonumber \\
&\ \ \ +\frac{\mu}{2}\int_{0}^{t}\left\|u_{xxy}(\tau)\right\|_{L^{2}}^{2}d\tau \nonumber \\
&\le C\left\|u_{0}\right\|_{H^{1}}^{2}\left(1 + N(t)^{p-1}+N(t)^{2(p-1)}\right)+\frac{\mu}{2}\int_{0}^{t}\left\|u_{xxy}(\tau)\right\|_{L^{2}}^{2}d\tau, \ \ t\in [0, T], \label{est-K1}
\end{align}
where we used $\left\|u_{0}\right\|_{L^2}\le\left\|u_{0}\right\|_{H^1}\le 1$ in the last inequality. 

Let us treat $K_{2}(t)$. Similarly as before, from \eqref{Linf-add} and \eqref{apriori-1st}, we have 
\begin{align}
K_{2}(t)
%&\le C\int_{0}^{t}\int_{\R^{2}}u^{2p}u_{xy}^{2}dxdyd\tau + \frac{\mu}{2}\int_{0}^{t}\int_{\R^{2}}u_{xxy}^{2}dxdyd\tau \nonumber \\
&\le C \int_{0}^{t}\left\|u(\tau)\right\|_{L^{\infty}}^{2p}\left\|u_{xy}(\tau)\right\|_{L^{2}}^{2}d\tau \nonumber \\
%+\frac{\mu}{2}\int_{0}^{t}\left\|u_{xxy}(\tau)\right\|_{L^{2}}^{2}d\tau \nonumber \\
%&\le C\int_{0}^{t}\left(\left\|u_{x}(\tau)\right\|_{L^{2}}\left\|u_{y}(\tau)\right\|_{L^{2}}+\left\|u(\tau)\right\|_{L^{2}}\left\|u_{xy}(\tau)\right\|_{L^{2}}\right)^{p}\left\|u_{xy}(\tau)\right\|_{L^{2}}^{2}d\tau \nonumber \\
%&\ \ \ \ +\frac{\mu}{2}\int_{0}^{t}\left\|u_{xxy}(\tau)\right\|_{L^{2}}^{2}d\tau \nonumber \\
%&\le C\int_{0}^{t}\left\{\left(\left\|u_{x}(\tau)\right\|_{L^{2}}^{2}+\left\|u_{y}(\tau)\right\|_{L^{2}}^{2}+\left\|u(\tau)\right\|_{L^{2}}^{2}\right)^{p}+\left\|u_{xy}(\tau)\right\|_{L^{2}}^{2p}\right\}\left\|u_{xy}(\tau)\right\|_{L^{2}}^{2}d\tau \nonumber \\
%&\ \ \ \ +\frac{\mu}{2}\int_{0}^{t}\left\|u_{xxy}(\tau)\right\|_{L^{2}}^{2}d\tau \nonumber \\
%&\le C\left(\left\|u_{0}\right\|_{H^{1}}^{2p} + N(t)^{p} \right)\left(\sup_{0\le t \le T}\int_{0}^{t}\left\|u_{xy}(\tau)\right\|_{L^{2}}^{2} d\tau\right)
%+\frac{\mu}{2}\int_{0}^{t}\left\|u_{xxy}(\tau)\right\|_{L^{2}}^{2}d\tau \nonumber \\
&\le C\left(\left\|u_{0}\right\|_{H^{1}}^{2p} + N(t)^{p} \right)\left\|u_{0}\right\|_{H^{1}}^{2} 
%+\frac{\mu}{2}\int_{0}^{t}\left\|u_{xxy}(\tau)\right\|_{L^{2}}^{2}d\tau \nonumber \\
\le C\left\|u_{0}\right\|_{H^{1}}^{2}\left(1 + N(t)^{p}\right),
%+\frac{\mu}{2}\int_{0}^{t}\left\|u_{xxy}(\tau)\right\|_{L^{2}}^{2}d\tau, 
\ \ t\in [0, T], \label{est-K2}
\end{align}
where we used $\left\|u_{0}\right\|_{H^1}\le 1$ in the last inequality.

Therefore, combining \eqref{CL-high-dxy} and \eqref{split-K1-K2} through \eqref{est-K2}, we arrive at 
\begin{align}
&\left\|u_{xy}(t)\right\|_{L^{2}}^{2}+\mu \int_{0}^{t}\left\|u_{xxy}(\tau)\right\|_{L^{2}}^{2}d\tau \nonumber \\
&\le C\left(\left\|u_{0}\right\|_{H^{1}}^{2}+\left\|\p_{x}\p_{y}u_{0}\right\|_{L^{2}}^{2}\right) +C\left\|u_{0}\right\|_{H^{1}}^{2}\left(N(t)^{p-1}+ N(t)^{p}+N(t)^{2(p-1)}\right), \ \ t\in [0, T]. \label{est-high-dxy}
\end{align}

Combining \eqref{est-high-dxx} and \eqref{est-high-dxy}, there exists a positive constant $C_{1}>0$ such that the following estimate holds: 
\begin{align}
&\left\|u_{xx}(t)\right\|_{L^{2}}^{2}+\left\|u_{xy}(t)\right\|_{L^{2}}^{2}+\mu \int_{0}^{t} \left(\left\|u_{xxx}(\tau)\right\|_{L^{2}}^{2} +\left\|u_{xxy}(\tau)\right\|_{L^{2}}^{2}\right)d\tau  \label{est-high-pre} \\
&\le C_{1}\left\{\left\|u_{0}\right\|_{H^{1}}^{2}+\left\|\p_{x}u_{0}\right\|_{H^{1}}^{2} + \left\|u_{0}\right\|_{H^{1}}^{2}\left(N(t)+N(t)^{p-1} + N(t)^{p}+N(t)^{2(p-1)}\right)\right\}, \ \ t\in [0, T]. \nonumber 
\end{align}
Taking $\left\|u_{0}\right\|_{H^{1}}$ small as $C_{1}\left\|u_{0}\right\|_{H^{1}}^{2}\le 1/2$, it follows from \eqref{DEF-N(t)} and Proposition~\ref{prop.regularity} that 
\begin{equation}\label{est-N(t)-pre}
N(t)\le 2C_{1}\left(\left\|u_{0}\right\|_{H^{1}}^{2}+\left\|\p_{x}u_{0}\right\|_{H^{1}}^{2}\right)+N(t)^{p-1}+N(t)^{p}+N(t)^{2(p-1)}, \ \ t\in [0, T]. 
\end{equation}
Finally, under the smallness assumption of $\left\|u_{0}\right\|_{H^{1}}+\left\|\p_{x}u_{0}\right\|_{H^{1}}$, in the same way to get \eqref{est-M(t)}, applying the continuity argument again, we arrive at the following estimate: 
\begin{equation}\label{est-N(t)}
N(t)\le 4C_{1}\left(\left\|u_{0}\right\|_{H^{1}}^{2}+\left\|\p_{x}u_{0}\right\|_{H^{1}}^{2}\right), \ \ t\in [0, T]. 
\end{equation}
By virtue of \eqref{est-high-pre} and \eqref{est-N(t)}, we eventually obtain the desired estimate \eqref{apriori-2nd}. 
\end{proof}
\begin{rem}
When $p=1$, 
the smallness condition for $u_0(x, y)$ in {\rm Proposition~\ref{prop.apriori-2nd}} 
can be reduced to ``$\|u_0\|_{H^{1}}$ is sufficiently small''. 
Indeed, if $p=1$, then by choosing $\|u_0\|_{H^{1}}$ which satisfies $C_{1}\left\|u_{0}\right\|_{H^{1}}^{2}\le 1/4$ in \eqref{est-high-pre}, 
we can obtain a priori estimate \eqref{est-N(t)} without the smallness of $\|\partial_xu_0\|_{H^1}$. 
\end{rem}
%%%%%%%%%%%%%%%%%%%%%%%%%%%%%%%%%%%%%%%%%%%%%%%%%%%%

%%%%%%%%%%%%%%%%%%%%%%%%%%%%%%%%%%%%%%%%%%%%%%%%%%%%
\begin{prop}\label{prop.apriori-xxy}
Let $p\ge1$ be an integer. Assume that $u_{0}\in H^{2, 1}(\R^{2})$ and $\left\|u_{0}\right\|_{H^{1}}+\left\|\p_{x}u_{0}\right\|_{H^{1}}$ is sufficiently small. 
Suppose that $u\in C([0, T]; H^{2, 1}(\R^{2}))$ is a solution to \eqref{ZKB} on $[0, T]$ for some $T>0$. 
Then, there exists a positive constant $C_{\dag}>0$, independent of $T$, such that 
\begin{align}
\left\|u_{xxy}(\cdot, t)\right\|_{L^{2}}^{2} +\mu \int_{0}^{t} \left\|u_{xxxy}(\cdot , \tau)\right\|_{L^{2}}^{2} d\tau 
\le C_{\dag}\left\|u_{0}\right\|_{H^{2, 1}}^{2}, \ \ t\in [0, T].  \label{apriori-xxy}
\end{align}
\end{prop}
%%%%%%%%%%%%%%%%%%%%%%%%%%%%%%%%%%%%%%%%%%%%%%%%%%%%
%%%%%%%%%%%%%%%%%%%%%%%%%%%%%%%%%%%%%%%%%%%%%%%%%%%%
\begin{proof}
Multiplying $-u_{xxxxyy}$ on the both sides of \eqref{ZKB} and integrating over $\R^{2}$, we have 
\begin{equation}\label{CL-L2dxxy}
\frac{1}{2}\frac{d}{dt}\left\|u_{xxy}(t)\right\|_{L^{2}}^{2}+\mu \left\|u_{xxxy}(t)\right\|_{L^{2}}^{2}=\beta \int_{\R^{2}}u^{p}u_{x}u_{xxxxyy}dxdy. 
\end{equation}
Then, integrating the both sides of \eqref{CL-L2dxxy}, it holds that 
\begin{align}
\left\|u_{xxy}(t)\right\|_{L^{2}}^{2}+2\mu \int_{0}^{t}\left\|u_{xxxy}(\tau)\right\|_{L^{2}}^{2}d\tau=\left\|\p_{x}^{2}\p_{y}u_{0}\right\|_{L^{2}}^{2}+2\beta\int_{0}^{t}\int_{\R^{2}}u^{p}u_{x}u_{xxxxyy}dxdyd\tau.  \label{CL-high-dxxy}
\end{align}
By making the integration by parts, we obtain 
\begin{align}
&2|\beta|\left|\int_{0}^{t}\int_{\R^{2}}u^{p}u_{x}u_{xxxxyy}dxdyd\tau\right| \nonumber \\
%&=2|\beta|\left|\int_{0}^{t}\int_{\R^{2}}\p_{x}\p_{y}\left(u^{p}u_{x}\right)u_{xxxy}dxdyd\tau\right| 
%=2|\beta|\left|\int_{0}^{t}\int_{\R^{2}}\p_{x}\left(pu^{p-1}u_{y}u_{x}+u^{p}u_{xy}\right)u_{xxxy}dxdyd\tau\right| \nonumber\\
&=2|\beta|\left|\int_{0}^{t}\int_{\R^{2}}\left\{p(p-1)u^{p-2}u_{x}^{2}u_{y}+pu^{p-1}u_{xx}u_{y}+2pu^{p-1}u_{x}u_{xy}+u^{p}u_{xxy}\right\}u_{xxxy}dxdyd\tau\right| \nonumber \\
&\le C\int_{0}^{t}\int_{\R^{2}}u^{2(p-2)}u_{x}^{4}u_{y}^{2}dxdyd\tau 
+C\int_{0}^{t}\int_{\R^{2}}u^{2(p-1)}u_{xx}^{2}u_{y}^{2}dxdyd\tau \nonumber\\
&\ \ \ \ +C\int_{0}^{t}\int_{\R^{2}}u^{2(p-1)}u_{x}^{2}u_{xy}^{2}dxdyd\tau 
+C\int_{0}^{t}\int_{\R^{2}}u^{2p}u_{xxy}^{2}dxdyd\tau 
+\frac{\mu}{2} \int_{0}^{t}\left\|u_{xxxy}(\tau)\right\|_{L^{2}}^{2}d\tau \nonumber\\
&=:H_{1}(t)+H_{2}(t)+H_{3}(t)+H_{4}(t)+\frac{\mu}{2} \int_{0}^{t}\left\|u_{xxxy}(\tau)\right\|_{L^{2}}^{2}d\tau. \label{split-H(t)}
\end{align}
Here, we note that the first term $H_{1}(t)$ does not appear when $p=1$. 

In order to evaluate $H_{1}(t)$, $H_{2}(t)$, $H_{3}(t)$ and $H_{4}(t)$, we prepare an estimate. 
Here and later, $\left\|u_{0}\right\|_{H^{2, 1}}$ is assumed to be sufficiently small 
as $\left\|u_{0}\right\|_{H^{2, 1}}\le 1$. 
It follows from Lemma~\ref{lem.est-Linfty}, \eqref{apriori-1st} and \eqref{apriori-2nd} that 
\begin{align}
\left\|u(t)\right\|_{L^{\infty}}^{2}
&\le \left\|u_{x}(t)\right\|_{L^{2}}^{2}+\left\|u_{y}(t)\right\|_{L^{2}}^{2}+\left\|u(t)\right\|_{L^{2}}^{2}+\left\|u_{xy}(t)\right\|_{L^{2}}^{2} \nonumber \\
&\le C\left(\left\|u_{0}\right\|_{H^{1}}^{2}+\left\|\p_{x}u_{0}\right\|_{H^{1}}^{2}\right)
\le C\left\|u_{0}\right\|_{H^{2, 1}}^{2}, \ \ t\in [0, T].  \label{u-est-Linfty}
\end{align}

Now, let us evaluate $H_{1}(t)$, $H_{2}(t)$, $H_{3}(t)$ and $H_{4}(t)$. For $H_{1}(t)$, by using Lemma~\ref{lem.est-Linfty}, \eqref{u-est-Linfty}, \eqref{apriori-1st} and \eqref{apriori-2nd}, we can see that 
\begin{align}
H_{1}(t)
&\le C\int_{0}^{t}\left\|u(\tau)\right\|_{L^{\infty}}^{2(p-2)}\left\|u_{x}(\tau)\right\|_{L^{\infty}}^{4}\left\|u_{y}(\tau)\right\|_{L^{2}}^{2}d\tau  \nonumber\\
&\le C\left(\sup_{0\le \tau \le t}\left\|u(\tau)\right\|_{L^{\infty}}\right)^{2(p-2)} \left(\sup_{0\le \tau \le t}\left\|u_{y}(\tau)\right\|_{L^{2}}^{2}\right)  \nonumber \\
&\ \ \ \ \times \int_{0}^{t}\left(\left\|u_{xx}(\tau)\right\|_{L^{2}}^{2}\left\|u_{xy}(\tau)\right\|_{L^{2}}^{2}+\left\|u_{x}(\tau)\right\|_{L^{2}}^{2}\left\|u_{xxy}(\tau)\right\|_{L^{2}}^{2}\right) d\tau  \nonumber\\
&\le C\left\|u_{0}\right\|_{H^{2, 1}}^{2(p-2)}\left\|u_{0}\right\|_{H^{1}}^{2}\sup_{0\le \tau \le t}\left(\left\|u_{xx}(\tau)\right\|_{L^{2}}^{2}+\left\|u_{x}(\tau)\right\|_{L^{2}}^{2}\right)  \nonumber\\
&\ \ \ \ \times \sup_{0\le t\le T}\int_{0}^{t}\left( \left\|u_{xy}(\tau)\right\|_{L^{2}}^{2} + \left\|u_{xxy}(\tau)\right\|_{L^{2}}^{2}\right)d\tau \nonumber \\
&\le C\left\|u_{0}\right\|_{H^{2, 1}}^{2(p-2)}\left\|u_{0}\right\|_{H^{1}}^{2}\left(\left\|u_{0}\right\|_{H^{1}}^{2}+\left\|\p_{x}u_{0}\right\|_{H^{1}}^{2}\right)^{2} 
\le C\left\|u_{0}\right\|_{H^{2, 1}}^{2p+2}, \ \ t\in [0, T]. \label{est-H1}
\end{align}
We similarly obtain the estimate for $H_{2}(t)$ as follows: 
\begin{align}
H_{2}(t)
&\le C\int_{0}^{t}\left\|u(\tau)\right\|_{L^{\infty}}^{2(p-1)}\left\|u_{xx}(\tau)\right\|_{L^{\infty}}^{2}\left\|u_{y}(\tau)\right\|_{L^{2}}^{2}d\tau  \nonumber\\
&\le C\left(\sup_{0\le \tau \le t}\left\|u(\tau)\right\|_{L^{\infty}}\right)^{2(p-1)} \left(\sup_{0\le \tau \le t}\left\|u_{y}(\tau)\right\|_{L^{2}}^{2}\right)  \nonumber\\
&\ \ \ \ \times \int_{0}^{t}\left(\left\|u_{xxx}(\tau)\right\|_{L^{2}}\left\|u_{xxy}(\tau)\right\|_{L^{2}}+\left\|u_{xx}(\tau)\right\|_{L^{2}}\left\|u_{xxxy}(\tau)\right\|_{L^{2}}\right) d\tau  \nonumber\\
&\le C\left\|u_{0}\right\|_{H^{2, 1}}^{2(p-1)} \left\|u_{0}\right\|_{H^{1}}^{2}\sup_{0\le t\le T}\int_{0}^{t}\left( \left\|u_{xxx}(\tau)\right\|_{L^{2}}^{2} + \left\|u_{xxy}(\tau)\right\|_{L^{2}}^{2} + \left\|u_{xx}(\tau)\right\|_{L^{2}}^{2}\right)d\tau \nonumber \\
&\ \ \ \ +\frac{\mu}{2} \int_{0}^{t}\left\|u_{xxxy}(\tau)\right\|_{L^{2}}^{2}d\tau \nonumber\\
&\le C\left\|u_{0}\right\|_{H^{2, 1}}^{2(p-1)} \left\|u_{0}\right\|_{H^{1}}^{2}\left(\left\|u_{0}\right\|_{H^{1}}^{2}+\left\|\p_{x}u_{0}\right\|_{H^{1}}^{2}\right)+\frac{\mu}{2} \int_{0}^{t}\left\|u_{xxxy}(\tau)\right\|_{L^{2}}^{2}d\tau \nonumber \\
&\le C\left\|u_{0}\right\|_{H^{2, 1}}^{2p+2}+\frac{\mu}{2} \int_{0}^{t}\left\|u_{xxxy}(\tau)\right\|_{L^{2}}^{2}d\tau, \ \ t\in [0, T]. \label{est-H2}
\end{align}
For $H_{3}(t)$, we can analogously see that 
\begin{align}
H_{3}(t)
&\le C\int_{0}^{t}\left\|u(\tau)\right\|_{L^{\infty}}^{2(p-1)}\left\|u_{x}(\tau)\right\|_{L^{\infty}}^{2}\left\|u_{xy}(\tau)\right\|_{L^{2}}^{2}d\tau  \nonumber\\
&\le C\left(\sup_{0\le \tau \le t}\left\|u(\tau)\right\|_{L^{\infty}}\right)^{2(p-1)}
\left(\sup_{0\le \tau \le t}\left\|u_{xy}(\tau)\right\|_{L^{2}}^{2}\right)  \nonumber\\
&\ \ \ \ \times \sup_{0\le t \le T}\int_{0}^{t}\left(\left\|u_{xx}(\tau)\right\|_{L^{2}}^{2}+\left\|u_{xy}(\tau)\right\|_{L^{2}}^{2}+\left\|u_{x}(\tau)\right\|_{L^{2}}^{2}+\left\|u_{xxy}(\tau)\right\|_{L^{2}}^{2}\right)d\tau  \nonumber\\
&\le C\left\|u_{0}\right\|_{H^{2, 1}}^{2(p-1)}\left(\left\|u_{0}\right\|_{H^{1}}^{2}+\left\|\p_{x}u_{0}\right\|_{H^{1}}^{2}\right)^{2} 
\le C\left\|u_{0}\right\|_{H^{2, 1}}^{2p+2}, \ \ t\in [0, T]. \label{est-H3}
\end{align}
Finally, we can easily have the estimate for $H_{4}(t)$ as 
\begin{align}
H_{4}(t)
&\le C\int_{0}^{t}\left\|u(\tau)\right\|_{L^{\infty}}^{2p}\left\|u_{xxy}(\tau)\right\|_{L^{2}}^{2}d\tau \nonumber \\
%&\le C\left(\sup_{0\le \tau \le t}\left\|u(\tau)\right\|_{L^{\infty}}\right)^{2p}\sup_{0\le t\le T}\int_{0}^{t}\left\|u_{xxy}(\tau)\right\|_{L^{2}}^{2}d\tau \nonumber \\
&\le C\left\|u_{0}\right\|_{H^{2, 1}}^{2p} \left(\left\|u_{0}\right\|_{H^{1}}^{2}+\left\|\p_{x}u_{0}\right\|_{H^{1}}^{2}\right)
\le C\left\|u_{0}\right\|_{H^{2, 1}}^{2p+2}, \ \ t\in [0, T]. \label{est-H4}
\end{align}
Moreover, we have $\left\|u_{0}\right\|_{H^{2, 1}}^{2p+2}\le \left\|u_{0}\right\|_{H^{2, 1}}^{2}$ 
by using the smallness condition $\left\|u_{0}\right\|_{H^{2, 1}}\le 1$. 
Therefore, summarizing up \eqref{CL-high-dxxy}, \eqref{split-H(t)} and \eqref{est-H1} through \eqref{est-H4}, the desired estimate \eqref{apriori-xxy} can be established. 
\end{proof}
%%%%%%%%%%%%%%%%%%%%%%%%%%%%%%%%%%%%%%%%%%%%%%%%%%%%

Eventually, summing up the local result and a priori estimates, we can prove the global well-posedness for \eqref{ZKB} in $H^{2, 1}(\R^{2})$. 
%%%%%%%%%%%%%%%%%%%%%%%%%%%%%%%%%%%%%%%%%%%%%%%%%%%%
\begin{proof}[{\bf End of the proof of Theorem~\ref{thm.GWP-H21}}]
By virtue of a priori estimates obtained in Propositions~\ref{prop.apriori-1st}, \ref{prop.apriori-2nd} and \ref{prop.apriori-xxy}, 
we can extend the local solution constructed in Theorem~\ref{thm.LWP} globally in time. 
Therefore, \eqref{ZKB} is globally well-posed in $H^{2, 1}(\R^{2})$ and the solution $u(x, y, t)$ satisfies \eqref{apriori-H21}. Thus, the proof of Theorem~\ref{thm.GWP-H21} has been completed. 
\end{proof}
%%%%%%%%%%%%%%%%%%%%%%%%%%%%%%%%%%%%%%%%%%%%%%%%%%%%

%%%%%%%%%%%%%%%%%%%%%%%%%%%%%%%%%%%%%%%%%%%%%%%%%%%%
\section{Decay estimates for solutions}  
%%%%%%%%%%%%%%%%%%%%%%%%%%%%%%%%%%%%%%%%%%%%%%%%%%%%

In this section, we shall prove Theorem~\ref{thm.main-decay}. Namely, we would like to derive the $L^{\infty}$-decay estimate \eqref{u-sol-decay-Linf} and the $L^{2}$-decay estimate \eqref{u-sol-decay-L2}. This section is divided into two parts below. 

%%%%%%%%%%%%%%%%%%%%%%%%%%%%%%%%%%%%%%%%%%%%%%%%%%%%
\subsection{Decay estimates for the linear solutions}  
%%%%%%%%%%%%%%%%%%%%%%%%%%%%%%%%%%%%%%%%%%%%%%%%%%%%

In this subsection, we would like to analyze the linearized problem \eqref{LZKB}. 
More precisely, we shall derive some decay estimates for the linear solution. 
\begin{comment}
Let us recall the following Cauchy problem: 
\begin{align}\tag{\ref{LZKB}}
\begin{split}
& \tilde{u}_{t} + \tilde{u}_{xxx} + \tilde{u}_{yyx} -\mu\tilde{u}_{xx}=0, \ \ (x, y) \in \R^{2}, \ t>0, \\
& \tilde{u}(x, y, 0) = u_{0}(x, y), \ \ (x, y) \in \R^{2}. 
\end{split}
\end{align}
We also recall that the solution to the above Cauchy problem can be written by 
\begin{equation}\tag{\ref{LZKB-sol}}
\tilde{u}(x, y, t)=(U(t)*u_{0})(x, y), \ \ (x, y) \in \R^{2}, \ t>0, 
\end{equation}
where the integral kernel $U(x, y, t)$ is defined by \eqref{DEF-U}. 
\end{comment}
%In what follows, we shall analyze \eqref{LZKB}. 
First, let us introduce the $L^{\infty}$-decay estimate for the solution $\tilde{u}(x, y, t)$ to \eqref{LZKB}. 
To doing that, we show that the following decay estimate for the integral kernel $U(x, y, t)$ associated with \eqref{LZKB}:  
%%%%%%%%%%%%%%%%%%%%%%%%%%%%%%%%%%%%%%%%%%%%%%%%%%%%
\begin{prop}\label{prop.L-decay-U}
Let $l$ be a non-negative integer. Then, we have 
\begin{equation}\label{L-decay-U}
\left\| \p_{x}^{l} U(\cdot, \cdot, t) \right\|_{L^{\infty}} \le C t^{ -\frac{3}{4} -\frac{l}{2} }, \ \ t>0,  
\end{equation}
where $U(x, y, t)$ is defined by \eqref{DEF-U}. 
\end{prop}
%%%%%%%%%%%%%%%%%%%%%%%%%%%%%%%%%%%%%%%%%%%%%%%%%%%%
%%%%%%%%%%%%%%%%%%%%%%%%%%%%%%%%%%%%%%%%%%%%%%%%%%%%
\begin{proof}
It follows from the definition of $U(x, y, t)$ and the Fourier transform that 
\begin{align}
\p_{x}^{l}U(x, y, t)
& = \frac{1}{(2\pi)^{2}} \int_{\R}\int_{\R} (i\xi )^{l} e^{ -\mu t \xi^{2} +it\xi^{3} +it\xi \eta^{2} }e^{ ix\xi + iy\eta } d\xi d\eta  \nonumber \\
& = \frac{1}{(2\pi)^{ \frac{3}{2} }} \int_{\R} (i\xi )^{l} e^{ -\mu t \xi^{2} +it\xi^{3}} \mathcal{F}^{-1}_{\eta} \left[ e^{it\xi \eta^{2}} \right](y) e^{ix\xi} d\xi \nonumber \\
& = \frac{t^{ -\frac{1}{2} }}{ 4\pi^{ \frac{3}{2} }} \int_{\R} (i\xi )^{l} |\xi|^{ -\frac{1}{2} } e^{ -\mu t \xi^{2} +it\xi^{3} -\frac{iy^{2}}{4t\xi} + \frac{i\pi}{4}\mathrm{sgn} \xi +ix\xi} d\xi.  
\label{re-DEF-U}
\end{align}
Here, we used the following fact: 
\begin{align}
\mathcal{F}^{-1}_{\eta} \left[ e^{it\xi \eta^{2}} \right](y) 
%& 
= \frac{1}{\sqrt{2\pi}} \int_{\R}e^{ it\xi \eta^{2}} e^{iy\eta} d\eta 
%= \frac{1}{\sqrt{2\pi}} e^{ -\frac{iy^{2}}{4t\xi} } \int_{\R} e^{ -\frac{t\xi}{i}\left( \eta + \frac{y}{2t\xi} \right)^{2} } d\eta \nonumber \\
= \frac{1}{\sqrt{2\pi}}\left( \frac{i\pi}{t\xi} \right)^{\frac{1}{2}} e^{ -\frac{iy^{2}}{4t\xi} }
= \frac{e^{ -\frac{iy^{2}}{4t\xi} + \frac{i\pi}{4}\mathrm{sgn} \xi }}{\sqrt{2t|\xi|}}. \label{fact}
\end{align}
From the above expression \eqref{re-DEF-U}, for all non-negative integer $l$, we can immediately derive 
\begin{align*}
\left|\p_{x}^{l}U(x, y, t)\right| 
&\le \frac{t^{ -\frac{1}{2} }}{ 4\pi^{ \frac{3}{2} }} \int_{\R} |\xi|^{ -\frac{1}{2} + l} e^{ -\mu t \xi^{2} } d\xi
%=\frac{t^{ -\frac{1}{2} }}{ 2\pi^{ \frac{3}{2} }} \int_{0}^{\infty} \xi^{ -\frac{1}{2} + l} e^{ -\mu t \xi^{2} } d\xi \\
%&\le \frac{  t^{ -\frac{3}{4} -\frac{l}{2} } }{ 4\pi^{ \frac{3}{2} } \mu^{\frac{1+2l}{4}} } \int_{0}^{\infty} r^{ -\frac{3}{4} + \frac{l}{2} }e^{ -r } dr 
= \frac{ \Gamma \left( \frac{1+2l}{4}\right) }{4\pi^{ \frac{3}{2} }\mu^{\frac{1+2l}{4}}} t^{ -\frac{3}{4} -\frac{l}{2} }, \ \ (x, y) \in \R^{2}, \ t>0. 
\end{align*}
Therefore, we obtain the desired estimate \eqref{L-decay-U}. 
\end{proof}
%%%%%%%%%%%%%%%%%%%%%%%%%%%%%%%%%%%%%%%%%%%%%%%%%%%%
By virtue of Proposition~\ref{prop.L-decay-U}, we can get the following decay estimate for the solution to \eqref{LZKB}: 
%%%%%%%%%%%%%%%%%%%%%%%%%%%%%%%%%%%%%%%%%%%%%%%%%%%%
\begin{cor}\label{cor.L-decay-LZKB}
Suppose that there exists a non-negative integer $j$ such that $\p_{y}^{j}u_{0}\in L^{1}(\R^{2})$. 
Then, for all non-negative integer $l$, we have  
\begin{equation}\label{L-decay-LZKB}
\left\| \p_{x}^{l}\p_{y}^{j}\left(U(t)*u_{0}\right)\right\|_{L^{\infty}} \le Ct^{-\frac{3}{4}-\frac{l}{2}}\|\p_{y}^{j}u_{0}\|_{L^{1}}, \ \ t>0, 
\end{equation}
where $U(x, y, t)$ is defined by \eqref{DEF-U}. 
\end{cor}
%%%%%%%%%%%%%%%%%%%%%%%%%%%%%%%%%%%%%%%%%%%%%%%%%%%%

Next, we shall introduce the $L^{2}$-decay estimate for the solution to \eqref{LZKB}. 
We note that the similar results have been obtained for the one-dimensional problems (cf.~\cite{F19-1, F19-2, K99-1}).  
%%%%%%%%%%%%%%%%%%%%%%%%%%%%%%%%%%%%%%%%%%%%%%%%%%%%
\begin{prop}\label{prop.linear-estL2}
Let $s_1$, $s_2$, and $j$ be non-negative integers satisfying $s_2\ge j$. 
Suppose $u_{0}\in H^{s_1,s_2}(\R^{2})$, $\p_{y}^{j}u_{0}\in L^{1}(\R^{2})$ and $\p_{y}^{j+1}u_{0}\in L^{1}(\R^{2})$. Then, the estimate 
\begin{equation}\label{linear-estL2}
\left\|\p_{x}^{l}\p_{y}^{j}\left(U(t)*u_{0}\right)\right\|_{L^{2}}
\le C(1+t)^{-\frac{1}{4}-\frac{l}{2}}\left( \left\|\p_{y}^{j}u_{0}\right\|_{L^{1}}+\left\|\p_{y}^{j+1}u_{0}\right\|_{L^{1}}\right)
+e^{-\mu t}\left\|\p_{x}^{l}\p_{y}^{j}u_{0}\right\|_{L^{2}}, \ \ t\ge0
\end{equation}
holds for all integer $l$ satisfying $0\le l\le s_1$, where $U(x, y, t)$ is defined by \eqref{DEF-U}. 
\end{prop}
%%%%%%%%%%%%%%%%%%%%%%%%%%%%%%%%%%%%%%%%%%%%%%%%%%%%
%%%%%%%%%%%%%%%%%%%%%%%%%%%%%%%%%%%%%%%%%%%%%%%%%%%%
\begin{proof}
It follows from the Plancherel theorem that 
\begin{align}
&\left\|\p_{x}^{l}\p_{y}^{j}\left(U(t)*u_{0}\right)\right\|_{L^{2}}^{2}
=\left\|(i\xi)^{l}(i\eta)^{j}e^{ -\mu t \xi^{2} + it\xi^{3} + it\xi \eta^{2} }\hat{u}_{0}(\xi, \eta)\right\|_{L^{2}_{\xi \eta}}^{2} \nonumber \\
&=\left(\int_{|\xi|\le1}+\int_{|\xi|\ge1}\right)e^{-2\mu t\xi^{2}}\left(\int_{\R}\left|(i\xi)^{l}(i\eta)^{j}\hat{u}_{0}(\xi, \eta)\right|^{2}d\eta\right) d\xi 
=:L_{1}(t)+L_{2}(t). \label{L-split} 
\end{align}

First, we shall evaluate $L_{1}(t)$. Since $|(i\eta)^{k}\hat{u}_{0}(\xi, \eta)|\le \|\p_{y}^{k}u_{0}\|_{L^{1}}$ holds for all $(\xi, \eta)\in \R^{2}$ and integer $k\ge0$, we can see that  
\begin{align}
L_{1}(t)
%&= \int_{|\xi|\le1}e^{-2\mu t\xi^{2}}\left|(i\xi)^{l}\right|^{2}\left(\int_{\R}\left|(i\eta)^{j}\hat{u}_{0}(\xi, \eta)\right|^{2}\left(1+\left|i\eta\right|^{2}\right) \left(1+\left|i\eta\right|^{2}\right)^{-1}d\eta\right) d\xi \nonumber \\
&\le \sup_{|\xi|\le 1, \, \eta \in \R}\left\{\left(1+\eta^{2}\right)\left|(i\eta)^{j}\hat{u}_{0}(\xi, \eta)\right|^{2}\right\}
\left(\int_{|\xi|\le 1}e^{-2\mu t\xi^{2}}|\xi|^{2l}d\xi\right) \left(\int_{\R}\frac{d\eta}{1+\eta^{2}}\right) \nonumber \\
&\le 2\pi \left( \left\|\p_{y}^{j}u_{0}\right\|_{L^{1}}^{2}+\left\|\p_{y}^{j+1}u_{0}\right\|_{L^{1}}^{2}\right)\int_{0}^{1}e^{-2\mu t\xi^{2}}\xi^{2l}d\xi \nonumber \\
&\le C(1+t)^{-\frac{1}{2}-l}\left( \left\|\p_{y}^{j}u_{0}\right\|_{L^{1}}^{2}+\left\|\p_{y}^{j+1}u_{0}\right\|_{L^{1}}^{2}\right), \ \ t\ge0. \label{linear-L1-est}
\end{align}

Next, we would like to treat $L_{2}(t)$. It follows from the Plancherel theorem that 
\begin{align}
L_{2}(t)
&\le e^{-2\mu t}\int_{|\xi|\ge 1}\int_{\R}\left|(i\xi)^{l}(i\eta)^{j}\hat{u}_{0}(\xi, \eta)\right|^{2} d\eta d\xi 
\le %e^{-2\mu t}\int_{\R^{2}}\left|(i\xi)^{l}(i\eta)^{j}\hat{u}_{0}(\xi, \eta)\right|^{2} d\eta d\xi =
e^{-2\mu t}\left\|\p_{x}^{l}\p_{y}^{j}u_{0}\right\|_{L^{2}}^{2}, \ \ t\ge0. \label{linear-L2-est}
\end{align}
Summing up \eqref{L-split} through \eqref{linear-L2-est}, we can get the desired estimate \eqref{linear-estL2}. 
\end{proof}
%%%%%%%%%%%%%%%%%%%%%%%%%%%%%%%%%%%%%%%%%%%%%%%%%%%%

%%%%%%%%%%%%%%%%%%%%%%%%%%%%%%%%%%%%%%%%%%%%%%%%%%%%
\subsection{Decay estimates for the nonlinear solutions} 
%%%%%%%%%%%%%%%%%%%%%%%%%%%%%%%%%%%%%%%%%%%%%%%%%%%%

In this subsection, we shall derive the decay estimates for the original solution $u(x, y, t)$ to the nonlinear problem \eqref{ZKB}, i.e. let us prove Theorem~\ref{thm.main-decay}. Before doing that, we would like to prepare the following two propositions to evaluate the $L^{\infty}$-norm and the $L^{2}$-norm of the Duhamel term in the integral equation \eqref{integral-eq}. 
%%%%%%%%%%%%%%%%%%%%%%%%%%%%%%%%%%%%%%%%%%%%%%%%%%%%
\begin{prop}\label{prop.Duhamel-est-Linf}
Let $t>a>0$ and $s\ge0$. 
Suppose $g \in C((0, \infty); H^{s, 1}(\R^{2}))$. Then, the estimate
\begin{align}
\left\|\p_{x}^{l}\int_{a}^{t}\p_{x}U(t-\tau)*g(\tau)d\tau\right\|_{L^{\infty}}
\le C t^{\frac{1}{4}}\sup_{a\le \tau \le t}\left(\left\|\p_{x}^{l}g(\cdot, \cdot, \tau)\right\|_{L^{2}}+\left\|\p_{x}^{l}\p_{y}g(\cdot, \cdot, \tau)\right\|_{L^{2}}\right) \label{Duhamel-est-Linf}
\end{align}
holds for all integer $l$ satisfying $0\le l \le s$, where $U(x, y, t)$ is defined by \eqref{DEF-U}. 
\end{prop}
%%%%%%%%%%%%%%%%%%%%%%%%%%%%%%%%%%%%%%%%%%%%%%%%%%%%
%%%%%%%%%%%%%%%%%%%%%%%%%%%%%%%%%%%%%%%%%%%%%%%%%%%%
\begin{proof}
It follows from the Fourier transform and the Cauchy--Schwarz inequality that 
\begin{align}
&\left|\p_{x}^{l}\int_{a}^{t}\p_{x}U(t-\tau)*g(\tau)d\tau\right| \nonumber \\
%&=\left| \int_{a}^{t}\int_{\R^{2}} (i\xi)^{l+1}\hat{U}(\xi, \eta, t-\tau)\hat{g}(\xi, \eta, \tau)e^{ix\xi+iy\eta}d\xi d\eta d\tau \right| \nonumber \\
&=\frac{1}{2\pi} \biggl| \int_{a}^{t}\int_{\R^{2}} (i\xi)^{l+1}e^{-\mu (t-\tau)\xi^{2}+i(t-\tau)\xi^{3}+i(t-\tau)\xi\eta^{2}+ix\xi+iy\eta}\hat{g}(\xi, \eta, \tau)d\xi d\eta d\tau\biggl| \nonumber \\
&\le \int_{a}^{t}\left(\int_{\R}\xi^{2}e^{-2\mu (t-\tau)\xi^{2}}d\xi\right)^{\frac{1}{2}}\left(\int_{\R}\frac{d\eta}{1+\eta^{2}}\right)^{\frac{1}{2}}
\left(\int_{\R^{2}}\left(1+\eta^{2}\right) \left|(i\xi)^{l}\hat{g}(\xi, \eta, \tau)\right|^{2}d\xi d\eta\right)^{\frac{1}{2}}d\tau \nonumber \\
&\le C \int_{a}^{t}(t-\tau)^{-\frac{3}{4}}\left(\left\|\p_{x}^{l}g(\cdot, \cdot, \tau)\right\|_{L^{2}}+\left\|\p_{x}^{l}\p_{y}g(\cdot, \cdot, \tau)\right\|_{L^{2}}\right)d\tau \nonumber \\
&\le C t^{\frac{1}{4}}\sup_{a\le \tau \le t}\left(\left\|\p_{x}^{l}g(\cdot, \cdot, \tau)\right\|_{L^{2}}+\left\|\p_{x}^{l}\p_{y}g(\cdot, \cdot, \tau)\right\|_{L^{2}}\right), \ \ (x, y)\in \R^{2}. 
\end{align}
This completes the proof of the desired result \eqref{Duhamel-est-Linf}. 
\end{proof}
%%%%%%%%%%%%%%%%%%%%%%%%%%%%%%%%%%%%%%%%%%%%%%%%%%%%

%%%%%%%%%%%%%%%%%%%%%%%%%%%%%%%%%%%%%%%%%%%%%%%%%%%%
\begin{prop}\label{prop.Duhamel-est-L2}
Let $s\ge0$. Suppose $g\in C((0, \infty); H^{s, 0}(\R^{2}))$ and $\p_{x}^{l}\p_{y}^{j}g\in C((0, \infty); L^{1}(\R^{2}))$ for all integer $l$ satisfying $0\le l \le s$ and $j=0, 1$. 
Then, the estimate
\begin{align}
&\left\|\p_{x}^{l}\int_{0}^{t}\p_{x}U(t-\tau)*g(\tau)d\tau\right\|_{L^{2}} \nonumber\\
&\le C\int_{0}^{t/2}(1+t-\tau)^{-\frac{3}{4}-\frac{l}{2}}\left(\left\|g(\cdot, \cdot, \tau)\right\|_{L^{1}}+\left\|\p_{y}g(\cdot, \cdot, \tau)\right\|_{L^{1}}\right)d\tau \nonumber\\
&\ \ \ +C\int_{t/2}^{t}(1+t-\tau)^{-\frac{3}{4}}\left(\left\|\p_{x}^{l}g(\cdot, \cdot, \tau)\right\|_{L^{1}}+\left\|\p_{x}^{l}\p_{y}g(\cdot, \cdot, \tau)\right\|_{L^{1}}\right)d\tau \nonumber\\
&\ \ \ +C\left(\int_{0}^{t}e^{-\mu(t-\tau)} \left\|\p_{x}^{l}g(\cdot, \cdot, \tau)\right\|_{L^{2}}^{2}d\tau \right)^{\frac{1}{2}}, \ \ t>0 \label{Duhamel-est-L2}
\end{align}
holds for all integer $l$ satisfying $0\le l \le s$, where $U(x, y, t)$ is defined by \eqref{DEF-U}. 
\end{prop}
%%%%%%%%%%%%%%%%%%%%%%%%%%%%%%%%%%%%%%%%%%%%%%%%%%%%
%%%%%%%%%%%%%%%%%%%%%%%%%%%%%%%%%%%%%%%%%%%%%%%%%%%%
\begin{proof}
For simplicity, we set 
\begin{equation}\label{Duhamel-D}
D(x, y, t):=\int_{0}^{t}\p_{x}U(t-\tau)*g(\tau)d\tau. 
\end{equation}
By using the Plancherel theorem and splitting the integral, we have
\begin{align}
\left\|\p_{x}^{l}D(\cdot, \cdot, t)\right\|_{L^{2}} 
&\le \left\|(i\xi)^{l}\hat{D}(\xi, \eta, t)\right\|_{L^{2}(|\xi|\le1, \, \eta \in \R)}+\left\|(i\xi)^{l}\hat{D}(\xi, \eta, t)\right\|_{L^{2}(|\xi|\ge1, \, \eta \in \R)}  \nonumber \\
&=:D_{1}(t)+D_{2}(t). \label{D-split}
\end{align}

First, we evaluate $D_{1}(t)$. From the definition of $U(x, y, t)$ by \eqref{DEF-U}, we obtain
\begin{align}
D_{1}(t)
&\le \int_{0}^{t}\left\|(i\xi)^{l+1}e^{ -\mu (t-\tau) \xi^{2} + i(t-\tau)\xi^{3} + i(t-\tau)\xi \eta^{2} }
\hat{g}(\xi, \eta, \tau)\right\|_{L^{2}(|\xi|\le1, \, \eta \in \R)}d\tau \nonumber\\
%&= \int_{0}^{t} \left( \int_{|\xi|\le 1} \int_{\R} \left|i\xi\right|^{2(l+1)}e^{-2\mu(t-\tau)\xi^{2}} \left|\hat{g}(\xi, \eta, \tau)\right|^{2}d\eta d\xi \right)^{\frac{1}{2}}d\tau \nonumber\\
%&= \int_{0}^{t} \biggl( \int_{|\xi|\le 1} \int_{\R} \left|i\xi\right|^{2(l+1)}e^{-2\mu(t-\tau)\xi^{2}}\left(1+\left|i\eta\right|^{2}\right)^{-1} \left(1+|i\eta|^{2}\right)\left|\hat{g}(\xi, \eta, \tau)\right|^{2}d\eta d\xi \biggl)^{\frac{1}{2}}d\tau \nonumber \\
&\le \int_{0}^{t/2}\sup_{|\xi|\le 1, \, \eta \in \R} \left\{\left(1+|\eta|\right)\left|\hat{g}(\xi, \eta, \tau)\right|\right\}
\left(\int_{|\xi|\le 1}\xi^{2(l+1)}e^{-2\mu(t-\tau)\xi^{2}}d\xi\right)^{\frac{1}{2}}\left(\int_{\R}\frac{d\eta}{1+\eta^{2}}\right)^{\frac{1}{2}} d\tau \nonumber \\
&\ \ \ \ +\int_{t/2}^{t}\sup_{|\xi|\le 1, \, \eta \in \R} \left\{\left(1+|\eta|\right)\left|(i\xi)^{l}\hat{g}(\xi, \eta, \tau)\right|\right\}
\left(\int_{|\xi|\le 1}\xi^{2}e^{-2\mu(t-\tau)\xi^{2}}d\xi\right)^{\frac{1}{2}}\left(\int_{\R}\frac{d\eta}{1+\eta^{2}}\right)^{\frac{1}{2}} d\tau \nonumber \\
&\le C\int_{0}^{t/2}(1+t-\tau)^{-\frac{3}{4}-\frac{l}{2}}\left(\left\|g(\cdot, \cdot, \tau)\right\|_{L^{1}}+\left\|\p_{y}g(\cdot, \cdot, \tau)\right\|_{L^{1}}\right)d\tau \nonumber\\
&\ \ \ \ +C\int_{t/2}^{t}(1+t-\tau)^{-\frac{3}{4}}\left(\left\|\p_{x}^{l}g(\cdot, \cdot, \tau)\right\|_{L^{1}}+\left\|\p_{x}^{l}\p_{y}g(\cdot, \cdot, \tau)\right\|_{L^{1}}\right)d\tau, \ \ t>0. \label{Duhamel-D1-est}
\end{align}
Here, we used the following estimate: 
\begin{equation*}
\left(\int_{|\xi|\le1}|\xi|^{j}e^{-2\mu(t-\tau)\xi^{2}}d\xi\right)^{\frac{1}{2}} \le C(1+t-\tau)^{-\frac{1+j}{4}}, \ \ j\ge0. 
\end{equation*}

Next, we shall treat $D_{2}(t)$. Applying the Cauchy--Schwarz inequality for the $\tau$-integral, we have 
\begin{align*}
\left|(i\xi)^{l}\hat{D}(\xi, t)\right|
%&=\left|\int_{0}^{t}(i\xi)^{l+1}e^{-\mu(t-\tau)|\xi|^{2}+ i(t-\tau)\xi^{3} + i(t-\tau)\xi \eta^{2}}\hat{g}(\xi, \eta, \tau)d\tau\right| \\
&\le \int_{0}^{t}|\xi|e^{-\mu(t-\tau)\xi^{2}}\left|(i\xi)^{l}\hat{g}(\xi, \eta, \tau)\right|d\tau \\
&\le \left(\int_{0}^{t}\xi^{2}e^{-\mu(t-\tau)\xi^{2}}d\tau \right)^{\frac{1}{2}}\left(\int_{0}^{t}e^{-\mu(t-\tau)\xi^{2}}\left|(i\xi)^{l}\hat{g}(\xi, \eta, \tau)\right|^{2}d\tau \right)^{\frac{1}{2}} \\
&\le C\left(\int_{0}^{t}e^{-\mu(t-\tau)\xi^{2}}\left|(i\xi)^{l}\hat{g}(\xi, \eta, \tau)\right|^{2}d\tau \right)^{\frac{1}{2}}, \ \ |\xi|\ge1, \ \eta \in \R. 
\end{align*}
Therefore, it follows that 
\begin{align}
D_{2}(t)
%&\le C\left(\int_{|\xi|\ge1}\int_{\R} \int_{0}^{t}e^{-\mu(t-\tau)\xi^{2}}\left|(i\xi)^{l}\hat{g}(\xi, \eta, \tau)\right|^{2}d\tau d\eta d\xi  \right)^{\frac{1}{2}} \nonumber \\
&\le C\left(\int_{0}^{t}e^{-\mu(t-\tau)}\int_{|\xi|\ge1}\int_{\R}\left|(i\xi)^{l}\hat{g}(\xi, \eta, \tau)\right|^{2} d\eta d\xi  d\tau \right)^{\frac{1}{2}} \nonumber \\
&\le C\left(\int_{0}^{t}e^{-\mu(t-\tau)}\left\|\p_{x}^{l}g(\cdot, \cdot, \tau)\right\|_{L^{2}}^{2}d\tau \right)^{\frac{1}{2}}, \ \ t>0. \label{Duhamel-D2-est}
\end{align}
Combining \eqref{Duhamel-D} through \eqref{Duhamel-D2-est}, we obtain the desired estimate \eqref{Duhamel-est-L2}. 
\end{proof}
%%%%%%%%%%%%%%%%%%%%%%%%%%%%%%%%%%%%%%%%%%%%%%%%%%%%

In what follows, let us prove the $L^{\infty}$-decay estimate \eqref{u-sol-decay-Linf}: 
%%%%%%%%%%%%%%%%%%%%%%%%%%%%%%%%%%%%%%%%%%%%%%%%%%%%
\begin{thm}\label{thm.u-sol-decay-Linf}
Let $p\ge2$ be an integer. Assume that $u_{0}\in H^{2, 1}(\R^{2})\cap L^{1}(\R^{2})$ and $E_{0}=\left\|u_{0}\right\|_{H^{2, 1}}+\left\|u_{0}\right\|_{L^{1}}$ is sufficiently small. 
Then, the solution $u(x, y, t)$ to \eqref{ZKB} satisfies the following estimate: 
\begin{equation}\tag{\ref{u-sol-decay-Linf}}
\left\|\p_{x}^{l}u(\cdot, \cdot, t)\right\|_{L^{\infty}} \le CE_{0}(1+t)^{-\frac{3}{4}-\frac{l}{2}}, \ \ t\ge0, \ l=0, 1.  
\end{equation}
\end{thm}
%%%%%%%%%%%%%%%%%%%%%%%%%%%%%%%%%%%%%%%%%%%%%%%%%%%%
%%%%%%%%%%%%%%%%%%%%%%%%%%%%%%%%%%%%%%%%%%%%%%%%%%%%
\begin{proof}
In order to prove \eqref{u-sol-decay-Linf}, let us introduce the following functions for $t\ge0$: 
\begin{equation}\label{DEF-H(t)}
\mathcal{H}(t):=\sup_{0\le \tau \le t}\sum_{l=0}^{1}H_{l}(\tau), \ \ 
H_{l}(t):= (1+t)^{\frac{3}{4}+\frac{l}{2}} \left\|\p_{x}^{l}u(\cdot, \cdot, t)\right\|_{L^{\infty}}, \ \ l=0, 1.
\end{equation}
Then, it follows from the integral equation \eqref{integral-eq} that 
\begin{align}
H_{l}(t)
&\le (1+t)^{\frac{3}{4}+\frac{l}{2}} \left\|\p_{x}^{l}U(t)*u_{0}\right\|_{L^{\infty}} +\frac{|\beta|}{p+1}(1+t)^{\frac{3}{4}+\frac{l}{2}}\left\|\p_{x}^{l}\int_{0}^{t/2}\p_{x}U(t-\tau)*\left(u^{p+1}(\tau)\right)d\tau\right\|_{L^{\infty}} \nonumber \\
&\ \ \ \ +\frac{|\beta|}{p+1}(1+t)^{\frac{3}{4}+\frac{l}{2}}\left\|\p_{x}^{l}\int_{t/2}^{t}\p_{x}U(t-\tau)*\left(u^{p+1}(\tau)\right)d\tau\right\|_{L^{\infty}} \nonumber \\
&=:H_{l.0}(t)+H_{l.1}(t)+H_{l.2}(t). \label{H(t)-est-split}
\end{align}

In what follows, we shall evaluate $H_{l.0}(t)$, $H_{l.1}(t)$ and $H_{l.2}(t)$ for $t\ge1$. First, let us evaluate $H_{l.0}(t)$. 
It directly follows from Corollary~\ref{cor.L-decay-LZKB} that 
\begin{align}
H_{l.0}(t)
\le C\left\|u_{0}\right\|_{L^{1}}(1+t)^{\frac{3}{4}+\frac{l}{2}} t^{-\frac{3}{4}-\frac{l}{2}} 
\le C\left\|u_{0}\right\|_{L^{1}}, \ \ t\ge1, \ l=0, 1. \label{Linf-decay-est-I0}
\end{align}

Next, we would like to treat $H_{l.1}(t)$. From the Young inequality and Proposition~\ref{prop.L-decay-U}, we have 
\begin{align}
&\left\|\p_{x}^{l}\int_{0}^{t/2}\p_{x}U(t-\tau)*\left(u^{p+1}(\tau)\right)d\tau\right\|_{L^{\infty}} \nonumber \\
&\le \int_{0}^{t/2}\left\|\p_{x}^{l+1}U(\cdot, \cdot, t-\tau)\right\|_{L^{\infty}}\left\|u^{p+1}(\cdot, \cdot, \tau)\right\|_{L^{1}}d\tau  \nonumber \\
&\le C\int_{0}^{t/2}(t-\tau)^{-\frac{3}{4}-\frac{l+1}{2}}\left\|u^{p+1}(\cdot, \cdot, \tau)\right\|_{L^{1}}d\tau 
\le Ct^{-\frac{3}{4}-\frac{l+1}{2}}\int_{0}^{t/2}\left\|u^{p+1}(\cdot, \cdot, \tau)\right\|_{L^{1}}d\tau. \label{Linf-decay-pre-I1}
\end{align}
In order to complete the evaluation \eqref{Linf-decay-pre-I1}, we shall prepare an estimate for the nonlinear term $u^{p+1}(x, y, t)$. 
It follows from the definition of $H_{l}(t)$ (i.e. \eqref{DEF-H(t)}) and a priori estimate \eqref{apriori-H21} that 
\begin{align}
\left\|u^{p+1}(\cdot, \cdot, \tau)\right\|_{L^{1}}
&\le \left\|u(\cdot, \cdot, \tau)\right\|_{L^{\infty}}^{p-1}\left\|u(\cdot, \cdot, \tau)\right\|_{L^{2}}^{2} 
\le C\left\|u_{0}\right\|_{H^{2, 1}}^{2}H_{0}(\tau)^{p-1} (1+\tau)^{-\frac{3(p-1)}{4}} \nonumber \\
&\le C\left\|u_{0}\right\|_{H^{2, 1}}^{2}H_{0}(\tau)^{p-1} (1+\tau)^{-\frac{3(p-1)}{4}}. \label{u^{p+1}-L1-pre}
\end{align}
Therefore, we can easily see that 
\begin{align}
\int_{0}^{t/2}\left\|u^{p+1}(\cdot, \cdot, \tau)\right\|_{L^{1}}d\tau \nonumber 
&\le C\left\|u_{0}\right\|_{H^{2, 1}}^{2}\int_{0}^{t/2}H_{0}(\tau)^{p-1}(1+\tau)^{-\frac{3(p-1)}{4}}d\tau \\
&\le C\left\|u_{0}\right\|_{H^{2, 1}}^{2}\mathcal{H}(t)^{p-1}
\begin{cases}
(1+t)^{\frac{1}{4}}, &p=2, \\
1, &p\ge3. 
\end{cases} \label{u^{p+1}-L1-int}
\end{align}
Combining \eqref{H(t)-est-split}, \eqref{Linf-decay-pre-I1} and \eqref{u^{p+1}-L1-int}, we arrive at
\begin{align}
H_{l.1}(t)
&\le C(1+t)^{\frac{3}{4}+\frac{l}{2}}t^{-\frac{3}{4}-\frac{l+1}{2}}\left\|u_{0}\right\|_{H^{2, 1}}^{2}\mathcal{H}(t)^{p-1}(1+t)^{\frac{1}{4}} \nonumber \\
&\le C\left\|u_{0}\right\|_{H^{2, 1}}^{2}\mathcal{H}(t)^{p-1}(1+t)^{-\frac{1}{4}}, \ \ t\ge1, \ l=0, 1. \label{Linf-decay-I1}
\end{align}

Next, we shall derive the estimate for $H_{l.2}(t)$. Now, it follows from Proposition~\ref{prop.Duhamel-est-Linf} that 
\begin{align} \label{Linf-decay-pre-I2}
H_{l.2}(t)\le C(1+t)^{\frac{3}{4}+\frac{l}{2}}t^{\frac{1}{4}}
\sup_{t/2\le \tau \le t}\left(\left\|\p_{x}^{l}\left(u^{p+1}(\cdot, \cdot, \tau)\right)\right\|_{L^{2}}+\left\|\p_{x}^{l}\p_{y}\left(u^{p+1}(\cdot, \cdot, \tau)\right)\right\|_{L^{2}}\right). 
\end{align}
Therefore, we need to prepare some estimates of $\p_{x}^{l}\left(u^{p+1}(x, y, t)\right)$ and $\p_{x}^{l}\p_{y}\left(u^{p+1}(x, y, t)\right)$ for $l=0, 1$. 
First, from the definition of $H_{l}(t)$ (i.e. \eqref{DEF-H(t)}) and a priori estimate \eqref{apriori-H21}, we have 
\begin{align}
\left\|u^{p+1}(\cdot, \cdot, \tau)\right\|_{L^{2}}
&\le \left\|u(\cdot, \cdot, \tau)\right\|_{L^{\infty}}^{p}\left\|u(\cdot, \cdot, \tau)\right\|_{L^{2}} 
\le C\left\|u_{0}\right\|_{H^{2, 1}}H_{0}(\tau)^{p} (1+\tau)^{-\frac{3p}{4}} \nonumber \\
&\le C\left\|u_{0}\right\|_{H^{2, 1}}\mathcal{H}(t)^{p}(1+\tau)^{-\frac{3p}{4}}. \label{u^{p+1}-L2-pre} 
\end{align}
In the same way as before, we obtain 
\begin{align}
\left\|\p_{x}\left(u^{p+1}(\cdot, \cdot, \tau)\right)\right\|_{L^{2}}+\left\|\p_{y}\left(u^{p+1}(\cdot, \cdot, \tau)\right)\right\|_{L^{2}}
%&\le (p+1)\left\|u(\cdot, \cdot, \tau)\right\|_{L^{\infty}}^{p}\left\|u_{y}(\cdot, \cdot, \tau)\right\|_{L^{2}} 
%\le C\left\|u_{0}\right\|_{H^{2, 1}}H_{0}(\tau)^{p} (1+\tau)^{-\frac{3p}{4}}\nonumber \\
%&
%\le C\left\|u_{0}\right\|_{H^{2, 1}}\mathcal{H}(t)^{p} (1+\tau)^{-\frac{3p}{4}}, \label{dyu^{p+1}-L2-pre} \\
%\left\|\p_{y}\left(u^{p+1}(\cdot, \cdot, \tau)\right)\right\|_{L^{2}}
%&\le (p+1)\left\|u(\cdot, \cdot, \tau)\right\|_{L^{\infty}}^{p}\left\|u_{x}(\cdot, \cdot, \tau)\right\|_{L^{2}} 
%\le C\left\|u_{0}\right\|_{H^{2, 1}}H_{0}(\tau)^{p} (1+\tau)^{-\frac{3p}{4}} \nonumber \\
%&
\le C\left\|u_{0}\right\|_{H^{2, 1}}\mathcal{H}(t)^{p}(1+\tau)^{-\frac{3p}{4}}. \label{dxu^{p+1}-L2-pre}
\end{align}
In addition, noticing $\p_{x}\p_{y}(u^{p+1})=p(p+1)u^{p-1}u_{x}u_{y}+(p+1)u^{p}u_{xy}$, we analogously have
\begin{align}
&\left\|\p_{x}\p_{y}\left(u^{p+1}(\cdot, \cdot, \tau)\right)\right\|_{L^{2}} \nonumber \\
&\le C\left(\left\|u(\cdot, \cdot, \tau)\right\|_{L^{\infty}}^{p-1}\left\|u_{x}(\cdot, \cdot, \tau)\right\|_{L^{\infty}}\left\|u_{y}(\cdot, \cdot, \tau)\right\|_{L^{2}}+\left\|u(\cdot, \cdot, \tau)\right\|_{L^{\infty}}^{p}\left\|u_{xy}(\cdot, \cdot, \tau)\right\|_{L^{2}} \right)\nonumber \\
&\le C\left\{H_{0}(\tau)^{p-1}(1+\tau)^{-\frac{3(p-1)}{4}}H_{1}(\tau)(1+\tau)^{-\frac{5}{4}}\left\|u_{0}\right\|_{H^{2, 1}}+ H_{0}(\tau)^{p}(1+\tau)^{-\frac{3p}{4}}\left\|u_{0}\right\|_{H^{2, 1}} \right\} \nonumber \\
&\le C\left\|u_{0}\right\|_{H^{2, 1}}\mathcal{H}(t)^{p}(1+\tau)^{-\frac{3p}{4}}. \label{dxdyu^{p+1}-L2-pre}
\end{align}
Therefore, summing up \eqref{Linf-decay-pre-I2} through \eqref{dxdyu^{p+1}-L2-pre}, we eventually obtain 
\begin{align} 
H_{l.2}(t)
&\le C(1+t)^{\frac{3}{4}+\frac{l}{2}}t^{\frac{1}{4}}
\left\|u_{0}\right\|_{H^{2, 1}}\mathcal{H}(t)^{p}(1+t)^{-\frac{3p}{4}} \nonumber \\
&\le C\left\|u_{0}\right\|_{H^{2, 1}}\mathcal{H}(t)^{p}(1+t)^{-\frac{3p-4-2l}{4}}, \ \ t\ge1, \ l=0, 1. \label{Linf-decay-I2}
\end{align}

Summarizing up \eqref{H(t)-est-split}, \eqref{Linf-decay-est-I0}, \eqref{Linf-decay-I1} and \eqref{Linf-decay-I2} and using the Young inequality, it follows from $p\ge2$ that 
\begin{align}
H_{l}(t)&\le H_{l.0}(t)+H_{l.1}(t)+H_{l.2}(t) \nonumber \\
&\le C\left\{ \left\|u_{0}\right\|_{L^{1}} + \left\|u_{0}\right\|_{H^{2, 1}}^{2}\mathcal{H}(t)^{p-1}(1+t)^{-\frac{1}{4}} + \left\|u_{0}\right\|_{H^{2, 1}}\mathcal{H}(t)^{p}(1+t)^{-\frac{3p-4-2l}{4}}\right\} \nonumber\\
&\le C\left\{\left\|u_{0}\right\|_{L^{1}}+\left\|u_{0}\right\|_{H^{2, 1}}^{p+1}+\left\|u_{0}\right\|_{H^{2, 1}}\mathcal{H}(t)^{p}\right\}, \ \ t\ge1, \ l=0, 1. \label{Linf-decay-est-H(t)}
\end{align}

On the other hand, by virtue of a priori estimate \eqref{apriori-H21}, we can easily check the uniform boundedness of $H_{l}(t)$ for $0\le t \le1$ and $l=0, 1$. 
Actually, in the same way to get \eqref{u-est-Linfty}, we have from Lemma~\ref{lem.est-Linfty} and \eqref{apriori-H21} that 
\begin{align}
\left\|u_{x}(t)\right\|_{L^{\infty}}^{2}
&\le \left\|u_{xx}(t)\right\|_{L^{2}}^{2}+\left\|u_{xy}(t)\right\|_{L^{2}}^{2}+\left\|u_{x}(t)\right\|_{L^{2}}^{2}+\left\|u_{xxy}(t)\right\|_{L^{2}}^{2} 
\le C\left\|u_{0}\right\|_{H^{2, 1}}^{2}, \ \ t\ge0.  \label{dxu-est-Linfty}
\end{align}
Thus, combining \eqref{DEF-H(t)}, \eqref{u-est-Linfty} and \eqref{dxu-est-Linfty}, it follows that 
\begin{equation}
H_{l}(t)
= (1+t)^{\frac{3}{4}+\frac{l}{2}} \left\|\p_{x}^{l}u(\cdot, \cdot, t)\right\|_{L^{\infty}} 
\le C\left\|u_{0}\right\|_{H^{2, 1}}, \ \ 0\le t\le 1, \ l=0, 1. \label{Linf-decay-est-H(t)-zero}
\end{equation}

Finally, let us derive the desired estimate \eqref{u-sol-decay-Linf}. According to \eqref{Linf-decay-est-H(t)} and \eqref{Linf-decay-est-H(t)-zero}, we can see that the estimate
\[
\mathcal{H}(t)\le C\left(\left\|u_{0}\right\|_{L^{1}}+\left\|u_{0}\right\|_{H^{2, 1}}+\left\|u_{0}\right\|_{H^{2, 1}}\mathcal{H}(t)^{p} \right), \ \ t\ge 0 
\]
 is true if $\left\|u_{0}\right\|_{H^{2, 1}}\le 1$ holds. 
Therefore, under the smallness assumption of $E_{0}$\ $(=\left\|u_{0}\right\|_{H^{2, 1}}+\left\|u_0\right\|_{L^1})$, in the same way to get \eqref{est-M(t)} and \eqref{est-N(t)}, applying the continuity argument again, we arrive at the following estimate:
\begin{equation}\label{H(t)-bdd}
\mathcal{H}(t)\le CE_{0}, \ \ t\ge0. 
\end{equation}
This means that the proof of the desired result \eqref{u-sol-decay-Linf} has been completed. 
\end{proof}
%%%%%%%%%%%%%%%%%%%%%%%%%%%%%%%%%%%%%%%%%%%%%%%%%%%%

Finally, we would like to prove the $L^{2}$-decay estimate \eqref{u-sol-decay-L2}: 
%%%%%%%%%%%%%%%%%%%%%%%%%%%%%%%%%%%%%%%%%%%%%%%%%%%%
\begin{thm}\label{thm.u-sol-decay-L2}
Let $p\ge2$ be an integer. Assume that $u_{0}\in H^{2, 1}(\R^{2})\cap L^{1}(\R^{2})$, $\p_{y}u_{0}\in L^{1}(\R^{2})$ and $E_{1}=\left\|u_{0}\right\|_{H^{2, 1}}+\left\|u_{0}\right\|_{L^{1}}+\left\|\p_{y}u_{0}\right\|_{L^{1}}$ is sufficiently small. 
Then, the solution $u(x, y, t)$ to \eqref{ZKB} satisfies the following estimate: 
\begin{equation}\tag{\ref{u-sol-decay-L2}}
\left\|\p_{x}^{l}u(\cdot, \cdot, t)\right\|_{L^{2}} \le CE_{1}(1+t)^{-\frac{1}{4}-\frac{l}{2}}, \ \ t\ge0, \ l=0, 1.  
\end{equation}
\end{thm}
%%%%%%%%%%%%%%%%%%%%%%%%%%%%%%%%%%%%%%%%%%%%%%%%%%%%
%%%%%%%%%%%%%%%%%%%%%%%%%%%%%%%%%%%%%%%%%%%%%%%%%%%%
\begin{proof}
In order to prove \eqref{u-sol-decay-L2}, let us introduce the following function: 
\begin{equation}\label{DEF-K(t)}
K(t):= (1+t)^{\frac{1}{4}} \left\|u(\cdot, \cdot, t)\right\|_{L^{2}} + (1+t)^{\frac{3}{4}} \left\|u_{x}(\cdot, \cdot, t)\right\|_{L^{2}}, \ \ t\ge0. 
\end{equation}
Then, it follows from the integral equation \eqref{integral-eq} that 
\begin{align}\label{K(t)-est-split}
K(t)
&\le \left\{ (1+t)^{\frac{1}{4}} \left\|U(t)*u_{0}\right\|_{L^{2}} + (1+t)^{\frac{3}{4}} \left\|\p_{x}\left(U(t)*u_{0}\right)\right\|_{L^{2}}\right\} \nonumber \\
&\ \ \ \ +\frac{|\beta|}{p+1}(1+t)^{\frac{1}{4}}\left\|\int_{0}^{t}\p_{x}U(t-\tau)*\left(u^{p+1}(\tau)\right)d\tau\right\|_{L^{2}} \nonumber \\
&\ \ \ \ +\frac{|\beta|}{p+1}(1+t)^{\frac{3}{4}}\left\|\p_x\int_{0}^{t}\p_{x}U(t-\tau)*\left(u^{p+1}(\tau)\right)d\tau\right\|_{L^{2}} 
=:I_{0}(t)+I_{1}(t)+I_{2}(t). 
\end{align}

In the following, we shall evaluate $I_{0}(t)$, $I_{1}(t)$ and $I_{2}(t)$. For $I_{0}(t)$, it directly follows from Proposition~\ref{prop.linear-estL2} that 
\begin{equation}\label{L2-decay-est-I0}
I_{0}(t)
\le C\left( \left\|u_{0}\right\|_{L^{1}}+\left\|\p_{y}u_{0}\right\|_{L^{1}}+\left\|u_{0}\right\|_{L^{2}}+\left\|\p_{x}u_{0}\right\|_{L^{2}}\right)
\le CE_{1}, \ \ t\ge0. 
\end{equation}
Therefore, to prove \eqref{u-sol-decay-L2}, it is sufficient to evaluate $I_{1}(t)$ and $I_{2}(t)$. For these terms, if we set
\begin{equation}\label{DEF-Hl(t)}
G_{l}(t):=\left(\int_{0}^{t}e^{-\mu(t-\tau)} \left\|\p_{x}^{l}\left(u^{p+1}(\cdot, \cdot, \tau)\right)\right\|_{L^{2}}^{2}d\tau \right)^{\frac{1}{2}}, \ \ l=0, 1, 
\end{equation}
then we have from Proposition~\ref{prop.Duhamel-est-L2} that 
\begin{align*}
I_{1}(t)&\le C(1+t)^{\frac{1}{4}}\left\{ \int_{0}^{t}(1+t-\tau)^{-\frac{3}{4}}\left( \left\|u^{p+1}(\cdot, \cdot, \tau)\right\|_{L^{1}} + \left\|\p_{y}\left(u^{p+1}(\cdot, \cdot, \tau)\right)\right\|_{L^{1}} \right)d\tau +G_{0}(t)\right\}, \\
I_{2}(t)&\le C(1+t)^{\frac{3}{4}}\biggl\{ \int_{0}^{t/2}(1+t-\tau)^{-\frac{3}{4}-\frac{1}{2}}\left( \left\|u^{p+1}(\cdot, \cdot, \tau)\right\|_{L^{1}} + \left\|\p_{y}\left(u^{p+1}(\cdot, \cdot, \tau)\right)\right\|_{L^{1}} \right)d\tau \nonumber \\
&\ \ \ \ +\int_{t/2}^{t}(1+t-\tau)^{-\frac{3}{4}}\left( \left\|\p_{x}\left(u^{p+1}(\cdot, \cdot, \tau)\right)\right\|_{L^{1}} + \left\|\p_{x}\p_{y}\left(u^{p+1}(\cdot, \cdot, \tau)\right)\right\|_{L^{1}} \right)d\tau + G_{1}(t)\biggl\}. 
\end{align*}
Here, noticing 
\[
(1+t)^{\frac{3}{4}} \le C(1+t)^{\frac{1}{4}}(1+t-\tau)^{\frac{1}{2}}, \ \ 0\le \tau \le \frac{t}{2},  
\]
then we obtain 
\begin{align}
I_{1}(t)+I_{2}(t)
&\le C(1+t)^{\frac{1}{4}}\int_{0}^{t}(1+t-\tau)^{-\frac{3}{4}}\left( \left\|u^{p+1}(\cdot, \cdot, \tau)\right\|_{L^{1}} + \left\|\p_{y}\left(u^{p+1}(\cdot, \cdot, \tau)\right)\right\|_{L^{1}} \right)d\tau \nonumber \\
&\ \ \ +C(1+t)^{\frac{3}{4}}\int_{t/2}^{t}(1+t-\tau)^{-\frac{3}{4}}\left( \left\|\p_{x}\left(u^{p+1}(\cdot, \cdot, \tau)\right)\right\|_{L^{1}} + \left\|\p_{x}\p_{y}\left(u^{p+1}(\cdot, \cdot, \tau)\right)\right\|_{L^{1}} \right)d\tau \nonumber \\
&\ \ \ +C\left\{ (1+t)^{\frac{1}{4}}G_{0}(t)+(1+t)^{\frac{3}{4}}G_{1}(t)\right\} \nonumber \\
&=:J_{1}(t)+J_{2}(t)+J_{3}(t). \label{DEF-J(t)s}
\end{align}

In what follows, let us evaluate $J_{1}(t)$, $J_{2}(t)$ and $J_{3}(t)$. 
Now, we shall introduce the following function: 
\[
\mathcal{K}(t):=\sup_{0\le \tau \le t}K(\tau), \ \ t\ge0. 
\]
First, we would like to treat $J_{1}(t)$. 
We start with preparing some $L^{1}$-decay estimates for the nonlinear term $u^{p+1}(x, y, t)$. Now, from \eqref{u-sol-decay-Linf} and \eqref{DEF-K(t)}, we obtain 
\begin{align}
\left\|u^{p+1}(\cdot, \cdot, \tau)\right\|_{L^{1}} 
&\le \left\|u(\cdot, \cdot, \tau)\right\|_{L^{\infty}}^{p-1}\left\|u(\cdot, \cdot, \tau)\right\|_{L^{2}}^{2} \nonumber \\
&\le CE_{0}^{p-1}(1+\tau)^{-\frac{3(p-1)}{4}}K(\tau)^{2}(1+\tau)^{-\frac{1}{2}} 
\le CE_{0}^{p-1}K(\tau)^{2}(1+\tau)^{-\frac{3p-1}{4}}. \label{u^{p+1}-L1-pre2} 
\end{align}
In addition to \eqref{u-sol-decay-Linf} and \eqref{DEF-K(t)}, applying the Cauchy--Schwarz inequality and \eqref{apriori-H21}, we get 
\begin{align}
\left\|\p_{y}\left(u^{p+1}(\cdot, \cdot, \tau)\right)\right\|_{L^{1}} 
&\le C\left\|u(\cdot, \cdot, \tau)\right\|_{L^{\infty}}^{p-1}\left\|u(\cdot, \cdot, \tau)\right\|_{L^{2}}\left\|u_{y}(\cdot, \cdot, \tau)\right\|_{L^{2}}  \nonumber \\
&\le CE_{0}^{p-1}(1+\tau)^{-\frac{3(p-1)}{4}}K(\tau)(1+\tau)^{-\frac{1}{4}}\left\|u_{0}\right\|_{H^{2, 1}} \nonumber \\
&\le CE_{0}^{p-1}K(\tau)(1+\tau)^{-\frac{3p-2}{4}}. \label{dyu^{p+1}-L1-pre}
\end{align}
Combining \eqref{u^{p+1}-L1-pre2} and \eqref{dyu^{p+1}-L1-pre}, we have 
\begin{align}
\left\|u^{p+1}(\cdot, \cdot, \tau)\right\|_{L^{1}} + \left\|\p_{y}\left(u^{p+1}(\cdot, \cdot, \tau)\right)\right\|_{L^{1}} 
\le CE_{0}^{p-1}\left(K(\tau)+K(\tau)^{2}\right)(1+\tau)^{-\frac{3p-2}{4}}. \label{u^{p+1}+dyu^{p+1}-L1}
\end{align}
Therefore, it follows from \eqref{DEF-J(t)s} and \eqref{u^{p+1}+dyu^{p+1}-L1} that 
\begin{align}\label{J1-est-pre}
J_{1}(t)&\le CE_{0}^{p-1}\left(\mathcal{K}(t)+\mathcal{K}(t)^{2}\right)
(1+t)^{\frac{1}{4}}\int_{0}^{t}(1+t-\tau)^{-\frac{3}{4}}(1+\tau)^{-\frac{3p-2}{4}}d\tau. 
\end{align}
Since 
\begin{align*}
&\int_{0}^{t}(1+t-\tau)^{-\frac{3}{4}}(1+\tau)^{-\frac{3p-2}{4}}d\tau
=\left(\int_{t/2}^{t}+\int_{0}^{t/2}\right)(1+t-\tau)^{-\frac{3}{4}}(1+\tau)^{-\frac{3p-2}{4}}d\tau \\
&\le C(1+t)^{-\frac{3p-2}{4}}(1+t)^{\frac{1}{4}}+C(1+t)^{-\frac{3}{4}}
\begin{cases}
\log(2+t), &p=2, \\
1, &p\ge3
\end{cases} \\
&\le C(1+t)^{-\frac{3(p-1)}{4}}+C(1+t)^{-\frac{3}{4}}\log(2+t) \le C(1+t)^{-\frac{3}{4}}\log(2+t), 
\end{align*}
we can complete the evaluation \eqref{J1-est-pre} as follows: 
\begin{align}\label{J1-est}
J_{1}(t)\le CE_{0}^{p-1}\left(\mathcal{K}(t)+\mathcal{K}(t)^{2}\right)(1+t)^{-\frac{1}{2}}\log(2+t), \ \ t\ge0. 
\end{align}

Next, we would like to evaluate $J_{2}(t)$. Before doing that, we shall prepare some additional $L^{1}$-decay estimates for the nonlinear term $u^{p+1}(x, y, t)$. 
In the same way to get \eqref{dyu^{p+1}-L1-pre}, it follows from the Cauchy--Schwarz inequality, \eqref{u-sol-decay-Linf} and \eqref{DEF-K(t)} that 
\begin{align}
\left\|\p_{x}\left(u^{p+1}(\cdot, \cdot, \tau)\right)\right\|_{L^{1}} 
&\le C\left\|u(\cdot, \cdot, \tau)\right\|_{L^{\infty}}^{p-1}\left\|u(\cdot, \cdot, \tau)\right\|_{L^{2}}\left\|u_{x}(\cdot, \cdot, \tau)\right\|_{L^{2}}  \nonumber \\
&\le CE_{0}^{p-1}(1+\tau)^{-\frac{3(p-1)}{4}}K(\tau)(1+\tau)^{-\frac{1}{4}}K(\tau)(1+\tau)^{-\frac{3}{4}}\nonumber \\
&\le CE_{0}^{p-1}K(\tau)^{2}(1+\tau)^{-\frac{3p+1}{4}}. \label{dxu^{p+1}-L1-pre}
\end{align}
Moreover, recalling $\p_{x}\p_{y}(u^{p+1})=p(p+1)u^{p-1}u_{x}u_{y}+(p+1)u^{p}u_{xy}$ and \eqref{apriori-H21}, we analogously have
\begin{align}
\left\|\p_{x}\p_{y}\left(u^{p+1}(\cdot, \cdot, \tau)\right)\right\|_{L^{1}}
&\le C\left\|u(\cdot, \cdot, \tau)\right\|_{L^{\infty}}^{p-1}\left\|u_{x}(\cdot, \cdot, \tau)\right\|_{L^{2}}\left\|u_{y}(\cdot, \cdot, \tau)\right\|_{L^{2}} \nonumber \\
&\ \ \ +C\left\|u(\cdot, \cdot, \tau)\right\|_{L^{\infty}}^{p-1}\left\|u(\cdot, \cdot, \tau)\right\|_{L^{2}}\left\|u_{xy}(\cdot, \cdot, \tau)\right\|_{L^{2}}  \nonumber \\
&\le CE_{0}^{p-1}(1+\tau)^{-\frac{3(p-1)}{4}}K(\tau)(1+\tau)^{-\frac{3}{4}}\left\|u_{0}\right\|_{H^{2, 1}} \nonumber \\
&\ \ \ +CE_{0}^{p-1}(1+\tau)^{-\frac{3(p-1)}{4}}K(\tau)(1+\tau)^{-\frac{1}{4}}\left\|u_{0}\right\|_{H^{2, 1}} \nonumber\\
&\le CE_{0}^{p-1}K(\tau)(1+\tau)^{-\frac{3p-2}{4}}.  \label{dxdyu^{p+1}-L1-pre}
\end{align}
Combining \eqref{dxu^{p+1}-L1-pre} and \eqref{dxdyu^{p+1}-L1-pre}, we obtain  
\begin{align}
\left\|\p_{x}\left(u^{p+1}(\cdot, \cdot, \tau)\right)\right\|_{L^{1}}  + \left\|\p_{x}\p_{y}\left(u^{p+1}(\cdot, \cdot, \tau)\right)\right\|_{L^{1}}
\le CE_{0}^{p-1}\left(K(\tau)+K(\tau)^{2}\right)(1+\tau)^{-\frac{3p-2}{4}}. \label{dxu^{p+1}+dxdyu^{p+1}-L1}
\end{align}
Hence, we can see from \eqref{DEF-J(t)s} and \eqref{dxu^{p+1}+dxdyu^{p+1}-L1} that 
\begin{align}
J_{2}(t)&\le CE_{0}^{p-1}\left(\mathcal{K}(t)+\mathcal{K}(t)^{2}\right)(1+t)^{\frac{3}{4}}\int_{t/2}^{t}(1+t-\tau)^{-\frac{3}{4}}(1+\tau)^{-\frac{3p-2}{4}}d\tau \nonumber \\
&\le CE_{0}^{p-1}\left(\mathcal{K}(t)+\mathcal{K}(t)^{2}\right)(1+t)^{\frac{3}{4}}(1+t)^{-\frac{3p-2}{4}}(1+t)^{\frac{1}{4}} \nonumber \\
&\le CE_{0}^{p-1}\left(\mathcal{K}(t)+\mathcal{K}(t)^{2}\right)(1+t)^{-\frac{3(p-2)}{4}}, \ \ t\ge0. \label{J2-est}
\end{align}

Finally, let us deal with $J_{3}(t)$. Now, we need to derive the $L^{2}$-decay estimates for the nonlinear term $u^{p+1}(x, y, t)$. 
Here, noticing $u^{p+1}=u^{p}u$ and $\p_{x}(u^{p+1})=(p+1)u^{p}u_{x}$, we obtain from \eqref{u-sol-decay-Linf} and \eqref{DEF-K(t)} that 
\begin{align}
&\left\|\p_{x}^{l}\left(u^{p+1}(\cdot, \cdot, \tau)\right)\right\|_{L^{2}} 
\le C\left\|u(\cdot, \cdot, \tau)\right\|_{L^{\infty}}^{p}\left\|\p_{x}^{l}u(\cdot, \cdot, \tau)\right\|_{L^{2}} \nonumber \\
&\le CE_{0}^{p}(1+\tau)^{-\frac{3p}{4}}K(\tau)(1+\tau)^{-\frac{1}{4}-\frac{l}{2}}
\le CE_{0}^pK(\tau)(1+\tau)^{-\frac{3p+1}{4}-\frac{l}{2}}, \ \ l=0, 1. \label{u^{p+1}-L2-pre2}
\end{align}
Therefore, from \eqref{DEF-Hl(t)} and \eqref{u^{p+1}-L2-pre2}, we obtain 
\begin{align}
G_{l}(t)
\le CE_{0}^p\mathcal{K}(t)\left(\int_{0}^{t}e^{-\mu(t-\tau)}(1+\tau)^{-\frac{3p+1}{2}-l}d\tau\right)^{\frac{1}{2}} 
\le CE_{0}^p\mathcal{K}(t)(1+t)^{-\frac{3p+1}{4}-\frac{l}{2}}, \ \ l=0, 1. \label{est-Hl(t)}
\end{align}
Thus, summing up \eqref{DEF-J(t)s} and \eqref{est-Hl(t)}, we can see that 
\begin{align}
J_{3}(t)=\sum_{l=0}^{1}(1+t)^{\frac{1}{4}+\frac{l}{2}}G_{l}(t)
\le CE_{0}^p\mathcal{K}(t)(1+t)^{-\frac{3p}{4}}, \ \ t\ge0. \label{est-J3}
\end{align}

Summarizing up \eqref{K(t)-est-split}, \eqref{L2-decay-est-I0}, \eqref{DEF-J(t)s}, \eqref{J1-est}, \eqref{J2-est} and \eqref{est-J3}, we eventually arrive at 
\begin{align}
K(t)&\le I_{0}(t)+J_{1}(t)+J_{2}(t)+J_{3}(t) \nonumber \\
&\le CE_{1} + CE_{0}^{p-1}\left(\mathcal{K}(t)+\mathcal{K}(t)^{2}\right)(1+t)^{-\frac{1}{2}}\log(2+t) \nonumber\\
&\ \ \ + CE_{0}^{p-1}\left(\mathcal{K}(t)+\mathcal{K}(t)^{2}\right)(1+t)^{-\frac{3(p-2)}{4}}
+ CE_{0}^p\mathcal{K}(t)(1+t)^{-\frac{3p}{4}} \nonumber \\
&\le C\left\{E_{1}+E_{0}\left(\mathcal{K}(t)+\mathcal{K}(t)^{2}\right)\right\}, \ \ t\ge 0, \label{K-est-final}
\end{align}
for $p\ge 2$ if $E_0\le 1$ holds. 

Finally, let us derive the desired estimate \eqref{u-sol-decay-L2}. By virtue of \eqref{K-est-final}, we can see that the following estimate is true: 
\[
\mathcal{K}(t)\le C\left\{E_{1}+E_{0}\left(\mathcal{K}(t)+\mathcal{K}(t)^{2}\right)\right\}, \ \ t\ge0. 
\]
Therefore, under the smallness assumption of $E_{1}$\ $(=E_0+\left\|\p_yu_0\right\|_{L^1})$, in the same way to get \eqref{est-M(t)}, \eqref{est-N(t)} and \eqref{H(t)-bdd}, applying the continuity argument again, we arrive at the following estimate:
\begin{equation}\label{K(t)-bdd}
\mathcal{K}(t)\le CE_{1}, \ \ t\ge0. 
\end{equation}
This means that the proof of the desired result \eqref{u-sol-decay-L2} has been completed. 
\end{proof}
%%%%%%%%%%%%%%%%%%%%%%%%%%%%%%%%%%%%%%%%%%%%%%%%%%%%

%%%%%%%%%%%%%%%%%%%%%%%%%%%%%%%%%%%%%%%%%%%%%%%%%%%%
\begin{proof}[\rm{\bf{End of the proof of Theorem~\ref{thm.main-decay}}}]
Combining the results given in Theorems~\ref{thm.u-sol-decay-Linf} and \ref{thm.u-sol-decay-L2} (i.e. the $L^{\infty}$-decay estimate \eqref{u-sol-decay-Linf} and the $L^{2}$-decay estimate \eqref{u-sol-decay-L2}), 
all the statements mentioned in Theorem~\ref{thm.main-decay} have been proven. 
\end{proof}
%%%%%%%%%%%%%%%%%%%%%%%%%%%%%%%%%%%%%%%%%%%%%%%%%%%%

%%%%%%%%%%%%%%%%%%%%%%%%%%%%%%%%%%%%%%%%%%%%%%%%%%%%
\section{Approximation formula and lower bound for the solutions}  
%%%%%%%%%%%%%%%%%%%%%%%%%%%%%%%%%%%%%%%%%%%%%%%%%%%%

In this section, we give the proofs of Theorem~\ref{thm.main-approximation} (approximation formula) and Theorem~\ref{thm.main-u-est-lower} (lower bound for the solution $u(x, y, t)$). First, we shall prove the approximation formula for the solutions to the linear and the nonlinear problems, in the next subsection. After that, let us derive the lower bound for the $L^{\infty}$-norm of the solutions. 

%%%%%%%%%%%%%%%%%%%%%%%%%%%%%%%%%%%%%%%%%%%%%%%%%%%%
\subsection{Approximation formula for the solutions}  
%%%%%%%%%%%%%%%%%%%%%%%%%%%%%%%%%%%%%%%%%%%%%%%%%%%%

First in this subsection, we shall construct an approximation formula for the solution to \eqref{LZKB}. 
\begin{comment}
In order to doing that, we shall recall the following functions defined in the introduction: 
\begin{align}
&V(x, y, t) := t^{ -\frac{3}{4} } V_{*} \left( xt^{-\frac{1}{2}}, yt^{-\frac{1}{4}} \right), \ \ (x, y)\in \R^{2}, \ \ t>0,  \tag{\ref{DEF-V}}\\
&V_{*}(x, y) := \frac{ 1 }{ 4\pi^{ \frac{3}{2} }\mu^{\frac{1}{4}} } \int_{0}^{\infty} r^{-\frac{3}{4} }e^{ -r } \cos \left( x \sqrt{\frac{r}{\mu}} -\frac{y^{2}}{4}\sqrt{\frac{\mu}{r}} +\frac{\pi}{4}\right) dr, \ \ (x, y)\in \R^{2}.  \tag{\ref{DEF-V*}}
\end{align}
Then, let us consider the following Cauchy problem which is also introduced in the introduction: 
\begin{align}\tag{\ref{approx-CP}}
\begin{split}
& v_{t} + v_{yyx} -\mu v_{xx}=0, \ \ (x, y) \in \R^{2}, \ t>0, \\
& v(x, y, 0) = u_{0}(x, y), \ \ (x, y) \in \R^{2}. 
\end{split}
\end{align}
\end{comment}
Before doing that, we would like to show that the integral kernel of the solution to \eqref{approx-CP} is given by $V(x, y, t)$ defined in \eqref{DEF-V}. 
Now, we transform the integral kernel for \eqref{approx-CP}, in the same way that we get \eqref{re-DEF-U}, and use the change of variables, 
then it follows from \eqref{DEF-V} and \eqref{DEF-V*} that 
\begin{align}
\frac{1}{2\pi}\mathcal{F}^{-1} \left[ e^{ -\mu t \xi^{2} + it\xi \eta^{2} } \right] (x, y) 
& = \frac{t^{ -\frac{1}{2} }}{ 4\pi^{ \frac{3}{2} }} \int_{\R} |\xi|^{ -\frac{1}{2} } e^{ -\mu t \xi^{2} -\frac{iy^{2}}{4t\xi} + \frac{i\pi}{4}\mathrm{sgn} \xi +ix\xi} d\xi \nonumber \\
%& = \frac{t^{ -\frac{1}{2} } }{ 4\pi^{ \frac{3}{2} }} \left( \int_{0}^{\infty} |\xi|^{ -\frac{1}{2} } e^{ -\mu t \xi^{2}  -\frac{iy^{2}}{4t\xi} + \frac{i\pi}{4}+ix\xi} d\xi 
%+ \int_{-\infty}^{0} |\xi|^{ -\frac{1}{2} } e^{ -\mu t \xi^{2} -\frac{iy^{2}}{4t\xi} -\frac{i\pi}{4}+ix\xi} d\xi \right) \nonumber\\
%& = \frac{  t^{ -\frac{3}{4} } }{ 8\pi^{ \frac{3}{2} } \mu^{\frac{1}{4}}} \left( \int_{0}^{\infty} r^{-\frac{3}{4} }e^{ -r -\frac{iy^{2}}{4}\sqrt{\frac{\mu}{rt}} +\frac{i\pi}{4} +ix \sqrt{\frac{r}{\mu t}}} dr + \int_{0}^{\infty} r^{-\frac{3}{4} }e^{ -r +\frac{iy^{2}}{4}\sqrt{\frac{\mu}{rt}} -\frac{i\pi}{4} -ix \sqrt{\frac{r}{\mu t}}} dr \right)  \nonumber\\
& = \frac{  t^{ -\frac{3}{4} } }{ 8\pi^{ \frac{3}{2} } \mu^{\frac{1}{4}} } \int_{0}^{\infty} r^{-\frac{3}{4} }e^{ -r } \left( e^{-\frac{iy^{2}}{4}\sqrt{\frac{\mu}{rt}} +\frac{i\pi}{4} +ix \sqrt{\frac{r}{\mu t}}} + e^{\frac{iy^{2}}{4}\sqrt{\frac{\mu}{rt}} -\frac{i\pi}{4} -ix \sqrt{\frac{r}{\mu t}}} \right) dr \nonumber\\
& = \frac{  t^{ -\frac{3}{4} } }{ 4\pi^{ \frac{3}{2} } \mu^{\frac{1}{4}} } \int_{0}^{\infty} r^{-\frac{3}{4} }e^{ -r } \cos \left( x \sqrt{\frac{r}{\mu t}} -\frac{y^{2}}{4}\sqrt{\frac{\mu}{r t}} 
+\frac{\pi}{4}\right) dr \nonumber\\
&=t^{ -\frac{3}{4} } V_{*} \left( xt^{-\frac{1}{2}}, yt^{-\frac{1}{4}} \right)
=V(x, y, t). \label{re-V}
\end{align}
Thus, we can see that $V(x, y, t)$ is the integral kernel of the solution to the Cauchy problem \eqref{approx-CP}. 
%Therefore, the solution $v(x, y, t)$ to \eqref{approx-CP} can be written by 
%\begin{equation}\tag{\ref{approx-sol}}
%v(x, y, t)=(V(t)*u_{0})(x, y), \ \ (x, y) \in \R^{2}, \ t>0. 
%\end{equation}
From the definition of $V(x, y, t)$ and the above expression \eqref{re-V}, 
we can immediately obtain the following $L^{\infty}$-decay estimates similar to Proposition~\ref{prop.L-decay-U} and Corollary~\ref{cor.L-decay-LZKB}: 
%%%%%%%%%%%%%%%%%%%%%%%%%%%%%%%%%%%%%%%%%%%%%%%%%%%%
\begin{prop}\label{prop.L-decay-V}
Let $l$ be a non-negative integer. Then, we have 
\begin{equation}\label{L-decay-V}
\left\| \p_{x}^{l} V(\cdot, \cdot, t) \right\|_{L^{\infty}} \le C t^{ -\frac{3}{4} -\frac{l}{2} }, \ \ t>0, 
\end{equation}
where $V(x, y, t)$ is defined by \eqref{DEF-V}. 
\end{prop}
%%%%%%%%%%%%%%%%%%%%%%%%%%%%%%%%%%%%%%%%%%%%%%%%%%%%
%%%%%%%%%%%%%%%%%%%%%%%%%%%%%%%%%%%%%%%%%%%%%%%%%%%%
\begin{cor}\label{cor.L-decay-approx-CP}
Suppose that there exists a non-negative integer $j$ such that $\p_{y}^{j}u_{0}\in L^{1}(\R^{2})$. 
Then, for all non-negative integer $l$, we have  
\begin{equation}\label{L-decay-approx-CP}
\left\| \p_{x}^{l}\p_{y}^{j}\left(V(t)*u_{0}\right)\right\|_{L^{\infty}} \le Ct^{-\frac{3}{4}-\frac{l}{2}}\left\|\p_{y}^{j}u_{0}\right\|_{L^{1}}, \ \ t>0,  
\end{equation}
where $V(x, y, t)$ is defined by \eqref{DEF-V}. 
\end{cor}
%%%%%%%%%%%%%%%%%%%%%%%%%%%%%%%%%%%%%%%%%%%%%%%%%%%%

Moreover, we are able to prove that $U(x, y, t)$ is well approximated by $V(x, y, t)$ as follows:  
%%%%%%%%%%%%%%%%%%%%%%%%%%%%%%%%%%%%%%%%%%%%%%%%%%%%
\begin{prop}\label{prop.L-ap-U}
Let $l$ be a non-negative integer. Then, we have 
\begin{equation}\label{L-ap-U}
\left\| \p_{x}^{l} \left( U(\cdot, \cdot, t) - V(\cdot, \cdot, t) \right) \right\|_{L^{\infty}} \le C t^{ -\frac{5}{4} -\frac{l}{2} }, \ \ t>0, 
\end{equation}
where $U(x, y, t)$ and $V(x, y, t)$ are defined by \eqref{DEF-U} and \eqref{DEF-V}, respectively. 
\end{prop}
%%%%%%%%%%%%%%%%%%%%%%%%%%%%%%%%%%%%%%%%%%%%%%%%%%%%
%%%%%%%%%%%%%%%%%%%%%%%%%%%%%%%%%%%%%%%%%%%%%%%%%%%%
\begin{proof}
From the mean value theorem, there exists $\theta = \theta(\xi, t)\in (0, 1)$ such that $e^{ it\xi^{3} } = 1 + it\xi^{3} e^{ i\theta t\xi^{3} }$ holds.
Therefore, it follows from \eqref{re-DEF-U} that 
\begin{align*}
U(x, y, t) 
& = \frac{t^{ -\frac{1}{2} }}{ 4\pi^{ \frac{3}{2} }} \int_{\R} |\xi|^{ -\frac{1}{2} } e^{ -\mu t \xi^{2} -\frac{iy^{2}}{4t\xi} + \frac{i\pi}{4}\mathrm{sgn} \xi + ix\xi} d\xi \nonumber \\
&\ \ \ + \frac{i t^{ \frac{1}{2} }}{ 4\pi^{ \frac{3}{2} }} \int_{\R} \xi^{3} |\xi|^{ -\frac{1}{2} -\mu t \xi^{2} + i\theta t\xi^{3} -\frac{iy^{2}}{4t\xi} + \frac{i\pi}{4}\mathrm{sgn} \xi +ix\xi} d\xi.  
\end{align*}
Then, noticing \eqref{re-V}, it means that 
\begin{equation}\label{rere-DEF-U}
U(x,y,t)-V(x,y,t)=\frac{i t^{ \frac{1}{2} }}{ 4\pi^{ \frac{3}{2} }} \int_{\R} \xi^{3} |\xi|^{ -\frac{1}{2} -\mu t \xi^{2} + i\theta t\xi^{3} -\frac{iy^{2}}{4t\xi} + \frac{i\pi}{4}\mathrm{sgn} \xi +ix\xi} d\xi =:R(x,y,t). 
\end{equation}
In addition, we can evaluate the remainder term $R(x, y, t)$ as follows: 
\begin{align}
\left|\p_{x}^{l}R(x, y, t)\right| 
& \le \frac{t^{ \frac{1}{2} }}{ 4\pi^{ \frac{3}{2} }} \int_{\R} |\xi|^{ \frac{5}{2}+l } e^{ -\mu t \xi^{2} } d\xi 
%= \frac{t^{ \frac{1}{2} }}{ 2\pi^{ \frac{3}{2} }} \int_{0}^{\infty} \xi^{ \frac{5}{2}+l } e^{ -\mu t \xi^{2} } d\xi  \nonumber \\
%& = \frac{  t^{ -\frac{5}{4} -\frac{l}{2} } }{ 4\pi^{ \frac{3}{2} }\mu^{\frac{7+2l}{4}} } \int_{0}^{\infty} r^{ \frac{3}{4} +\frac{l}{2} }e^{ -r } dr 
= \frac{ \Gamma \left(  \frac{7+2l}{4} \right) }{4\pi^{ \frac{3}{2} }\mu^{\frac{7+2l}{4}}} t^{ -\frac{5}{4} -\frac{l}{2} }, \ \ (x, y) \in \R^{2}, \ t>0. \label{U-R-est}
\end{align}
Thus, combining \eqref{rere-DEF-U} and \eqref{U-R-est}, we eventually arrive at the desired estimate \eqref{L-ap-U}. 
\end{proof}
%%%%%%%%%%%%%%%%%%%%%%%%%%%%%%%%%%%%%%%%%%%%%%%%%%%%

%%%%%%%%%%%%%%%%%%%%%%%%%%%%%%%%%%%%%%%%%%%%%%%%%%%%
\begin{cor}\label{cor.L-decay-LZKB-ap}
Suppose that there exists a non-negative integer $j$ such that $\p_{y}^{j}u_{0}\in L^{1}(\R^{2})$. 
Then, for all non-negative integer $l$, we have  
\begin{equation}\label{L-decay-LZKB-ap}
\left\| \p_{x}^{l}\p_{y}^{j}\left((U-V)(t)*u_{0}\right)(\cdot, \cdot) \right\|_{L^{\infty}} \le Ct^{-\frac{5}{4}-\frac{l}{2}}\left\|\p_{y}^{j}u_{0}\right\|_{L^{1}}, \ \ t>0, 
\end{equation}
where $U(x, y, t)$ and $V(x, y, t)$ are defined by \eqref{DEF-U} and \eqref{DEF-V}, respectively. 
\end{cor}
%%%%%%%%%%%%%%%%%%%%%%%%%%%%%%%%%%%%%%%%%%%%%%%%%%%%

By virtue of the above result Corollary~\ref{cor.L-decay-LZKB-ap}, we can construct the approximation formula for the solution to the nonlinear problem \eqref{ZKB}. Actually, Theorem~\ref{thm.main-approximation} can be proven as follows:  
%%%%%%%%%%%%%%%%%%%%%%%%%%%%%%%%%%%%%%%%%%%%%%%%%%%%
\begin{proof}[\rm{\bf{End of the proof of Theorem~\ref{thm.main-approximation}}}]
First, we shall evaluate the Duhamel term of the integral equation \eqref{integral-eq} again. 
Combining \eqref{Linf-decay-pre-I1}, \eqref{u^{p+1}-L1-int} and \eqref{H(t)-bdd}, we are able to see that 
\begin{align*}
\left\|\p_{x}^{l}\int_{0}^{t/2}\p_{x}U(t-\tau)*\left(u^{p+1}(\tau)\right)d\tau\right\|_{L^{\infty}} \le CE_{0}(1+t)^{-1-\frac{l}{2}}, \ \ t\ge0, \ l=0, 1. 
\end{align*}
Moreover, summing up \eqref{u^{p+1}-L2-pre} through \eqref{dxdyu^{p+1}-L2-pre} and \eqref{H(t)-bdd}, it follows from Proposition~\ref{prop.Duhamel-est-Linf} that 
\begin{align*}
\left\|\p_{x}^{l}\int_{t/2}^{t}\p_{x}U(t-\tau)*\left(u^{p+1}(\tau)\right)d\tau\right\|_{L^{\infty}} \le CE_{0}(1+t)^{-\frac{3p-1}{4}}, \ \ t\ge0, \ l=0, 1. 
\end{align*}
By virtue of the above two estimates, under the assumption $p>(4+2l)/3$, 
the following result has been established: 
\begin{align}
&\limsup_{t\to \infty}t^{\frac{3}{4}+\frac{l}{2}}\left\|\p_{x}^{l}\int_{0}^{t}\p_{x}U(t-\tau)*\left(u^{p+1}(\tau)\right)d\tau\right\|_{L^{\infty}} \nonumber \\
&\le C\lim_{t\to \infty}t^{\frac{3}{4}+\frac{l}{2}} (1+t)^{-\min\{1+\frac{l}{2}, \frac{3p-1}{4}\}}
\le C\lim_{t\to \infty}t^{-\min\{\frac{1}{4}, \frac{3p-4-2l}{4}\}}=0, \ \ l=0, 1. \label{Duhamel-est-final}
\end{align}

Now, it follows from \eqref{integral-eq} and \eqref{approx-sol} that 
\begin{align*}
\p_{x}^{l}\left(u(x, y, t)-v(x, y, t)\right)
=\p_{x}^{l}\left((U-V)(t)*u_{0}\right)(x, y)-\frac{\beta}{p+1}\p_{x}^{l}\int_{0}^{t}\p_{x}U(t-\tau)*\left(u^{p+1}(\tau)\right)d\tau. 
\end{align*}
Finally, applying Corollary~\ref{cor.L-decay-LZKB-ap} and using \eqref{Duhamel-est-final}, we are able to see that 
\begin{align*}
&\limsup_{t\to \infty}t^{\frac{3}{4}+\frac{l}{2}}\left\|\p_{x}^{l}\left(u(\cdot, \cdot, t)-v(\cdot, \cdot, t)\right)\right\|_{L^{\infty}}\\
&\le \left(\limsup_{t \to \infty} t^{ \frac{5}{4}+\frac{l}{2} } \left\|\p_{x}^{l}\left((U-V)(t)*u_{0}\right)(\cdot, \cdot)\right\|_{L^{\infty}} \right)  \left(\lim_{t \to \infty} t^{ -\frac{1}{2}}\right) \\
&\ \ \ \ +\lim_{t\to \infty}t^{\frac{3}{4}+\frac{l}{2}}\left\|\p_{x}^{l}\int_{0}^{t}\p_{x}U(t-\tau)*\left(u^{p+1}(\tau)\right)d\tau\right\|_{L^{\infty}}=0, \ \ l=0, 1. 
\end{align*}
Thus, the proof of the desired formula \eqref{u-sol-approximation} has been completed. 
\end{proof}
%%%%%%%%%%%%%%%%%%%%%%%%%%%%%%%%%%%%%%%%%%%%%%%%%%%%

We have already seen that the solution $u(x, y, t)$ to \eqref{ZKB} is well approximated by the solution $v(x, y, t)$ to \eqref{approx-CP}. 
Next in this subsection, we shall prove an useful approximation formula for the solution to \eqref{approx-CP}. In order to state such a result, let us define the following new functions: 
\begin{align}
& \mathcal{V}_{j}(x, y, t) := \int_{\R} V(x, y-w, t)  M_{j}(w) dw, \ \ (x, y)\in \R^{2}, \ t>0,  \label{DEF-mathV}\\
& M_{j}(y) := \int_{\R} \p_{y}^{j}u_{0}(x, y) dx, \ \ \p_{y}^{j}u_{0}(\cdot, y)\in L^{1}(\R), \ \mathrm{a.e.} \ y\in \R, \ j \in \mathbb{N}\cup \{0\}, \label{M(y)}
\end{align}
where $V(x, y, t)$ is defined by \eqref{DEF-V}. Then, we can show that $\p_{x}^{l}\p_{y}^{j}v(x, y, t)$ converges to $\p_{x}^{l}\mathcal{V}_{j}(x, y, t)$ as $t\to \infty$ in the $L^{\infty}$-sense, if the initial data $u_{0}(x, y)$ satisfies $\p_{y}^{j}u_{0} \in L^{1}(\R^{2})$ for some $j\in \mathbb{N}\cup \{0\}$. More precisely, the following approximation formula can be established: 
%%%%%%%%%%%%%%%%%%%%%%%%%%%%%%%%%%%%%%%%%%%%%%%%%%%%
\begin{thm}\label{thm.L-ap-LZKB}
Suppose that there exists a non-negative integer $j$ such that $\p_{y}^{j}u_{0}\in L^{1}(\R^{2})$. 
Then, for all non-negative integer $l$, we have  
\begin{equation}\label{L-ap-LZKB}
\lim_{t \to \infty} t^{ \frac{3}{4}+\frac{l}{2} } \left\| \p_{x}^{l}\p_{y}^{j}(V(t)*u_{0})(\cdot, \cdot) - \p_{x}^{l}\mathcal{V}_{j}(\cdot, \cdot, t) \right\|_{L^{\infty}} = 0,
\end{equation}
where $V(x, y, t)$ and $\mathcal{V}_{j}(x, y, t)$ are defined by \eqref{DEF-V} and \eqref{DEF-mathV}, respectively. 
\end{thm}
%%%%%%%%%%%%%%%%%%%%%%%%%%%%%%%%%%%%%%%%%%%%%%%%%%%%
%%%%%%%%%%%%%%%%%%%%%%%%%%%%%%%%%%%%%%%%%%%%%%%%%%%%
\begin{proof}
First, since $\p_{y}^{j}u_{0}\in L^{1}(\R^{2})$, for any $\e>0$, there exists a constant $L=L(\e)>0$ such that 
\begin{equation}\label{L1-far}
\int_{|(x, y)| \ge L} \left|\p_{y}^{j}u_{0}(x, y)\right| dxdy < \e.
\end{equation}
Therefore, from \eqref{DEF-mathV} and \eqref{M(y)}, splitting the integral and applying the mean value theorem, there exists $\theta=\theta(x, z)\in (0, 1)$ such that  
\begin{align}
&\p_{x}^{l}\p_{y}^{j}(V(t)*u_{0})(x, y) -\p_{x}^{l}\mathcal{V}_{j}(x, y, t) \nonumber \\
& = \int_{\R^{2}} \p_{x}^{l}V(x-z, y-w, t)\p_{w}^{j}u_{0}(z, w)dzdw - \int_{\R}\p_{x}^{l}V(x, y-w, t) M_{j}(w) dw \nonumber \\
& = \left( \int_{|(z, w)| \le L}+\int_{|(z, w)| \ge L} \right)\p_{x}^{l}\left( V(x-z, y-w, t)-V(x, y-w, t) \right)\p_{w}^{j}u_{0}(z, w)dzdw \nonumber \\
& = \int_{|(z, w)| \le L} (-z)\p_{x}^{l+1}V(x-\theta z, y-w, t)\p_{w}^{j}u_{0}(z, w)dzdw \nonumber \\
&\ \ \ \ + \int_{|(z, w)| \ge L} \p_{x}^{l}\left( V(x-z, y-w, t)-V(x, y-w, t)\right)\p_{w}^{j}u_{0}(z, w)dzdw \nonumber \\
& =: J_{1}(x, y, t) + J_{2}(x, y, t). \label{L-ap-split}
\end{align}
%We omit the rest of the proof, since it can be shown by the same way to get Theorem~2.8 in \cite{FH23}.  
%\begin{comment}

In what follows, we shall evaluate $J_{1}(x, y, t)$ and $J_{2}(x, y, t)$. For $J_{1}(x, y, t)$, it follows from \eqref{L-ap-split} and Proposition~\ref{prop.L-decay-V} that 
\begin{align}
\left| J_{1}(x, y, t) \right| 
& \le \int_{|(z, w)| \le L} |z|\,\left|\p_{x}^{l+1}V(x-\theta z, y-w, t)\right|\,\left|\p_{w}^{j}u_{0}(z, w)\right|dzdw \nonumber \\
& \le Ct^{ -\frac{3}{4} -\frac{l+1}{2} } \int_{|(z, w)| \le L} |(z, w)|\,\left|\p_{w}^{j}u_{0}(z, w)\right|dzdw \nonumber \\
& \le CL \left\| \p_{y}^{j}u_{0} \right\|_{L^1}t^{ -\frac{5}{4} -\frac{l}{2} }, \ \ (x, y) \in \R^{2}, \ t>0. \label{L-ap-est-J1}
\end{align}
Next, let us treat $J_{2}(x, y, t)$. From \eqref{L-ap-split}, Proposition~\ref{prop.L-decay-V} and \eqref{L1-far}, we get
\begin{align}
\left| J_{2}(x, y, t) \right| 
& \le \int_{|(z, w)| \ge L} \left( \left|\p_{x}^{l}V(x-z, y-w, t)\right| + \left|\p_{x}^{l}V(x, y-w, t)\right| \right) \left|\p_{w}^{j}u_{0}(z, w)\right|dzdw \nonumber \\
& \le Ct^{ -\frac{3}{4} -\frac{l}{2} } \int_{|(z, w)| \ge L} \left|\p_{w}^{j}u_{0}(z, w)\right|dzdw \nonumber \\
&\le C\e t^{ -\frac{3}{4} -\frac{l}{2} }, \ \ (x, y) \in \R^{2}, \ t>0. \label{L-ap-est-J2}
\end{align}

Combining \eqref{L-ap-split}, \eqref{L-ap-est-J1} and \eqref{L-ap-est-J2}, we can see that 
\[
\left\| \p_{x}^{l}\p_{y}^{j}(V(t)*u_{0})(\cdot, \cdot) - \p_{x}^{l}\mathcal{V}_{j}(\cdot, \cdot, t) \right\|_{L^{\infty}} 
\le Ct^{ -\frac{5}{4}-\frac{l}{2} } + C\e t^{ -\frac{3}{4}-\frac{l}{2} }, \ \ t>0.
\] 
Therefore, we eventually arrive at
\[
\limsup_{t \to \infty} t^{ \frac{3}{4}+\frac{l}{2} } \left\| \p_{x}^{l}\p_{y}^{j}(V(t)*u_{0})(\cdot, \cdot) - \p_{x}^{l}\mathcal{V}_{j}(\cdot, \cdot, t) \right\|_{L^{\infty}} \le C\e.
\] 
Thus, we get the desired formula \eqref{L-ap-LZKB}, because $\e>0$ can be chosen arbitrarily small. 
%\end{comment}
\end{proof}
%%%%%%%%%%%%%%%%%%%%%%%%%%%%%%%%%%%%%%%%%%%%%%%%%%%%

%%%%%%%%%%%%%%%%%%%%%%%%%%%%%%%%%%%%%%%%%%%%%%%%%%%%
\subsection{Lower bound for the solutions}  
%%%%%%%%%%%%%%%%%%%%%%%%%%%%%%%%%%%%%%%%%%%%%%%%%%%%

The above formula \eqref{L-ap-LZKB} gives us an approximation for the solution $v(x, y, t)$ to \eqref{approx-CP}, 
i.e. we can think that the main part of $\p_{x}^{l}\p_{y}^{j}v(x, y, t)$ is governed by $\p_{x}^{l}\mathcal{V}_{j}(x, y, t)$. 
Indeed, this formula helps us to derive the lower bound for the $L^{\infty}$-norm of the solution $u(x, y, t)$ to \eqref{ZKB}. 
In the rest of this section, we would like to explain such results. To doing that, let us show the following fact: 
%%%%%%%%%%%%%%%%%%%%%%%%%%%%%%%%%%%%%%%%%%%%%%%%%%%%
\begin{prop}\label{prop.mathV-lower}
Suppose that there exists a non-negative integer $j$ such that $\p_{y}^{j}u_{0}\in L^{1}(\R^{2})$. 
Then, for all non-negative integer $l$, we have  
\begin{equation}\label{mathV-lower}
\liminf_{t\to \infty}t^{\frac{3}{4}+\frac{l}{2}}\left\| \p_{x}^{l} \mathcal{V}_{j}(\cdot, \cdot, t) \right\|_{L^{\infty}} 
\ge \frac{\Gamma \left(\frac{1+2l}{4}\right)}{{ 2^{\frac{5}{2}}\pi^{ \frac{3}{2} } \mu^{\frac{1+2l}{4}} }}\left|\int_{\R}M_{j}(y)dy\right|, 
\end{equation}
where $\mathcal{V}_{j}(x, y, t)$ and $M_{j}(y)$ are defined by \eqref{DEF-mathV} and \eqref{M(y)}, respectively.  
\end{prop}
%%%%%%%%%%%%%%%%%%%%%%%%%%%%%%%%%%%%%%%%%%%%%%%%%%%%
%%%%%%%%%%%%%%%%%%%%%%%%%%%%%%%%%%%%%%%%%%%%%%%%%%%%
\begin{proof}
First, we take $(x, y)=(0, 0)$ in \eqref{DEF-mathV}, then we have from \eqref{DEF-V} and \eqref{DEF-V*} that 
\begin{align}
&\left\| \p_{x}^{l}\mathcal{V}_{j}(\cdot, \cdot, t) \right\|_{L^{\infty}} 
\ge \left| (\p_{x}^{l}\mathcal{V}_{j})(0, 0, t) \right| 
= \left| \int_{\R} (\p_{x}^{l} V)(0, -w, t)  M_{j}(w) dw \right| \nonumber \\
& = \frac{ t^{-\frac{3}{4}-\frac{l}{2}} }{ 4\pi^{ \frac{3}{2} } \mu^{\frac{1+2l}{4}} } \left| \int_{\R} \left( \int_{0}^{\infty} r^{-\frac{3}{4} +\frac{l}{2}}e^{ -r } \cos \left( -\frac{w^{2}}{4}\sqrt{\frac{\mu}{rt}} + \frac{(1+2l)\pi}{4}\right) dr \right) M_{j}(w) dw \right|. \label{mathV-lower-pre}
\end{align}
Here, we used the fact $\frac{d^{l}}{d\theta^{l}}\cos \theta =\cos (\theta+(l\pi)/2)$. 
%\[
%\cos^{(l)} \left( x\sqrt{\frac{r}{\mu t}}-\frac{(y-w)^{2}}{4}\sqrt{\frac{\mu}{rt}} +\frac{\pi}{4} \right) = \left(\frac{r}{\mu t}\right)^{\frac{l}{2}} \cos \left( x\sqrt{\frac{r}{t}}-\frac{(y-w)^{2}}{4}\sqrt{\frac{\mu}{rt}} +\frac{\pi}{4} + \frac{l\pi}{2}\right). 
%\]

For each $r>0$ and $w\in \R$, we have 
$
\displaystyle \lim_{t \to \infty} \cos ( -(w^{2}\sqrt{\mu})/(4\sqrt{rt}) + (1+2l)/4)=\cos((1+2l)/4)
$. 
Also, we note that 
\[
\left|\left( \int_{0}^{\infty} r^{-\frac{3}{4} +\frac{l}{2}}e^{ -r } \cos \left( -\frac{w^{2}}{4}\sqrt{\frac{\mu}{rt}} + \frac{(1+2l)\pi}{4}\right) dr \right) M_{j}(w)\right|
\le  \Gamma \left(\frac{1+2l}{4}\right) \left|M_{j}(w)\right|
\]
and $M_{j}\in L^1(\R)$. Therefore, applying the Lebesgue dominated convergence theorem, we obtain 
\begin{align}
&\lim_{t\to \infty}\left|\int_{\R}\left( \int_{0}^{\infty} r^{-\frac{3}{4} +\frac{l}{2}}e^{ -r } \cos \left( -\frac{w^{2}}{4}\sqrt{\frac{\mu}{rt}} + \frac{(1+2l)\pi}{4}\right) dr \right) M_{j}(w) dw\right| \nonumber \\
&=
\left|\int_{\R}\left(\int_{0}^{\infty}r^{\frac{1+2l}{4}-1}e^{-r}\cos \left(\frac{(1+2l)\pi}{4}\right)dr\right)M_{j}(w)dw\right| 
%= \Gamma \left(\frac{1+2l}{4}\right) \left|\cos \left(\frac{(1+2l)\pi}{4}\right)\right|\left|\int_{\R}M_{j}(w)dw\right| 
= \frac{1}{\sqrt{2}}\Gamma \left(\frac{1+2l}{4}\right) \left|\int_{\R}M_{j}(w)dw\right|. \label{limit}
\end{align}
Therefore, from \eqref{mathV-lower-pre} and \eqref{limit}, we can immediately get the desired result \eqref{mathV-lower}. 
\end{proof}
%%%%%%%%%%%%%%%%%%%%%%%%%%%%%%%%%%%%%%%%%%%%%%%%%%%%

By virtue of Theorem~\ref{thm.L-ap-LZKB} and Proposition~\ref{prop.mathV-lower}, we can derive the lower bound for the $L^{\infty}$-norm of the linear solution $\tilde{u}(x, y, t)$ to \eqref{LZKB}. Actually, the following estimate is true:  
%%%%%%%%%%%%%%%%%%%%%%%%%%%%%%%%%%%%%%%%%%%%%%%%%%%%
\begin{cor}\label{cor.L-lower-LZKB}
Suppose that there exists a non-negative integer $j$ such that $\p_{y}^{j}u_{0}\in L^{1}(\R^{2})$. 
Then, for all non-negative integer $l$, we have  
\begin{equation}\label{L-lower-LZKB}
\liminf_{t\to \infty}t^{\frac{3}{4}+\frac{l}{2}}\left\| \p_{x}^{l}\p_{y}^{j}(U(t)*u_{0})(\cdot, \cdot) \right\|_{L^{\infty}} 
 \ge \frac{\Gamma \left(\frac{1+2l}{4}\right)}{{ 2^{\frac{5}{2}}\pi^{ \frac{3}{2} } \mu^{\frac{1+2l}{4}} }} \left|\int_{\R^{2}}\p_{y}^{j}u_{0}(x, y)dxdy\right|,
 \end{equation}
where $U(x, y, t)$ is defined by \eqref{DEF-U}. 
\end{cor}
%%%%%%%%%%%%%%%%%%%%%%%%%%%%%%%%%%%%%%%%%%%%%%%%%%%%
%%%%%%%%%%%%%%%%%%%%%%%%%%%%%%%%%%%%%%%%%%%%%%%%%%%%
\begin{proof}
Combining Corollary~\ref{cor.L-decay-LZKB-ap}, Theorem~\ref{thm.L-ap-LZKB} and Proposition~\ref{prop.mathV-lower}, we obtain 
\begin{align*}
&\liminf_{t\to \infty}t^{\frac{3}{4}+\frac{l}{2}}\left\| \p_{x}^{l}\p_{y}^{j}(U(t)*u_{0})(\cdot, \cdot) \right\|_{L^{\infty}}  \\
&\ge \liminf_{t\to \infty}t^{\frac{3}{4}+\frac{l}{2}}\left\|\p_{x}^{l}\mathcal{V}_{j}(\cdot, \cdot, t)\right\|_{L^{\infty}}
-\lim_{t\to \infty}t^{\frac{3}{4}+\frac{l}{2}}\left\| \p_{x}^{l}\p_{y}^{j}(V(t)*u_{0})(\cdot, \cdot) -\p_{x}^{l}\mathcal{V}_{j}(\cdot, \cdot, t)\right\|_{L^{\infty}} \\
&\ \ \ \ -\left(\limsup_{t\to \infty}t^{\frac{5}{4}+\frac{l}{2}}\left\| \p_{x}^{l}\p_{y}^{j}\left(\left(U-V\right)(t)*u_{0}\right)(\cdot, \cdot) \right\|_{L^{\infty}}\right) \left(\lim_{t\to \infty}t^{-\frac{1}{2}}\right)  \\
&\ge  \frac{\Gamma \left(\frac{1+2l}{4}\right)}{{ 2^{\frac{5}{2}}\pi^{ \frac{3}{2} } \mu^{\frac{1+2l}{4}} }}\left|\int_{\R^{2}}\p_{y}^{j}u_{0}(x, y)dxdy\right|. 
\end{align*}
Thus the desired result \eqref{L-lower-LZKB} has been established. 
\end{proof}
%%%%%%%%%%%%%%%%%%%%%%%%%%%%%%%%%%%%%%%%%%%%%%%%%%%%

Finally, we are able to derive the lower bound for the $L^{\infty}$-norm of the original solution $u(x, y, t)$ to the nonlinear problem \eqref{ZKB} as follows: 
%%%%%%%%%%%%%%%%%%%%%%%%%%%%%%%%%%%%%%%%%%%%%%%%%%%%
\begin{proof}[\rm{\bf{End of the proof of Theorem~\ref{thm.main-u-est-lower}}}]
Applying Corollary~\ref{cor.L-lower-LZKB} and \eqref{Duhamel-est-final} for \eqref{integral-eq}, we obtain
\begin{align*}
&\liminf_{t\to \infty}t^{\frac{3}{4}+\frac{l}{2}}\left\| \p_{x}^{l}u(\cdot, \cdot, t) \right\|_{L^{\infty}} \\
&\ge
\liminf_{t\to \infty}t^{\frac{3}{4}+\frac{l}{2}}\left\| \p_{x}^{l}(U(t)*u_{0})(\cdot, \cdot) \right\|_{L^{\infty}} 
- \frac{|\beta|}{p+1}\limsup_{t\to \infty}t^{\frac{3}{4}+\frac{l}{2}}\left\|\p_{x}^{l}\int_{0}^{t}\p_{x}U(t-\tau)*\left(u^{p+1}(\tau)\right)d\tau\right\|_{L^{\infty}} \\
 &\ge \frac{\Gamma \left(\frac{1+2l}{4}\right)}{{ 2^{\frac{5}{2}}\pi^{ \frac{3}{2} } \mu^{\frac{1+2l}{4}} }} \left|\int_{\R^{2}}u_{0}(x, y)dxdy\right|. 
 %=:c_{0}\left|\int_{\R^{2}}u_{0}(x, y)dxdy\right|. 
 \end{align*}
 Therefore, we can conclude that the desired estimate \eqref{u-est-lower} is true. 
\end{proof}
%%%%%%%%%%%%%%%%%%%%%%%%%%%%%%%%%%%%%%%%%%%%%%%%%%%%

%%%%%%%%%%%%%%%%%%%%%%%%%%%%%%%%%%%%%%%%%%%%%%%%%%%%
\section{Asymptotic profile for the solutions}  
%%%%%%%%%%%%%%%%%%%%%%%%%%%%%%%%%%%%%%%%%%%%%%%%%%%%

Finally in this section, we would like to derive the detailed asymptotic profile for the solution $u(x, y, t)$ to \eqref{ZKB}. 
More precisely, we shall prove Theorem~\ref{thm.main-u-asymptotic}. 
\begin{comment}
First, let us recall the following approximation Cauchy problem \eqref{approx-CP}: 
\begin{align}\tag{\ref{approx-CP}}
\begin{split}
& v_{t} + v_{yyx} -\mu v_{xx}=0, \ \ (x, y) \in \R^{2}, \ t>0, \\
& v(x, y, 0) = u_{0}(x, y), \ \ (x, y) \in \R^{2}. 
\end{split}
\end{align}
\end{comment}
By virtue of Theorem~\ref{thm.main-approximation}, we can see that the solution $u(x, y, t)$ to \eqref{ZKB} is well approximated by the solution $v(x, y, t)$ to \eqref{approx-CP}. Therefore, in order to prove Theorem~\ref{thm.main-u-asymptotic}, it is sufficient to derive the asymptotic profile for $v(x, y, t)$. To state such a result, we define the following new function: 
\begin{align}\label{DEF-psij}
\begin{split}
\psi_{j}(x, y, t):=&\,
\mathcal{M}\left[\p_{y}^{j}u_{0}\right]\left(yt^{-1}\right)V(-x, y, t), \\
=&\, \mathcal{M}\left[\p_{y}^{j}u_{0}\right]\left(yt^{-1}\right)V_{*} \left( -xt^{-\frac{1}{2}}, yt^{-\frac{1}{4}} \right)t^{ -\frac{3}{4} }, \ \ (x, y)\in \R^{2}, \ t>0, \ j\in \mathbb{N}\cup\{0\},
\end{split}
\end{align}
where $V(x, y, t)$ and $V_{*}(x, y)$ are defined by \eqref{DEF-V} and \eqref{DEF-V*}, respectively. On the other hand, the functional $\mathcal{M}\left[\p_{y}^{j}u_{0}\right](y)$ is defined by
\begin{equation}\label{DEF-mathMj}
\mathcal{M}\left[\p_{y}^{j}u_{0}\right](y):=\sqrt{2\pi }\left(\int_{\R}\mathcal{F}_{\xi}^{-1}\left[\mathcal{F}\left[\p_{y}^{j}u_{0}\right]\left(\xi, \frac{y}{2\xi }\right)\right](x)dx\right), \ \ \p_{y}^{j}u_{0}\in L^{1}(\R^{2}). 
\end{equation}
If $j=0$, this function $\psi_{0}(x, y, t)$ is equal to $\psi(x, y, t)$ defined by \eqref{DEF-psi} in the introduction, 
\begin{equation}\label{psi=psi0}
\mathrm{i.e.} \ \ \psi_{0}(x, y, t)= \psi(x, y, t), \ \ (x, y)\in \R^{2}, \ t>0. 
\end{equation}
In addition, we remark that $\p_{x}^{l}\psi_{j}(x, y, t)$ satisfy the following fact: 
\[
\left\| \p_{x}^{l}\psi_{j}(\cdot, \cdot, t)\right\|_{L^{\infty}}=\left\|\mathcal{M}\left[\p_{y}^{j}u_{0}\right](\cdot)\p_{x}^{l}V_{*}(\cdot, \cdot)\right\|_{L^{\infty}}t^{-\frac{3}{4}-\frac{l}{2}}, \ \ t>0, \ l, j \in \mathbb{N}\cup \{0\}. 
\]
In what follows, we would like to show that $\p_{x}^{l}\p_{y}^{j}v(x, y, t)$ tends to $\p_{x}^{l}\psi_{j}(x, y, t)$ as $t\to \infty$. 
Actually, combining the techniques used for the parabolic equations and for the Schr$\ddot{\mathrm{o}}$dinger equation, the following asymptotic formula can be established: 
%%%%%%%%%%%%%%%%%%%%%%%%%%%%%%%%%%%%%%%%%%%%%%%%%%%%
\begin{thm}\label{thm.linear-ap-v}
Let $l$ be a non-negative integer and $\alpha>1/2$. Suppose that there exists a non-negative integer j such that $\p_{y}^{j}u_{0}\in L^{1}(\R^{2})$. 
If $y^{2}\p_{y}^{j}u_{0}\in L^{1}(\R^{2})$ when $l$ is a positive integer and $y^{2}D_{x}^{-\alpha}\p_{y}^{j}u_{0}\in L^{1}(\R^{2})$ when $l=0$, 
then we have 
\begin{equation}\label{linear-ap-v}
\lim_{t \to \infty} t^{ \frac{3}{4}+\frac{l}{2} } \left\|\p_{x}^{l}\p_{y}^{j}v(\cdot, \cdot, t)-\p_{x}^{l}\psi_{j}(\cdot, \cdot, t)  \right\|_{L^{\infty}}=0, 
\end{equation}
where $v(x, y, t)$ is the solution to \eqref{approx-CP}, while $\psi_{j}(x, y, t)$ is defined by \eqref{DEF-psij}. 
\end{thm}
%%%%%%%%%%%%%%%%%%%%%%%%%%%%%%%%%%%%%%%%%%%%%%%%%%%%
%%%%%%%%%%%%%%%%%%%%%%%%%%%%%%%%%%%%%%%%%%%%%%%%%%%%
\begin{proof}
First, we set
\begin{equation}\label{DEF-phi}
\phi(\xi, y, t):=\mathcal{F}_{x}[v](\xi, y, t)=\frac{1}{\sqrt{2\pi}}\int_{\R}e^{-ix\xi }v(x, y, t)dx. 
\end{equation}
Then, since $v(x, y, t)$ is the solution to \eqref{approx-CP}, for fixed $\xi \in \R$, we have 
\begin{align*}
& \phi_{t}+(i\xi) \phi_{yy}+\mu \xi^{2}\phi=0, \ \ y\in \R, \ t>0, \\
& \phi(\xi, y, 0) = \mathcal{F}_{x}[u_{0}](\xi, y), \ \ y \in \R. 
\end{align*}
Also, if we set $\chi(\xi, y, t)=e^{\mu t\xi^{2}}\phi(\xi, y, t)$, it follows that 
\begin{align*}
& \chi_{t}+(i\xi)\chi_{yy}=0, \ \ y\in \R, \ t>0, \\
& \chi(\xi, y, 0) = \mathcal{F}_{x}[u_{0}](\xi, y), \ \ y \in \R. 
\end{align*}
Since this is the free Schr$\ddot{\mathrm{o}}$dinger equation, we can see that 
\[
\chi(\xi, y, t)=\frac{1}{\sqrt{4\pi i\xi t}}\int_{\R}e^{-\frac{(y-w)^{2}}{4i \xi t}}\mathcal{F}_{x}[u_{0}](\xi, w)dw, \ \ \xi \neq0, \ y\in \R, \ t>0. 
\]
Therefore, noticing \eqref{DEF-phi}, $\phi(\xi, y, t)$ can be expressed as follows: 
\begin{align}
\p_{y}^{j}\phi(\xi, y, t)
&= \frac{e^{-\mu t\xi^{2}}}{\sqrt{4\pi i\xi t}} \int_{\R}e^{-\frac{(y-w)^{2}}{4i \xi t}}\mathcal{F}_{x}\left[\p_{y}^{j}u_{0}\right](\xi, w)dw \nonumber \\
&=\frac{e^{-\mu t\xi^{2}-\frac{i\pi}{4}\mathrm{sgn} \xi}}{\sqrt{4\pi t|\xi|}}\int_{\R}e^{-\frac{(y-w)^{2}}{4i \xi t}}\mathcal{F}_{x}\left[\p_{y}^{j}u_{0}\right](\xi, w)dw, \ \ \xi \neq0, \ y\in \R, \ t>0. \label{phi-sol}
\end{align} 

In what follows, we shall derive the leading term of $v(x, y, t)$. 
First, let us decompose the integrand in \eqref{phi-sol}. 
From the mean value theorem, there exists $\theta=\theta(y, w, \xi, t) \in (0, 1)$ such that 
\begin{align}
&\int_{\R}e^{-\frac{(y-w)^{2}}{4i \xi t}}\mathcal{F}_{x}\left[\p_{y}^{j}u_{0}\right](\xi, w)dw 
%&=e^{-\frac{y^{2}}{4i \xi t}}\int_{\R}e^{\frac{yw}{2i \xi t}-\frac{w^{2}}{4i \xi t}}\mathcal{F}_{x}\left[\p_{y}^{j}u_{0}\right](\xi, w)dw \nonumber \\
=e^{\frac{iy^{2}}{4 \xi t}}\int_{\R}e^{-\frac{iyw}{2\xi t}}\left(1+\frac{iw^{2}}{4 \xi t}e^{\frac{i\theta w^{2}}{4 \xi t}}\right)\mathcal{F}_{x}\left[\p
_{y}^{j}u_{0}\right](\xi, w)dw \nonumber \\
&=\sqrt{2\pi}e^{\frac{iy^{2}}{4 \xi t}}\mathcal{F}\left[\p_{y}^{j}u_{0}\right]\left(\xi, \frac{y}{2\xi t}\right)
+\frac{ie^{\frac{iy^{2}}{4 \xi t}}}{4 \xi t}\int_{\R}e^{-\frac{iyw}{2\xi t}+\frac{i\theta w^{2}}{4 \xi t}}w^{2}\mathcal{F}_{x}\left[\p_{y}^{j}u_{0}\right](\xi, w)dw. \label{integrand}
\end{align}
Combining \eqref{phi-sol} and \eqref{integrand}, $\p_{y}^{j}\phi(\xi, y, t)$ can be rewritten by 
\begin{align}
\p_{y}^{j}\phi(\xi, y, t)
&=\frac{e^{-\mu t\xi^{2}+\frac{iy^{2}}{4\xi t}-\frac{i\pi}{4}\mathrm{sgn} \xi}}{\sqrt{2t|\xi|}}\mathcal{F}\left[\p_{y}^{j}u_{0}\right]\left(\xi, \frac{y}{2\xi t}\right) \nonumber\\
&\ \ \ +\frac{ie^{-\mu t\xi^{2}+\frac{iy^{2}}{4\xi t}-\frac{i\pi}{4}\mathrm{sgn} \xi}}{8\sqrt{\pi}t^{\frac{3}{2}}|\xi|^{\frac{1}{2}}\xi}\int_{\R}e^{-\frac{iyw}{2\xi t}+\frac{i\theta w^{2}}{4 \xi t}}w^{2}\mathcal{F}_{x}\left[\p_{y}^{j}u_{0}\right](\xi, w)dw \nonumber \\
&=:\mathcal{F}_{x}[W_{j}](\xi, y, t)+\mathcal{F}_{x}[R_{j}](\xi, y, t). \label{phi-sol-re}
\end{align}
Therefore, it follows from \eqref{DEF-phi} and \eqref{phi-sol-re} that 
\begin{equation}\label{v-sol-rewrite}
\p_{x}^{l}\p_{y}^{j}v(x, y, t)=\p_{x}^{l}W_{j}(x, y, t)+\p_{x}^{l}R_{j}(x, y, t). 
\end{equation}
Now, recalling \eqref{re-V} and using the change of variables as $\xi \mapsto -\xi$ , we have 
\begin{align*}
V(-x, y, t)
&=\frac{t^{ -\frac{1}{2} }}{ 4\pi^{ \frac{3}{2} }} \int_{\R} |\xi|^{ -\frac{1}{2} } e^{ -\mu t \xi^{2} -\frac{iy^{2}}{4t\xi} + \frac{i\pi}{4}\mathrm{sgn} \xi -ix\xi} d\xi  
%&=\frac{1}{2\pi}\cdot \frac{1}{\sqrt{2\pi}}\int_{\R}\frac{e^{-\mu t\xi^{2}+\frac{iy^{2}}{4\xi t}-\frac{i\pi}{4}\mathrm{sgn} \xi}}{\sqrt{2t|\xi|}}e^{ix\xi }d\xi
=\frac{1}{2\pi}\mathcal{F}_{\xi}^{-1}\left[ \frac{e^{-\mu t\xi^{2}+\frac{iy^{2}}{4\xi t}-\frac{i\pi}{4}\mathrm{sgn} \xi}}{\sqrt{2t|\xi|}}  \right](x). 
\end{align*}
Therefore, we can see that 
\begin{equation}\label{V-fourier}
2\pi \mathcal{F}_{x}\left[V(-x, y, t)\right](\xi)=\frac{e^{-\mu t\xi^{2}+\frac{iy^{2}}{4\xi t}-\frac{i\pi}{4}\mathrm{sgn} \xi}}{\sqrt{2t|\xi|}}. 
\end{equation}
Thus, it follows from \eqref{phi-sol-re} and \eqref{V-fourier} that 
\begin{align}
W_{j}(x, y, t)
&=\mathcal{F}_{\xi}^{-1}\left[
2\pi \mathcal{F}_{x}\left[V(-x, y, t)\right](\xi) 
\mathcal{F}\left[\p_{y}^{j}u_{0}\right]\left(\xi, \frac{y}{2\xi t}\right) 
\right](x) \nonumber \\
&=\sqrt{2\pi }\left(V(-\cdot, y, t)*\mathcal{F}_{\xi}^{-1}\left[\mathcal{F}\left[\p_{y}^{j}u_{0}\right]\left(\xi, \frac{y}{2\xi t}\right)\right](\cdot)\right)(x). \label{re-W} 
\end{align}
Hence, it follows from \eqref{re-W}, \eqref{DEF-psij} and \eqref{DEF-mathMj} that 
\begin{align*}
&\p_{x}^{l}\left(W_{j}(x, y, t)-\psi_{j}(x, y, t)\right) \\
&=\sqrt{2\pi }\p_{x}^{l}\biggl\{\left(V(-\cdot, y, t)*\mathcal{F}_{\xi}^{-1}\left[\mathcal{F}\left[\p_{y}^{j}u_{0}\right]\left(\xi, \frac{y}{2\xi t}\right)\right](\cdot)\right)(x) \\
&\ \ \ \ \ \ \ \ \ \ \ \ \ \ \ \ -\left(\int_{\R}\mathcal{F}_{\xi}^{-1}\left[\mathcal{F}\left[\p_{y}^{j}u_{0}\right]\left(\xi, \frac{y}{2\xi t}\right)\right](x)dx\right)V(-x, y, t)\biggl\}. 
\end{align*}
In the completely same way to prove Theorem~\ref{thm.L-ap-LZKB}, we immediately obtain 
\begin{equation}\label{v-ap-W-psi}
\lim_{t \to \infty} t^{ \frac{3}{4}+\frac{l}{2} } \left\|\p_{x}^{l}\left(W_{j}(\cdot, \cdot, t)-\psi_{j}(\cdot, \cdot, t)\right)  \right\|_{L^{\infty}} = 0. 
\end{equation}

On the other hand, from the definition of $R_{j}(x, y, t)$ in \eqref{phi-sol-re}, we get 
\begin{align}
\left|\p_{x}^{l}R_{j}(x, y, t)\right| 
%&=
%\frac{1}{\sqrt{2\pi}}
%\left|
%\int_{\R} 
%(i\xi)^{l}\left(
%\frac{ie^{-\mu t\xi^{2}+\frac{iy^{2}}{4\xi t}-\frac{i\pi}{4}\mathrm{sgn} \xi}}{8\sqrt{\pi}t^{\frac{3}{2}}|\xi|^{\frac{1}{2}}\xi}\int_{\R}e^{-\frac{iyw}{2\xi t}+\frac{i\theta w^{2}}{4 \xi t}}w^{2}\mathcal{F}_{x}\left[\p_{y}^{j}u_{0}\right](\xi, w)dw
%\right)
%e^{ix\xi}d\xi
%\right| \nonumber\\
&=
\frac{1}{\sqrt{2\pi}}
\left|
\int_{\R} 
w^{2}\left(\int_{\R}(i\xi)^{l}
\frac{ie^{-\mu t\xi^{2}+\frac{iy^{2}}{4\xi t}-\frac{i\pi}{4}\mathrm{sgn} \xi}}{8\sqrt{\pi}t^{\frac{3}{2}}|\xi|^{\frac{1}{2}}\xi}e^{-\frac{iyw}{2\xi t}+\frac{i\theta w^{2}}{4 \xi t}}\mathcal{F}_{x}\left[\p_{y}^{j}u_{0}\right](\xi, w)
e^{ix\xi}d\xi\right)dw
\right| \nonumber\\
&\le 
t^{-\frac{3}{2}}
\begin{cases}
\displaystyle \int_{\R}w^{2}\left\|\mathcal{F}_{x}\left[D_{x}^{-\alpha}\p_{y}^{j}u_{0}\right](\cdot, w)\right\|_{L^{\infty}_{\xi}}
\left(\int_{\R} |\xi|^{\alpha-\frac{3}{2}}e^{-\mu t\xi^{2}}d\xi\right)dw,& l=0, \\[1em]
\displaystyle \int_{\R}w^{2}\left\|\mathcal{F}_{x}\left[\p_{y}^{j}u_{0}\right](\cdot, w)\right\|_{L^{\infty}_{\xi}}
\left(\int_{\R} |\xi|^{-\frac{3}{2}+l}e^{-\mu t\xi^{2}}d\xi\right)dw, & l\ge1
\end{cases} \nonumber \\
%&\le C
%\begin{cases}
%\displaystyle t^{-\frac{3}{2}-\frac{1}{2}-\frac{1}{2}\left(\alpha-\frac{3}{2}\right)}\int_{\R}
%w^{2}\left\|D_{x}^{-\alpha}\p_{y}^{j}u_{0}(\cdot, w)\right\|_{L^{1}_{x}}
%\left(\int_{0}^{\infty} r^{\frac{2\alpha-1}{4}-1}e^{-r}dr\right)dw,& l=0,\\[1em]
%\displaystyle t^{-\frac{3}{2}-\frac{1}{2}-\frac{1}{2}\left(-\frac{3}{2}+l\right)}\int_{\R}
%w^{2}\left\|\p_{y}^{j}u_{0}(\cdot, w)\right\|_{L^{1}_{x}}
%\left(\int_{0}^{\infty} r^{\frac{-1+2l}{4}-1}e^{-r}dr\right)dw,& l\ge1
%\end{cases}\nonumber \\
&\le C
\begin{cases}
\displaystyle \Gamma\left(\frac{2\alpha-1}{4}\right)\left\|y^{2}D_{x}^{-\alpha}\p_{y}^{j}u_{0}\right\|_{L^{1}}t^{-\frac{5}{4}-\frac{\alpha}{2}}, & l=0,  \\[1em]
\displaystyle \Gamma\left(\frac{-1+2l}{4}\right)\left\|y^{2}\p_{y}^{j}u_{0}\right\|_{L^{1}}t^{-\frac{5}{4}-\frac{l}{2}}, & l\ge1
\end{cases}\nonumber \\
&\le Ct^{-\frac{5}{4}-\frac{l}{2}}, \ \ (x, y)\in \R^{2}, \ t\ge1. \label{R-estimate}
\end{align}

Finally, let us prove the approximation formula \eqref{linear-ap-v}. From \eqref{v-sol-rewrite}, we have 
\begin{align*}
\p_{x}^{l}\p_{y}^{j}v(x, y, t)-\p_{x}^{l}\psi_{j}(x, y, t)
=\p_{x}^{l}\left(W_{j}(x, y, t)-\psi_{j}(x, y, t)\right)+\p_{x}^{l}R_{j}(x, y, t).  
\end{align*}
Therefore, combining \eqref{v-ap-W-psi} and \eqref{R-estimate}, we obtain 
\begin{align*}
&\limsup_{t \to \infty} t^{ \frac{3}{4}+\frac{l}{2} } \left\|\p_{x}^{l}\p_{y}^{j}v(\cdot, \cdot, t)-\p_{x}^{l}\psi_{j}(\cdot, \cdot, t) \right\|_{L^{\infty}} \\
&\le \lim_{t \to \infty} t^{ \frac{3}{4}+\frac{l}{2} } \left\|\p_{x}^{l}\left(W_{j}(\cdot, \cdot, t)-\psi_{j}(\cdot, \cdot, t)\right)  \right\|_{L^{\infty}} +\left(\limsup_{t\to \infty} t^{ \frac{3}{4}+\frac{l}{2} }\left\|\p_{x}^{l}R_{j}(\cdot, \cdot, t)\right\|_{L^{\infty}}\right) \left(\lim_{t\to \infty}t^{-\frac{1}{2}}\right)=0. 
\end{align*}
This completes the proof of the desired result \eqref{linear-ap-v}. 
\end{proof}
%%%%%%%%%%%%%%%%%%%%%%%%%%%%%%%%%%%%%%%%%%%%%%%%%%%%

%%%%%%%%%%%%%%%%%%%%%%%%%%%%%%%%%%%%%%%%%%%%%%%%%%%%
\begin{cor}\label{cor.linear-ap-v}
Let $l$ be a non-negative integer and $\alpha>1/2$. Suppose that there exists a non-negative integer j such that $\p_{y}^{j}u_{0}\in L^{1}(\R^{2})$. 
If $y^{2}\p_{y}^{j}u_{0}\in L^{1}(\R^{2})$ when $l$ is a positive integer and $y^{2}D_{x}^{-\alpha}\p_{y}^{j}u_{0}\in L^{1}(\R^{2})$ when $l=0$, 
then we have 
\begin{equation}\label{linear-ap}
\lim_{t \to \infty} t^{ \frac{3}{4}+\frac{l}{2} } \left\|\p_{x}^{l}\p_{y}^{j}\tilde{u}(\cdot, \cdot, t)-\p_{x}^{l}\psi_{j}(\cdot, \cdot, t)  \right\|_{L^{\infty}}=0, 
\end{equation}
where $\tilde{u}(x, y, t)$ is the solution to \eqref{LZKB}, while $\psi_{j}(x, y, t)$ is defined by \eqref{DEF-psi}. 
\end{cor}
%%%%%%%%%%%%%%%%%%%%%%%%%%%%%%%%%%%%%%%%%%%%%%%%%%%%
%%%%%%%%%%%%%%%%%%%%%%%%%%%%%%%%%%%%%%%%%%%%%%%%%%%%
\begin{proof}
It follows from \eqref{LZKB-sol} and \eqref{approx-sol} that 
\begin{align*}
\p_{x}^{l}\p_{y}^{j}\tilde{u}(x, y, t)-\p_{x}^{l}\psi_{j}(x, y, t) 
%&=\p_{x}^{l}\p_{y}^{j}\tilde{u}(x, y, t)-\p_{x}^{l}\p_{y}^{j}v(x, y, t)+\p_{x}^{l}\p_{y}^{j}v(x, y, t)-\p_{x}^{l}\psi_{j}(x, y, t) \\
=\p_{x}^{l}\p_{y}^{j}\left((U-V)(t)*u_{0}\right)(x, y)+\p_{x}^{l}\p_{y}^{j}v(x, y, t)-\p_{x}^{l}\psi_{j}(x, y, t). 
\end{align*}
Therefore, applying Corollary~\ref{cor.L-decay-LZKB-ap} and Theorem~\ref{thm.linear-ap-v}, we obtain 
\begin{align*}
&\limsup_{t \to \infty} t^{ \frac{3}{4}+\frac{l}{2} } \left\|\p_{x}^{l}\p_{y}^{j}\tilde{u}(\cdot, \cdot, t)-\p_{x}^{l}\psi_{j}(\cdot, \cdot, t)  \right\|_{L^{\infty}} \\
&\le \left(\limsup_{t \to \infty} t^{ \frac{5}{4}+\frac{l}{2} } \left\|\p_{x}^{l}\p_{y}^{j}\left((U-V)(t)*u_{0}\right)(\cdot, \cdot)\right\|_{L^{\infty}}\right) \left(\lim_{t \to \infty} t^{ -\frac{1}{2}} \right)\\
&\ \ \ \ +\lim_{t \to \infty} t^{ \frac{3}{4}+\frac{l}{2} } \left\|\p_{x}^{l}\p_{y}^{j}v(\cdot, \cdot, t)-\p_{x}^{l}\psi_{j}(\cdot, \cdot, t)  \right\|_{L^{\infty}}=0. 
\end{align*}
This completes the proof of the desired result \eqref{linear-ap}. 
\end{proof}
%%%%%%%%%%%%%%%%%%%%%%%%%%%%%%%%%%%%%%%%%%%%%%%%%%%%

Finally, we can prove the asymptotic formula \eqref{u-asymptotic}, i.e. Theorem~\ref{thm.main-u-asymptotic} is true: 
%%%%%%%%%%%%%%%%%%%%%%%%%%%%%%%%%%%%%%%%%%%%%%%%%%%%
\begin{proof}[\rm{\bf{End of the proof of Theorem~\ref{thm.main-u-asymptotic}}}]
It directly follows from Theorems~\ref{thm.main-approximation}, \ref{thm.linear-ap-v} and \eqref{psi=psi0} that 
\begin{align*}
&\lim_{t\to \infty}t^{\frac{3}{4}+\frac{l}{2}}\left\|\p_{x}^{l}\left(u(\cdot, \cdot, t)-\psi(\cdot, \cdot, t)\right)\right\|_{L^{\infty}} \\
&\le \lim_{t\to \infty}t^{\frac{3}{4}+\frac{l}{2}}\left\|\p_{x}^{l}\left(u(\cdot, \cdot, t)-v(\cdot, \cdot, t)\right)\right\|_{L^{\infty}}
+\lim_{t\to \infty}t^{\frac{3}{4}+\frac{l}{2}}\left\|\p_{x}^{l}\left(v(\cdot, \cdot, t)-\psi(\cdot, \cdot, t)\right)\right\|_{L^{\infty}} 
=0. 
\end{align*}
Thus, the desired asymptotic formula \eqref{u-asymptotic} has been established. 
\end{proof}
%%%%%%%%%%%%%%%%%%%%%%%%%%%%%%%%%%%%%%%%%%%%%%%%%%%%

%%%%%%%%%%%%%%%%%%%%%%%%%%%%%%%%%%%%%%%%%%%%%%%%%%%%
\section*{Acknowledgments}
%%%%%%%%%%%%%%%%%%%%%%%%%%%%%%%%%%%%%%%%%%%%%%%%%%%%
The first author is supported by Grant-in-Aid for Young Scientists Research No.22K13939, Japan Society for the Promotion of Science. 
The second author is supported by Grant-in-Aid for Young Scientists Research~(B) No.17K14220, Japan Society for the Promotion of Science. 
The authors also would like to thank the anonymous referees for their helpful and valuable comments on this paper.
%%%%%%%%%%%%%%%%%%%%%%%%%%%%%%%%%%%%%%%%%%%%%%%%%%%%

%%%%%%%%%%%%%%%%%%%%%%%%%%%%%%%%%%%%%%%%%%%%%%%%%%%%

%%%%%%%%%%%%%%%%%%%%%%%%%%%%%%%%%%%%%%%%%%%%%%%%%%%%

%%%%%%%%%%%%%%%%%%%%%%%%%%%%%%%%%%%%%%%%%%%%%%%%%%%%
\bigskip
\par\noindent
\begin{flushleft}Ikki Fukuda\\
Division of Mathematics and Physics, \\
Faculty of Engineering, \\
Shinshu University, \\
4-17-1, Wakasato, Nagano, 380-8553, JAPAN\\
E-mail: i\_fukuda@shinshu-u.ac.jp
\vskip15pt
Hiroyuki Hirayama\\
Faculty of Education, \\
University of Miyazaki, \\
1-1, Gakuenkibanadai-nishi, Miyazaki, 889-2192, JAPAN\\
E-mail: h.hirayama@cc.miyazaki-u.ac.jp
\end{flushleft}
%%%%%%%%%%%%%%%%%%%%%%%%%%%%%%%%%%%%%%%%%%%%%%%%%%%%

%%%%%%%%%%%%%%%%%%%%%%%%%%%%%%%%%%%%%%%%%%%%%%%%%%%%
\end{document}